\documentclass[leqno,11pt]{amsart}
\usepackage{amssymb, amsmath}
\usepackage{amssymb,amsfonts}
\usepackage{amsmath,latexsym}
\usepackage{pdfsync}
\usepackage{bbm}

\usepackage{color}
\newcommand{\andf}{\quad\hbox{and}\quad}
\newcommand{\with}{\quad\hbox{with}\quad}
\def\Supp{\mathop{\rm Supp}\nolimits\ }
\newcommand{\newcom}{\newcommand}

\def\inte#1{
\displaystyle\mathop{#1\kern0pt}^\circ }
\newcommand{\w}[1]{\langle {#1} \rangle}
\newcom{\al}{\alpha}
\newcom{\de}{\delta}
\newcom{\Th}{\Theta}
\newcom{\be}{\beta}
\newcom{\s}{\sigma}
\newcom{\eps}{\epsilon}
\newcom{\ve}{\varepsilon}
\newcom{\ga}{\gamma}
\newcom{\Ga}{\Gamma}
\newcom{\ka}{\kappa}
\newcom{\Lam}{\Lambda}
\newcom{\lam}{\lambda}
\newcom{\vp}{\varphi}
\newcom{\om}{\omega}
\newcom{\Sig}{\Sigma}
\newcom{\sig}{\sigma}
\newcom{\tht}{\theta}
\newcom{\tri}{\triangle}
\newcom{\oo}{\infty}
\newcom{\h}{{\rm h}}
\newcom{\rmv}{{\rm v}}
\newcom{\hs}{\hslash}
\newcom{\vphi}{\varphi}
\newcom{\cB}{{\mathcal B}}
\newcom{\cC}{{\mathcal C}}
\newcom{\cD}{{\mathcal D}}
\newcom{\cF}{{\mathcal F}}
\newcom{\cL}{{\mathcal L}}
\newcom{\cM}{{\mathcal M}}
\newcom{\cP}{{\mathcal P}}
\newcom{\cS}{{\mathcal S}}
\newcom{\cQ}{{\mathcal Q}}
\newcom{\cT}{{\mathcal T}}
\newcom{\cY}{{\mathcal Y}}
\newcom{\cZ}{{\mathcal Z}}
\newcom{\R}{\Bbb R}
\newcom{\T}{\Bbb T}
\newcom{\N}{\Bbb N}
\newcom{\Z}{\Bbb Z}
\newcom{\C}{\Bbb C}
\newcom{\E}{\Bbb E}
\let\wh=\widehat

\let\e=\varepsilon
\def\Ap{A_\Phi}
\def\fc{\frak{c}}
\def\Bp{B_\Phi}
\def\up{u_{\Phi}}
\def\bp{b_\Phi}
\def\upk{u_{\Phi_\kappa}}
\def\bpk{b_{\Phi_\kappa}}
\def\ubp{(u,b)_\Phi}
\def\ubpk{(u,b)_{\Phi_\kappa}}
\def\gp{G_\Phi}
\def\hp{H_\Phi}
\def\ghp{(G,H)_\Phi}
\def\gpk{G_{\Phi_\kappa}}
\def\hpk{H_{\Phi_\kappa}}
\def\ghpk{(G,H)_{\Phi_\kappa}}
\def\ghpk{(G,H)_{\Phi_\kappa}}
\def\vp{\vphi_\Phi}
\def\pp{\psi_\Phi}
\def\vpk{\vphi_{\Phi_\kappa}}
\def\ppk{\psi_{\Phi_\kappa}}
\def\ppp{(\vphi,\psi)_\Phi}
\def\pppk{(\vphi,\psi)_{\Phi_\kappa}}

\def\UBp{(U,B)_\Phi}
\def\mp{(m_U,m_B)_\Phi}
\def\UBpk{(U,B)_{\Phi_\kappa}}
\def\mpk{(m_U,m_B)_{\Phi_\kappa}}
\def\Mp{(M_U,M_B)_\Phi}
\def\Mpk{(M_U,M_B)_{\Phi_\kappa}}
\def\mMp{(m_U,m_B,M_B,M_U)_\Phi}
\def\mMpk{(m_U,m_B,M_B,M_U)_{\Phi_\kappa}}
\def\ep{e^\Psi}
\def\epp{e^{2\Psi}}
\def\egp{e^{\ga\Psi}}
\def\epk{e^{\Psi_\ka}}
\def\eppk{e^{2\Psi_\ka}}
\def\egpk{e^{\ga\Psi_\ka}}
\def\eap{e^{a\Psi}}
\def\ebp{e^{b\Psi}}
\def\ecp{e^{c\Psi}}
\def\edp{e^{d\Psi}}
\def\tt{\langle t\rangle}
\def\ttau{\langle \tau \rangle}
\def\fyt{\f{y}{2\tt}}
\def\fykt{\f{y}{2\ka\tt}}
\def\lk{l_\ka}

\def\tauk{\ttau^{\lk-\frak{c}\e}}
\def\taulk{\ttau^{\f12+\lk-\frak{c}\e}}

\def\taullk{\ttau^{1+\lk-\frak{c}\e}}
\def\elk{\ell_\ka}

\def\etauk{\ttau^{\elk-\frak{c}\e}}
\def\etaulk{\ttau^{\f12+\elk-\frak{c}\e}}

\def\tauek{\ttau^{\elk-\frak{c}\e}}
\def\tauelk{\ttau^{\f12+\elk-\frak{c}\e}}
\def\tauelk{\ttau^{\f12+\elk-\frak{c}\e}}
\def\tauellk{\ttau^{1+\elk-\frak{c}\e}}

\def\dtht{\dot{\tht}}
\def\thtk{\tht_\ka}
\def\dthtk{\dot{\thtk}}
\def\Dhk{\Delta^{\h}_k}
\def\Dh{\Delta^{\h}}
\def\Sh{S^\h}
\def\wtDh{\widetilde{\Delta}^{\h}}
\def\hsf{\hslash^\f12}
\def\qhs{\sqrt{\hslash'}}

\newcom{\f}{\frac}
\newcom{\dint}{\displaystyle\int}
\newcom{\dsum}{\displaystyle\sum}
\newcom{\dlim}{\displaystyle\lim}
\newcom{\ov}{\overline}
\newcom{\wt}{\widetilde}
\newcom{\pa}{\partial}
\newcom{\p}{\partial}
\newcom\na{\nabla}
\newcom{\D}{\Delta}
\newcom\rto{\rightarrow}
\newcom\lto{\leftarrow}
\newcom\mto{\mapsto}
\newcom{\disp}{\displaystyle}
\newcom{\non}{\nonumber}
\newcom{\no}{\noindent}
\newcom{\QED}{$\square$}

\def\eqdefa{\buildrel\hbox{\footnotesize def}\over =}

\newcommand{\beq}{\begin{equation}}
\newcommand{\eeq}{\end{equation}}
\newcommand{\beqo}{\begin{equation*}}
\newcommand{\eeqo}{\end{equation*}}
\newcommand{\ben}{\begin{eqnarray}}
\newcommand{\een}{\end{eqnarray}}
\newcommand{\beno}{\begin{eqnarray*}}
\newcommand{\eeno}{\end{eqnarray*}}

\newtheorem{Def}{Definition}[section]
\newtheorem{thm}{Theorem}[section]
\newtheorem{lem}{Lemma}[section]
\newtheorem{rmk}{Remark}[section]
\newtheorem{col}{Corollary}[section]
\newtheorem{prop}{Proposition}[section]
\renewcommand{\theequation}{\thesection.\arabic{equation}}

\newtheorem{theorem}{Theorem}[section]

\newtheorem{lemma}[theorem]{Lemma}

\begin{document}
\title[Global solutions of MHD boundary layer equations]
{Global small analytic solutions of MHD boundary layer equations}

\author[N. Liu]{Ning Liu}
\address[N. Liu]
 {Academy of Mathematics $\&$ Systems Science, The Chinese Academy of
Sciences, Beijing 100190, CHINA. } \email{liuning16@mails.ucas.ac.cn}
\author[P. Zhang]{Ping Zhang}%
\address[P. Zhang]
 {Academy of Mathematics $\&$ Systems Science
and  Hua Loo-Keng Key Laboratory of Mathematics, The Chinese Academy of
Sciences, Beijing 100190, CHINA, and School of Mathematical Sciences, University of Chinese Academy of Sciences, Beijing 100049, CHINA. }
\email{zp@amss.ac.cn}

\date{\today}

\begin{abstract}
In this paper, we prove the global existence and the large time decay estimate of  solutions
to the two-dimensional MHD boundary layer equations with small initial data, which is analytical in the tangential variable.
The main idea of the proof is motivated by that of \cite{PZ5}. The additional difficulties are: 1. there appears
the magnetic field; 2. the far field here depends on the tangential variable; 3. the
Reynolds number is different from magnetic Reynolds number.
In particular, we solved an open question in \cite{XY19} concerning the large time existence of
analytical solutions to the MHD boundary layer equations.
\end{abstract}

\maketitle

\vspace{0.1cm}


\renewcommand{\theequation}{\thesection.\arabic{equation}}
\setcounter{equation}{0}

\section{Introduction}\label{sect1}

In this paper, we consider the global well-posedness of the following two-dimensional MHD boundary layer equations in the upper space $\R^2_+\eqdefa \bigl\{(x,y): x\in\R,y\in \R_+  \bigr\},$
\begin{equation}\label{MHDbl}
 \quad\left\{\begin{array}{l}
\p_t u_1-\p^2_yu_1+ u_1\p_x u_1+u_2\p_y u_1+\p_x p=b_1\p_x b_1+b_2\p_y b_1,\\
\p_t b_1-\ka\p^2_yb_1+u_1\p_x b_1+u_2\p_y b_1=b_1\p_x u_1+b_2\p_y u_1,\\
\p_x u_1+\p_y u_2=0, \ \ \p_xb_1+\p_yb_2=0,\\
u_1|_{y=0}=u_2|_{y=0}=0, \ \ \p_yb_1|_{y=0}=b_2|_{y=0}=0,\\
\lim_{y\rto+\infty} u_1=U_1,\ \ \lim_{y\rto+\infty} b_1=B_1,\\
u_1|_{t=0}=u_{1,0},\ \ b_1|_{t=0}=b_{1,0},
\end{array}\right.
\end{equation}
where $(u_1,u_2)$ and $(b_1,b_2)$ represent the velocity of  fluid and the magnetic field respectively, $\ka>0$ is a constant which represents the difference between the Reynolds number and the magnetic Reynolds number, $(U_1,B_1,p)(t,x)$ are the traces of the tangential fields and pressure of the outflow on the boundary, which satisfies Bernoulli's law:
\begin{equation}\label{B'slaw}
 \quad\left\{\begin{array}{l}
\p_t U_1+ U_1\p_x U_1+\p_x p=B_1\p_x B_1,\\
\p_t B_1+U_1\p_x B_1=B_1\p_x U_1.
\end{array}\right.
\end{equation}

The MHD boundary layer system \eqref{MHDbl} was derived in \cite{GP17, LXY19cpam, LXY19} by considering the high Reynolds number limit
to the incompressible viscous MHD system (see \cite{Cow57, David01}) near a non-slip boundary when both the Reynolds number
and the magnetic Reynolds
number have the same order.  In particular, when $b_1$ equals some constant in \eqref{MHDbl},
the system reduces to the classical Prandtl equations  (simplified as $(PE)$ in the sequel)
which was proposed by Prandtl \cite{Pra} in 1904 in order to explain the disparity between the boundary conditions verified by  ideal fluid and viscous fluid with small viscosity.
  One may check \cite{EE, Olei} and
references therein for more introductions on boundary layer theory in the absence of the magnetic field.
 Especially we refer to \cite{Guo} for a comprehensive  recent survey.
\smallskip

Under the monotonicity condition on the tangential velocity field in the normal direction to the boundary, Ole$\breve{i}$nik \cite{Olei63}
established the local well-posedness of $(PE)$ with initial data in Sobolev space by using the Croco transformation. Lately, this result
was proved via energy method in \cite{Alex, MW} independently by taking care of the cancellation property in the convection terms of $(PE).$
Under a favorable condition on the pressure, a global-in-time weak solution was proved in \cite{Xin}.
 In general, when the monotonicity condition is violated, separation of the boundary layer is expected and observed in classical fluid.
 Especially when the background shear flow has a non-degenerate critical point, there are some interesting ill-posendness result to both
 linear and nonlinear Prandtl equations, one may check \cite{GD, Ger2, Guo, LY17} and the references therein for more details.
 However, surprisingly for the MHD boundary layer system \eqref{MHDbl}, Liu, Xie and Yang \cite{LXY19cpam} succeeded in proving  the local well-posedness in Sobolev space
 without any monotonicity assumption on the tangential velocity. The only essential assumption there is that the background tangential magnetic field
 has a positive lower bound. This result agrees with the general physical understanding that the magnetic field stabilizes the boundary layer.
 \smallskip

On the other hand,  for the data which is analytic in both $x$ and $y$ variables,
Sammartino and Caflisch \cite{Caf} established the local
well-posedness result of $(PE)$. The analyticity in $y$
variable was removed by Lombardo, Cannone and Sammartino in
\cite{Can}. The main argument used in  \cite{Can, Caf} is to apply
the abstract Cauchy-Kowalewskaya theorem. Recently, G\'ervard-Varet and Masmoudi \cite{GM} proved the
 well-posedness of $(PE)$ for a class of data with
Gevrey regularity. The optimal  result in this direction (with Gevrey class $2$) was obtained by  Dietert and G\'ervard-Varet  in \cite{DG08}.
 The question of the long time existence for Prandtl system with small analytic data was first addressed in \cite{ZZ} and an almost global existence result was provided in \cite{IV16}. Finally in \cite{PZ5}, Paicu and the second author proved the global well-posedness of $(PE)$ with small analytical data.\smallskip

 We mention that the main idea in \cite{PZ5} is to use the analytic energy
estimate. This idea dates back to
\cite{Ch04}, where Chemin introduced a tool to  make analytical type estimates to the Fujita-Kato  solutions of 3-D
Navier-Stokes system  and to
control  the size of the analytic radius simultaneously. This idea was used in the context of
anisotropic Navier-Stokes system \cite{CGP} (see also \cite{mz1, mz2}), which implies  the global
well-posedness of three dimensional Navier-Stokes system with a
class of ``ill prepared data", the $B^{-1}_{\infty,\infty}(\R^3)$
norm of which blow up as a small parameter goes to zero.
\smallskip

While for the  MHD boundary layer equation \eqref{MHDbl}, corresponding to the results in \cite{IV16,ZZ},
Xie and Yang \cite{XY19} obtained a lower bound for the lifespan of the analytic solutions.
And the authors left the following open question in \cite{XY19}: ``However, it is not known whether one can obtain
a global or almost global in time solution like the work on Prandtl system when the background shear velocity is taken to be
a Gaussian error function in \cite{IV16}." The goal of this paper is to solve this problem.

\smallskip

In order to do it, for any constant $\bar{B}_\ka,$ we take a cut-off function
$\chi\in C^\infty[0,\infty)$ with $\chi(y)=\left\{
\begin{array}{ll}
y\quad \mbox{if} \  y\geq 2,\\
0\quad  \mbox{if} \  y\leq 1,
\end{array}
\right.$ and make the following change of variables:
\beq \label{change}
\begin{split}
&u\eqdefa u_1-\chi'(y)U \andf v\eqdefa u_2+\chi(y)\p_xU,\\
&b\eqdefa b_1-\chi'(y)B-\bar{B}_\ka \andf h\eqdefa b_2+\chi(y)\p_xB,
\end{split}
\eeq
where $U\eqdefa U_1$ and $B\eqdefa B_1-\bar{B}_\ka$.

Then in view of \eqref{MHDbl} and \eqref{B'slaw}, $(u,v,b,h)$ solves
\begin{equation}\label{eqs1}
 \quad\left\{\begin{array}{l}
\p_t u-\p^2_yu-\bar{B}_\ka\p_xb+ u\p_x u-b\p_x b+v\p_yu-h\p_y b+\chi'(U\p_xu-B\p_xb)\\\qquad+\chi'(\p_xUu-\p_xBb)+\chi(-\p_xU\p_yu+\p_xB\p_yb)+\chi''(Uv-Bh)=m_U,\\
\p_t b-\ka\p^2_yb-\bar{B}_\ka\p_xu+u\p_x b-b\p_x u+v\p_y b-h\p_yu+\chi'(U\p_xb-B\p_xu)\\ \qquad+\chi'(\p_xBu-\p_xUb)+\chi(-\p_xU\p_yb+\p_xB\p_yu)+\chi''(Bv-Uh)=m_B,\\
\p_x u+\p_y v=0, \ \ \p_xb+\p_yh=0,\\
u|_{y=0}=v|_{y=0}=0, \ \ \p_yb|_{y=0}=h|_{y=0}=0,\\
\lim_{y\rto+\infty} u=0,\ \ \lim_{y\rto+\infty} b=0,\\
\lim_{y\rto+\infty} v=0,\ \ \lim_{y\rto+\infty} h=0,\\
u|_{t=0}=u_0\eqdefa u_{1,0}-\chi'U_{0},\ \ b|_{t=0}=b_0\eqdefa b_{1,0}-\chi'B_{0}-\bar{B}_\ka,
\end{array}\right.
\end{equation}
where $(U_0(x),B_0(x))\eqdefa (U(0,x),B(0,x)),$ and the terms $(m_U,m_B)$ are given by
\begin{equation}\label{def:m}
\quad\left\{\begin{array}{l}
m_U\eqdefa (1-\chi')(\p_tU-\bar{B}_\ka \p_xB)+\chi'''U+(1-{\chi'}^2+\chi\chi'')(U\p_xU-B\p_xB),\\
m_B\eqdefa (1-\chi')(\p_tB-\bar{B}_\ka \p_xU)+\chi'''B+(1-{\chi'}^2-\chi\chi'')(U\p_xB-B\p_xU).
\end{array}\right.
\end{equation}
It's easy to observe that $(m_U,m_B)$ are supported in $y\in[0,2]$ for each $t>0.$
\smallskip

Due to $\p_x u+\p_y v=0$ and $\p_xb+\p_yh=0$, there exist two potential functions $(\vphi,\psi)$ so that $(u,b)=\p_y (\vphi,\psi)$ and $(v,h)=-\p_x(\vphi,\psi)$.  While it follows from the boundary conditions in \eqref{eqs1} that
\beno
\p_x \dint_0^\infty (u,b)(t,x,y) dy= -\dint_0^\infty \p_y (v,h)(t,x,y)dy=0,
\eeno
which implies $\int_0^\infty (u,b)(t,x,y)dy=(C_1,C_2)(t)$. Yet Since $(u,b)$ decays to zero as $|x|$ tends to $\infty$, we have $C_1(t)=C_2(t)=0$.
So that we can impose  the following boundary conditions for the primitive functions $(\vphi,\psi)$:
\beno
(\vphi,\psi)|_{y=0}=0 \andf \dlim_{y\rto+\infty}(\vphi,\psi)=0.
\eeno
Then by integrating $(u,b)$ equations of \eqref{eqs1} with respect to y variable over $[y,\infty[$, we find
\begin{equation}\label{eqs2}
 \quad\left\{\begin{array}{l}
\p_t \vphi-\p^2_y\vphi -\bar{B}_\ka\p_x\psi+ u\p_x \vphi-b\p_x\psi+2\int_y^\infty (\p_x\vphi\p_yu-\p_x\psi\p_y b)dy'\\ \qquad+\chi'(U\p_x\vphi-B\p_x\psi)+2\int_y^\infty \chi''(U\p_x\vphi-B\p_x\psi)dy'+\chi(-\p_xUu+\p_xBb)\\ \qquad+2\chi'(\p_xU\vphi-\p_xB\psi)+2\int_y^\infty \chi''(\p_xU\vphi-\p_xB\psi)dy'=M_U,\\
\p_t\psi-\ka\p^2_y\psi-\bar{B}_\ka\p_x\vphi+u\p_x\psi-b\p_x\vphi+\chi'(U\p_x\psi-B\p_x\vphi)\\
\qquad+\chi(-\p_xUb+\p_xBu)=M_B,\\
\vphi|_{y=0}=\psi|_{y=0}=0\andf
\lim_{y\rto+\infty} \vphi=\lim_{y\rto+\infty} \psi=0,\\
\vphi|_{t=0}=\vphi_0\eqdefa -\int_y^\infty u_0\,dy',\ \   \psi|_{t=0}=\psi_0\eqdefa -\int_y^\infty b_0\,dy',
\end{array}\right.
\end{equation}
where $(M_U,M_B)\eqdefa -\int_y^\infty (m_U,m_B)\,dy'$ are also supported in $y\in[0,2]$ for each $t>0.$
\smallskip

Before proceeding, let us  recall from \cite{BCD} that
\beq\label{2.1a}
\begin{split}
&\Delta_k^{\rm h}a=\cF^{-1}(\varphi(2^{-k}|\xi|)\widehat{a}),\andf
S^{\rm h}_ka=\cF^{-1}(\chi(2^{-k}|\xi|)\widehat{a}),
\end{split} \eeq
where $\cF a$ and $\widehat{a}$  denote the partial  Fourier transform of
the distribution $a$ with respect to $x$ variable,  that is, $
\widehat{a}(\xi,y)=\cF_{x\to\xi}(a)(\xi,y),$
  and $\chi(\tau),$ ~$\varphi(\tau)$ are
smooth functions such that
 \beno
&&\Supp \varphi \subset \Bigl\{\tau \in \R\,/\  \ \frac34 \leq
|\tau| \leq \frac83 \Bigr\}\andf \  \ \forall
 \tau>0\,,\ \sum_{k\in\Z}\varphi(2^{-k}\tau)=1,\\
&&\Supp \chi \subset \Bigl\{\tau \in \R\,/\  \ \ |\tau|  \leq
\frac43 \Bigr\}\quad \ \ \ \andf \  \ \, \chi(\tau)+ \sum_{k\geq
0}\varphi(2^{-k}\tau)=1.
 \eeno

\begin{Def}\label{def2.1}
{\sl  Let~$s$ be in~$\R$. For~$u$ in~${S}_h'(\R^2_+),$ which
means that $u$ belongs to ~$S'(\R^2_+)$ and
satisfies~$\lim_{k\to-\infty}\|S_k^{\rm h}u\|_{L^\infty}=0,$ we set
$$
\|u\|_{\cB^{s,0}}\eqdefa\big\|\big(2^{ks}\|\Delta_k^{\rm h}
u\|_{L^2_+}\big)_{k\in\Z}\bigr\|_{\ell ^{1}(\Z)}.
$$
\begin{itemize}

\item
For $s\leq \frac{1}{2}$, we define $ \cB^{s,0}(\R^2_+)\eqdefa
\big\{u\in{S}_h'(\R^2_+)\;\big|\; \|u\|_{\cB^{s,0}}<\infty\big\}.$

\item
If $j$ is  a positive integer and if~$\frac{1}{2}+j< s\leq \frac{3}{2}+j$, then we define~$ \cB^{s,0}(\R^2_+)$  as the subset of distributions $u$ in~${S}_h'(\R^2_+)$ such that $\p_x^j u$ belongs to~$ \cB^{s-j,0}(\R^2_+).$
\end{itemize}

In all that follows, we always denote $\cB^s_\h$ to the Besov space $B^s_{2,1}(\R_\h).$
}
\end{Def}

In  order to obtain a better description of the regularizing effect
of the diffusion equation, we need to use Chemin-Lerner
type spaces $\widetilde{L}^{p}_T(\cB^{s,0}(\R^2_+))$ (see \cite{CN95}).
\begin{Def}\label{def2.2}
{\sl Let $p\in[1,\,+\infty]$ and $T_0,T\in[0,\,+\infty]$. We define
$\widetilde{L}^{p}(T_0,T; \cB^{s,0}(\R^2_+))$ as the completion of
$C([T_0,T]; \,\cS(\R^2_+))$ by the norm
$$
\|a\|_{\widetilde{L}^{p}(T_0,T; \cB^{s,0})} \eqdefa \sum_{k\in\Z}2^{ks}
\Big(\int_{T_0}^T\|\Delta_k^{\rm h}\,a(t) \|_{L^2_+}^{p}\,
dt\Big)^{\frac{1}{p}}
$$
with the usual change if $p=\infty.$ In particular, when $T_0=0,$ we simplify the notation $ \|a\|_{\widetilde{L}^{p}(0,T; \cB^{s,0})}$
 as $\|a\|_{\widetilde{L}_T^{p}(\cB^{s,0})}.$}
\end{Def}

In order to globally control the evolution of the analytic band to the solution of \eqref{eqs1},
motivated by \cite{IV16,PZ5}, we introduce the following key quantities:
\begin{equation}\label{defGH}
G\eqdefa u+\fyt\vphi \andf H\eqdefa b+\fykt\psi.
\end{equation}
And as in \cite{Ch04, CGP, CMSZ, PZ5, mz1, mz2, ZZ},  for any locally bounded function
$\Phi$ on $\R^+\times \R$, we define \beq\label{S3eq1}
\up(t,x,y)\eqdefa\cF_{\xi\to
x}^{-1}\bigl(e^{\Phi(t,\xi)}\widehat{u}(t,\xi,y)\bigr). \eeq We also
introduce a key quantity $\tht(t)$ to describe the evolution of
the analytic band of $u:$
\begin{equation}\label{def:theta}
 \quad\left\{\begin{array}{l}
\displaystyle \dtht(t)=\tt^\f14 \|\ep\p_y \ghp (t)\|_{\cB^{\f12,0}}+\e^{-\f12}\tt^\f54\|\UBp(t)\|_{\cB^\f12_\h},\\
\displaystyle \tht|_{t=0}=0.
\end{array}\right.
\end{equation}
Here  the phase function $\Phi$ is defined by
\beq\label{def:Phi} \Phi(t,\xi)\eqdefa (\de-\lam \tht(t))|\xi|, \eeq
and the weighted function $\Psi(t,y)$ is determined by \beq\label{def:Psi} \Psi(t,y)\eqdefa \f{y^2}{8\tt} \with \w{t}\eqdefa 1+t, \eeq
which satisfies \beq\label{S3eq5} \p_t \Psi+2(\p_y \Psi)^2=0. \eeq

\smallskip

Our first result of this paper states as follows:

\begin{thm}\label{thm1.1}
{\sl Let $\kappa\in ]0,2[,$ $\bar{B}_\ka=\left\{\begin{array}{l} 1\quad \mbox{if}\quad \ka=1,\\
0\quad\mbox{otherwise},\end{array}\right.$
and $\e, \de>0.$ We assume that the far field states $(U,B)$ satisfy
\begin{subequations} \label{UBdecay12}
\begin{gather}
\label{UBdecay}
\|\tt^\f94 e^{\de |D_x|}(U,B)\|_{\wt{L}^\oo(\R^+;\cB^\f32_\h)}+\|\tt^\f74 e^{\de |D_x|}(\p_tU,\p_tB,U,B)\|_{\wt{L}^2(\R^+;\cB^\f12_\h)} \leq \e,\\
\label{UBdecay2}
\int_0^\oo \tt^{\f54} \|e^{\de |D_x|}(U,B)\|_{\cB^\f12_\h}\,dt \leq {\e}.
\end{gather}
\end{subequations}
Let the initial data $(u_0,b_0,\vphi_0,\psi_0)$ (see the definitions in \eqref{eqs1} and \eqref{eqs2}) satisfy the compatibility condition: $u_0|_{y=0}=\p_y b_0|_{y=0}=0,$ $\int_0^\infty u_0\,dy=\int_0^\infty b_0\, dy=0,$ and
\beq
\|e^\f{y^2}8 e^{\de |D_x|}(u_0, b_0,\vphi_0,\psi_0) \|_{\cB^{\f12,0}}<\infty \andf \| e^\f{y^2}8 e^{\de |D_x|} (G_0, H_0)\|_{\cB^{\f12,0}} \leq \sqrt{\e},\label{smalldata2}
\eeq
where  $G_0\eqdefa u_0+\fyt\vphi_0\andf H_0\eqdefa b_0+\fykt\psi_0.$
Then there exist positive constants $\lam, \fc$ and $\e_0(\lam,\fc,\ka,\de)$  so that for $\e\leq\e_0$ 
and $\lk\eqdefa \f{\ka(2-\ka)}4\ \in]0, 1/4],$ the system \eqref{eqs1} has a unique global solution $(u,b)$ which satisfies $\sup_{t\in[0,\infty[}\tht(t)\leq \f\de{2\lam}$, and
\beq\label{S3eq6}
\| \ep \ubp\|_{\wt{L}^\infty_t (\cB^{\f12,0})}+\sqrt{l_\ka}\| \ep \p_y \ubp\|_{\wt{L}^2_t (\cB^{\f12,0})}\leq
 \|e^\f{y^2}8 e^{\de |D_x|}(u_0, b_0) \|_{\cB^{\f12,0}}+ C\sqrt{\e}.
\eeq
Furthermore, for any $t>0$ and $\ga\in]0,1[$, there hold
\begin{subequations} \label{S3eq789}
\begin{gather}
\label{S3eq7}
\|\taulk \ep\ubp\|_{\wt{L}^\infty_t (\cB^{\f12,0})}+\|\taulk \ep \p_y \ubp\|_{\wt{L}^2_{[\f{t}2,t]} (\cB^{\f12,0})}\\
\qquad\qquad\qquad\qquad\qquad\qquad\qquad\qquad\qquad\qquad \leq C\bigl( \|e^\f{y^2}8 e^{\de |D_x|}(u_0, b_0,\vphi_0,\psi_0) \|_{\cB^{\f12,0}}+\sqrt{\e}\bigr),\nonumber\\
\label{S3eq8}
\|\taullk \ep \ghp\|_{\wt{L}^\infty_t (\cB^{\f12,0})}+\|\taullk \ep \p_y\ghp\|_{\wt{L}^2_{[\f{t}2,t]} (\cB^{\f12,0})}\\
\qquad\qquad\qquad\qquad\qquad\qquad\qquad\qquad\qquad\qquad\qquad\leq C\sqrt{\e}\bigl(1+\|e^{\f{y^2}8}e^{\de|D_x|}(u_0,b_0)\|_{\cB^{\f12,0}}\bigr),\nonumber\\
\label{S3eq9}
\|\taullk \egp \ubp\|_{\wt{L}^\infty_t (\cB^{\f12,0})}+\|\taullk \egp \p_y \ubp\|_{\wt{L}^2_{[\f{t}2,t]} (\cB^{\f12,0})}\\
\qquad\qquad\qquad\qquad\qquad\qquad\qquad\qquad\qquad\qquad\qquad\leq C\sqrt{\e}\bigl(1+\|e^{\f{y^2}8}e^{\de|D_x|}(u_0,b_0)\|_{\cB^{\f12,0}}\bigr).\nonumber
\end{gather}
\end{subequations}
}
\end{thm}

\begin{rmk}\label{rmk:c}

\begin{itemize}
\item[(1)] The constants $\fc$ and $C$ in the above theorem are independent of $\e$. Furthermore, unlike the constant $C,$ which
 may change from line to line, the constant $\fc$ can be  fixed in Sections \ref{sect5}, \ref{sect6} and \ref{sect7}. So that by choosing $\e_0$
     to be sufficiently small, we always have $\fc\e<\lk.$ This will be crucial for us to globally control the evolution of the analytic band.
     Moreover, the constant $\bar{B}_\ka$ can be any constant if $\ka=1,$ otherwise $0.$ This restriction is only used
     in the derivation of the system \eqref{eqs3} for $(G,H).$

\item[(2)] Compared with the well-posedness result for the classical Prandtl system $(PE)$ in \cite{PZ5},
here we allow the far field $(U,B)$ to depend on the tangential variable. And exactly due to this type of far field, there is a loss of
$\fc \e$ in the decay estimates \eqref{S3eq789}.

\item[(3)] The main idea of the proof to the above theorem is to use analytic energy estimate, which is motivated by \cite{PZ5} and
which originates from \cite{Ch04}.
\end{itemize}

\end{rmk}

We remark that in the proof of Theorem \ref{thm1.1},  the role of $u$ and $b$ are the same in the process of the analytic energy estimate.
 By introducing the changes of variables: $(\bar{t},\bar{x},\bar{y})=(t,x,\f{y}{\sqrt{\ka}}),$ and denoting $(\bar{u},\bar{v},\bar{b},\bar{h})\eqdefa (u,\f{v}{\sqrt{\ka}},b,\f{h}{\sqrt{\ka}})$, we find that \eqref{MHDbl} can be equivalently reformulated as
\begin{equation*}
 \quad\left\{\begin{array}{l}
\p_{\bar{t}}\bar{u}-\f1\ka \p^2_{\bar{y}}\bar{u}+\bar{u}\p_{\bar{x}}\bar{u}+\bar{v}\p_{\bar{y}}\bar{u}=\bar{b}\p_{\bar{x}}\bar{b}+\bar{h}\p_{\bar{y}}\bar{b},\\
\p_{\bar{t}}\bar{b}-\p^2_{\bar{y}}\bar{b}+\bar{u}\p_{\bar{x}}\bar{b}+\bar{v}\p_{\bar{y}}\bar{b}=\bar{b}\p_{\bar{x}}\bar{u}+\bar{h}\p_{\bar{y}}\bar{u},\\
\p_{\bar{x}}\bar{u}+\p_{\bar{y}}\bar{v}=0, \ \ \p_{\bar{x}}\bar{b}+\p_{\bar{y}}\bar{h}=0.
\end{array}\right.
\end{equation*}
So that the proof of Theorem \ref{thm1.1} also ensures the global well-posedness result of \eqref{MHDbl} for $\ka\in ]1/2,\infty[.$
Yet for conciseness, we prefer to present the detailed result corresponding to $\ka\in ]1/2,\infty[.$


In order to do so, similar to \eqref{def:theta}, we
introduce  $\thtk(t)$ via
\begin{equation}\label{def:thetakappa}
 \quad\left\{\begin{array}{l}
\displaystyle \dthtk(t)=\tt^\f14 \|\epk\p_y \ghpk (t)\|_{\cB^{\f12,0}}+\e^{-\f12}\tt^\f54\|\UBpk(t)\|_{\cB^\f12_\h},\\
\displaystyle \thtk|_{t=0}=0.
\end{array}\right.
\end{equation}
Here the phase function $\Phi_\ka$ is defined by
\beq\label{def:Phikappa} \Phi_\ka(t,\xi)\eqdefa (\de-\lam \thtk(t))|\xi|, \eeq
and the weighted function $\Psi_\ka(t,y)$ is determined by \beq\label{def:Psikappa} \Psi_\ka(t,y)\eqdefa \f{y^2}{8\ka\tt}=\ka^{-1}{\Psi(t,y)}, \eeq
which satisfies \beq\label{S3eq21} \p_t \Psi_\ka+2\ka(\p_y \Psi_\ka)^2=0. \eeq

The second result of this paper states as follows:

\begin{thm}\label{thm1.2}
{\sl Let  $\kappa$ be in $]1/2,\infty[,$ $\bar{B}_\ka=\left\{\begin{array}{l} 1\quad \mbox{if}\quad \ka=1,\\
0\quad\mbox{otherwise},\end{array}\right.$ and $\e,\de>0.$ We assume that far field states $(U,B)$  satisfy \eqref{UBdecay} and \eqref{UBdecay2}.
Let the initial data $(u_0,b_0,\vphi_0,\psi_0)$ satisfy the compatibility conditions listed in Theorem \ref{thm1.1} and
\beq
\label{small2} \|e^\f{y^2}{8\ka} e^{\de |D_x|} (u_0,b_0,\vphi_0,\psi_0) \|_{\cB^{\f12,0}}<\infty \andf
\| e^\f{y^2}{8\ka} e^{\de |D_x|} (G_0,H_0)\|_{\cB^{\f12,0}} \leq \sqrt{\e}.
\eeq
Then there exist positive constants $ \lam,\fc$ and $\e_0(\lam,\fc,\ka,\de)$ so that for $\e\leq \e_0$ and  $\elk\eqdefa \f{2\ka-1}{4\ka^2}\ \in ]0,1/4],$
the system \eqref{eqs1} has a unique global solution $(u,b)$ which satisfies $\sup_{t\in[0,\infty[} \tht_\ka(t)\leq \f\de{2\lam}$, and
\beq\label{S3eq22}
\| \epk \ubpk\|_{\wt{L}^\infty_t (\cB^{\f12,0})}+\sqrt{\ka\ell_\ka}\| \epk \p_y \ubpk\|_{\wt{L}^2_t (\cB^{\f12,0})}\leq
\|e^\f{y^2}{8\ka} e^{\de |D_x|} (u_0,b_0) \|_{\cB^{\f12,0}}+ C\sqrt{\e}.
\eeq
Furthermore, for any $t>0$ and $\ga\in]0,1[$, there hold
\begin{subequations} \label{S3eq2345}
\begin{gather}
\label{S3eq23}
\|\tauelk \epk\ubpk\|_{\wt{L}^\infty_t (\cB^{\f12,0})}+\|\tauelk \epk \p_y \ubpk\|_{\wt{L}^2_{[\f{t}2,t]} (\cB^{\f12,0})}\\
\qquad\qquad\qquad\qquad\qquad\qquad\qquad\qquad\qquad\qquad \leq C\bigl(\|e^\f{y^2}{8\ka} e^{\de |D_x|} (u_0,b_0,\vphi_0,\psi_0) \|_{\cB^{\f12,0}}+\sqrt{\e}\bigr), \nonumber\\
\label{S3eq24}
\|\tauellk \epk \ghpk\|_{\wt{L}^\infty_t (\cB^{\f12,0})}+\|\tauellk \epk \p_y\ghpk\|_{\wt{L}^2_{[\f{t}2,t]} (\cB^{\f12,0})}\\
\qquad\qquad\qquad\qquad\qquad\qquad\qquad\qquad\qquad\qquad \qquad\leq C\sqrt{\e}\bigl(1+\|e^{\f{y^2}{8\ka}}e^{\de|D_x|}(u_0,b_0)\|_{\cB^{\f12,0}}\bigr), \nonumber\\
\label{S3eq25}
\|\tauellk \egpk \ubpk\|_{\wt{L}^\infty_t (\cB^{\f12,0})}+\|\tauellk \egpk \p_y \ubpk\|_{\wt{L}^2_{[\f{t}2,t]} (\cB^{\f12,0})} \\
\qquad\qquad\qquad\qquad\qquad\qquad\qquad\qquad\qquad\qquad \qquad\leq C\sqrt{\e}\bigl(1+\|e^{\f{y^2}{8\ka}}e^{\de|D_x|}(u_0,b_0)\|_{\cB^{\f12,0}}\bigr). \nonumber
\end{gather}
\end{subequations}
}
\end{thm}

\smallskip

 Let us end this introduction by the notations that we shall use
in this context.\vspace{0.2cm}

For~$a\lesssim b$, we mean that there is a uniform constant $C,$
which may be different on different lines, such that $a\leq Cb$. $\tt=1+t,$  $(a\ |\
b)_{L^2_+}\eqdefa\int_{\R^2_+}a(x,y) {b}(x,y)\,dx\,dy$
 stands for
the $L^2$ inner product of $a,b$ on $\R^2_+$ (resp. $\R_+$) and
$L^p_+=L^p(\R^2_+)$  with $\R^2_+\eqdefa\R\times\R_+.$ For $X$ a Banach space
and $I$ an interval of $\R,$ we denote by $L^q(I;\,X)$ the set of
measurable functions on $I$ with values in $X,$ such that
$t\longmapsto\|f(t)\|_{X}$ belongs to $L^q(I).$  In particular,  we denote by
$L^p_T(L^q_{\rm h}(L^r_{\rm v}))$ the space $L^p([0,T];
L^q(\R_{x_{}};L^r(\R_{y}^+))).$  Finally,
 $
\left(d_{k}\right)_{k\in\Z}$  designates  a nonnegative generic element in the sphere
of $\ell^1(\Z)$
 so that  $\sum_{k\in\Z}d_k=1.$

\medskip

\renewcommand{\theequation}{\thesection.\arabic{equation}}
\setcounter{equation}{0}
\section{Sketch of the proof to theorem \ref{thm1.1} and \ref{thm1.2}}\label{sect3}

The goal of this section is to sketch the structure of the proof to Theorems \ref{thm1.1} and \ref{thm1.2}.
Let us start with the outline of the proof to  Theorem \ref{thm1.1}.

In what follows,  we shall always assume that $t<T^*$ with $T^*$ being defined by
\beq\label{def:T^*}
T^* \eqdefa \sup\bigl\{ t>0,\quad \tht(t)<\de/\lam \ \bigr\}\with \tht(t)\ \ \mbox{being determined by \eqref{def:theta}}.
\eeq
So that by virtue of \eqref{def:Phi}, for any $t<T^*$,  the following convex inequality holds
\beq\label{S3eq11}
\Phi(t,\xi)\leq\Phi(t,\xi-\eta)+ \Phi(t,\eta) \qquad \ \forall \xi,\eta\in\R.
\eeq

In section \ref{sect5}, we shall deal with the {\it a priori} decay estimates for the analytic solutions of \eqref{eqs2}.
\begin{prop}\label{prop3.1}
{\sl Let $(\vphi,\psi)$ be a smooth enough solution of \eqref{eqs2}. Then if
$\ka\in ]0,2[,$ there exist positive constants $\fc, \lam_0$ and a small enough constant $\e_0$ so that for any $\e\leq \e_0, \lam\geq\lam_0$ and for
 any $t<T^*$, we have
\beq\label{S3eq12as}
\begin{split}
\|\tauk \ep\ppp&\|_{\wt{L}^\infty_t (\cB^{\f12,0})}+\|\tauk \p_y\ep\ppp\|_{\wt{L}^2_{[\f{t}2,t]} (\cB^{\f12,0})}\\
&\lesssim
\|e^\f{y^2}8 e^{\de |D_x|}(\vphi_0,\psi_0) \|_{\cB^{\f12,0}}+ \sqrt{\e}\with  \lk\eqdefa \f{\ka(2-\ka)}4.
\end{split}
\eeq}
\end{prop}

In section \ref{sect6},  we shall deal with the {\it a priori} decay estimates for the analytic solutions of \eqref{eqs1}.
\begin{prop}\label{prop3.2}
{\sl Let $(u,b)$ be a smooth enough solution of \eqref{eqs1}. Then if
$\ka\in ]0,2[,$  there exist positive constants $\fc, \lam_0$ and a small enough constant $\e_0$ so that for any $\e\leq\e_0,$
$\lam\geq\lam_0$ and  for any $t<T^*$, we have
\beq
\begin{split}
\|\taulk \ep\ubp&\|_{\wt{L}^\infty_t (\cB^{\f12,0})} +\|\taulk \ep\p_y\ubp\|_{\wt{L}^2_{[\f{t}2,t]}  (\cB^{\f12,0})}\\
 &\lesssim \|e^\f{y^2}8 e^{\de |D_x|}(u_0, b_0,\vphi_0,\psi_0) \|_{\cB^{\f12,0}}+\sqrt{\e}\with \lk\eqdefa \f{\ka(2-\ka)}4.\end{split}
\eeq}
\end{prop}

In section \ref{sect7},  we shall deal with the {\it a priori} decay estimates of $(G,H)$ which will be
the most crucial ingredient used in the control of the analytic radius.

\begin{prop}\label{prop3.3}
{\sl Let $(G,H)$ be determined by \eqref{defGH}. Then if $\ka\in ]0,2[$,  there exist positive constants $\fc, \lam_0$ and a small enough constant $\e_0$ so that for any $\e\leq\e_0, \lam\geq\lam_0$ and  for any $t<T^*$, we have
\beq
\begin{split}
\|&\taullk \ep\ghp\|_{\wt{L}^\infty_t (\cB^{\f12,0})} +\|\taullk \ep\p_y\ghp\|_{\wt{L}^2_{[\f{t}2,t]}  (\cB^{\f12,0})} \\
&+\int_0^t \ttau^\f14 \|\ep \p_y\ghp(\tau) \|_{\cB^{\f12,0}}d\tau \leq C\sqrt{\e}\bigl(1+\|e^{\f{y^2}8}e^{\de|D_x|}(u_0,b_0)\|_{\cB^{\f12,0}}\bigr),\end{split}
\eeq where $ \lk\eqdefa \f{\ka(2-\ka)}4.$}
\end{prop}

With the above propositions, we still need the following lemma concerning the relationship
 between the functions $(G,H)$ given by \eqref{defGH} and the solutions of \eqref{eqs1} and \eqref{eqs2}, which will be frequently used in the subsequent sections.

\begin{lem}\label{lem3.1}
{\sl  Let $(u,b)$ and $(\vphi,\psi)$ be smooth enough solution of \eqref{eqs1} and \eqref{eqs2} respectively on $[0,T]$. Let $(G,H)$, $\Phi$ and $\Psi$ be defined respectively   by \eqref{defGH}, \eqref{def:Phi} and \eqref{def:Psi}. Then if  $\ka\in ]0,2[,$ for any $\ga\in]0,1[$ and $t\leq T$, one has
\begin{subequations} \label{S3eq567}
\begin{gather}
\| \egp \Dhk \vp\|_{L^2_+}\lesssim \tt^\f12 \|\ep  \Dhk \gp\|_{L^2_+}, \ \ \ \| \egp \Dhk \pp\|_{L^2_+}\lesssim \tt^\f12 \|\ep  \Dhk \hp\|_{L^2_+}\label{S3eq15} ,\\
\| \egp \Dhk \up\|_{L^2_+}\lesssim \|\ep  \Dhk \gp\|_{L^2_+},\ \qquad \ \| \egp \Dhk \bp\|_{L^2_+}\lesssim  \|\ep  \Dhk \hp\|_{L^2_+} \label{S3eq16},\\
\| \egp \Dhk \p_y \up \|_{L^2_+}\lesssim  \|\ep  \Dhk \p_y \gp\|_{L^2_+},\ \ \ \| \egp \Dhk \p_y \bp\|_{L^2_+}\lesssim \|\ep  \Dhk \p_y\hp\|_{L^2_+}\label{S3eq17} .
\end{gather}
\end{subequations}
}
\end{lem}
Let us postpone the proof of this lemma till the end of this section.

We are now in a position to complete the proof of Theorem \ref{thm1.1}.

\begin{proof}[Proof of Theorem \ref{thm1.1}]
The general strategy to prove the existence result for a non-linear partial differential equation is first to construct appropriate approximate solutions, then perform uniform estimates for such approximate solution sequence, and finally pass to the limit in the approximate problem. For simplicity, here we only present the {\it a priori} estimate for smooth enough solutions of \eqref{eqs1} in the analytical framework.

Indeed, let $(u,b)$ and $(\vphi,\psi)$ be smooth enough solutions of \eqref{eqs1} and \eqref{eqs2} respectively on $[0,T^\star[$, where $T^\star$ is the maximal time of existence of the solution. Let $(G,H)$ be defined by \eqref{defGH}. For any $t<T^*$ (of course here $T^*\leq T^\star$) with $T^*$ being defined by \eqref{def:T^*}, we deduce from \eqref{def:theta} that
$$\tht(t) \leq\int_0^t \bigl(\ttau^\f14 \|\ep\p_y \ghp (\tau)\|_{\cB^{\f12,0}}+\e^{-\f12}\ttau^\f54\|\UBp(\tau)\|_{\cB^\f12_\h}\bigr)d\tau.$$
Notice that we can take $\e_0$ to be so small that $\e_0<\f{l_\ka}{\fc}.$ This together with
the assumption \eqref{UBdecay2} and Proposition \ref{prop3.3} ensures that
$$
\tht(t) \leq C\sqrt{\e}\bigl(1+\|e^{\f{y^2}8}e^{\de|D_x|}(u_0,b_0)\|_{\cB^{\f12,0}}\bigr)\quad \mbox{for any}\ t<T^\ast \andf \e\leq \e_0.
$$
Then by  taking $\e_0$ so small that $\e_0\leq \min\Bigl(\Bigl(\f\de{2C\lam\bigl(1+\|e^{\f{y^2}8}e^{\de|D_x|}(u_0,b_0)\|_{\cB^{\f12,0}}\bigr)}\Bigr)^2, \f{l_\ka}{\fc}\Bigr)$, we achieve
$$
\sup_{t\in[0,T^*[} \tht(t)\leq \f\de{2\lam} \quad \ \mbox{for any}\ \e\leq\e_0.
$$

Then in view of \eqref{def:T^*}, we deduce by a continuous argument that $T^*=\infty$. And \eqref{S6eq21} holds for $t=\infty$, which implies \eqref{S3eq6}. Furthermore, Propositions \ref{prop3.2} and \ref{prop3.3} imply \eqref{S3eq7} and \eqref{S3eq8} respectively. Finally, along
with  Lemma \ref{lem3.1}, we deduce \eqref{S3eq9} from Proposition \ref{prop3.3}. This completes the existence part of Theorem \ref{thm1.1}. The uniqueness part of Theorem \ref{thm1.1} has been proved in \cite{XY19}, which can also be proved
by similar {\it a priori} estimates for the difference between any two solutions of \eqref{eqs1}. (One may check
\cite{ZZ} for the detailed proof to uniqueness of analytic solution to the classical Prandtl system $(PE).$)
\end{proof}

Next
let us  turn to the outline of the proof to Theorem \ref{thm1.2}.
Firstly, similar to \eqref{def:T^*}, we define $T^*_\ka$ via
\beq\label{def:T^*kappa}
T^*_\ka \eqdefa \sup\bigl\{ t>0, \quad \thtk(t)<\de/\lam\ \bigr\}\with \thtk(t)\ \ \mbox{being determined by \eqref{def:thetakappa}}.
\eeq
So that by virtue of \eqref{def:Phikappa}, for any $t<T^*_\ka$, the convex inequality \eqref{S3eq11} still holds for $\Phi_\ka.$

In section \ref{sect5}, we shall deal with the {\it a priori} decay estimates for the analytic solutions of \eqref{eqs2}
in the case when $\ka\in]1/2,\infty[$.

\begin{prop}\label{prop3.4}
{\sl Let $(\vphi,\psi)$ be a smooth enough solution of \eqref{eqs2}. Then if $\ka\in]1/2,\infty[,$
there exist positive  constants $\lam_0, \fc$ and a small enough constant $\e_0$  so that for any $\e\leq\e_0, \lam\geq\lam_0$ and
 for any $t<T^*_\ka$, we have
\beq
\begin{split}
\|\etauk \epk\pppk&\|_{\wt{L}^\infty_t (\cB^{\f12,0})}+\|\etauk \epk\p_y\pppk\|_{\wt{L}^2_{[\f{t}2,t]} (\cB^{\f12,0})}\\
& \qquad\lesssim
\|e^\f{y^2}{8\ka} e^{\de |D_x|} (\vphi_0,\psi_0) \|_{\cB^{\f12,0}}+ \sqrt{\e}\with \elk\eqdefa \f{2\ka-1}{4\ka^2}.
\end{split}
\eeq}
\end{prop}

In section \ref{sect6},  we shall deal with the following {\it a priori} decay estimates for the analytic solutions of \eqref{eqs1}.

\begin{prop}\label{prop3.5}
{\sl Let $(u,b)$ be a smooth enough solution of \eqref{eqs1}. Then if $\ka\in]1/2,\infty[,$
there exist positive  constants $\lam_0, \fc$ and a small enough constant $\e_0$  so that for any $\e\leq\e_0, \lam\geq\lam_0$ and for any $t<T^*_\ka$, we have
\beq
\begin{split}
\|\etaulk \epk\ubpk&\|_{\wt{L}^\infty_t (\cB^{\f12,0})} +\|\etaulk \epk\p_y\ubpk\|_{\wt{L}^2_{[\f{t}2,t]}  (\cB^{\f12,0})}\\
 &\lesssim \|e^\f{y^2}{8\ka} e^{\de |D_x|} (u_0,b_0,\vphi_0,\psi_0) \|_{\cB^{\f12,0}}+\sqrt{\e} \with \elk\eqdefa \f{2\ka-1}{4\ka^2}.\end{split}
\eeq}
\end{prop}

In section \ref{sect7},  we shall deal with the {\it a priori} decay estimates of $(G,H).$
\begin{prop}\label{prop3.6}
{\sl Let $(G,H)$ be determined by \eqref{defGH}. Then   if $\ka\in]1/2,\infty[,$
there exist positive  constants $\lam_0, \fc$ and a small enough constant $\e_0$  so that for any $\e\leq\e_0, \lam\geq\lam_0$ and for any $t<T^*_\ka$, we have
\beq
\begin{split}
&\|\tauellk \epk\ghpk\|_{\wt{L}^\infty_t (\cB^{\f12,0})} +\|\tauellk \epk\p_y\ghpk\|_{\wt{L}^2_{[\f{t}2,t]}  (\cB^{\f12,0})} \\
&+\int_0^t \ttau^\f14 \|\epk \p_y\ghpk(\tau) \|_{\cB^{\f12,0}}d\tau \leq C\sqrt{\e}\bigl(1+\|e^{\f{y^2}{8\ka}}e^{\de|D_x|}(u_0,b_0)\|_{\cB^{\f12,0}}\bigr),\end{split}
\eeq
where $\elk\eqdefa \f{2\ka-1}{4\ka^2}.$ }
\end{prop}

We also need the following version of Lemma \ref{lem3.1} for the case when $\ka>\f12$.

\begin{lem}\label{lem3.2}
{\sl  Let $(u,b)$ and $(\vphi,\psi)$ be smooth enough solution of \eqref{eqs1} and \eqref{eqs2} respectively on $[0,T]$. Let $(G,H)$, $\Phi_\ka$ and $\Psi_\ka$ be defined respectively   by \eqref{defGH}, \eqref{def:Phikappa} and \eqref{def:Psikappa}. Then if  $\ka\in ]1/2,\infty[,$ for any $\ga\in]0,1[$ and $t\leq T$, one has
\begin{subequations} \label{S3eq3012}
\begin{gather}
\| \egpk \Dhk \vpk\|_{L^2_+}\lesssim \tt^\f12 \|\epk  \Dhk \gpk\|_{L^2_+}, \ \ \ \| \egpk \Dhk \ppk\|_{L^2_+}\lesssim \tt^\f12 \|\epk  \Dhk \hpk\|_{L^2_+}\label{S3eq30} ,\\
\| \egpk \Dhk \upk\|_{L^2_+}\lesssim \|\epk  \Dhk \gpk\|_{L^2_+},\ \ \ \| \egpk \Dhk \bpk\|_{L^2_+}\lesssim  \|\epk  \Dhk \hpk\|_{L^2_+}\label{S3eq31} ,\\
\| \egpk \Dhk \p_y \upk \|_{L^2_+}\lesssim  \|\epk  \Dhk \p_y \gpk\|_{L^2_+},\ \ \ \| \egpk \Dhk \p_y \bpk\|_{L^2_+}\lesssim \|\epk  \Dhk \p_y\hpk\|_{L^2_+}\label{S3eq32} .
\end{gather}
\end{subequations}
}
\end{lem}

We shall present the proof of this lemma at the end of this section.

We now present the proof of Theorem \ref{thm1.2}.

\begin{proof}[Proof of Theorem \ref{thm1.2}]
 Along the same line to the proof of Theorem \ref{thm1.1}, let $(u,b)$ and $(\vphi,\psi)$ be smooth enough solutions of \eqref{eqs1} and \eqref{eqs2} respectively on $[0,T^\star_\ka[$, where $T^\star_\ka$ is the maximal time of existence of the solution. Let $(G,H)$ be defined by \eqref{defGH}. For any $t<T^*_\ka$ (of course here $T^*_\ka\leq T^\star_\ka$) with $T^*_\ka$ being defined by \eqref{def:T^*kappa}, we deduce from \eqref{def:thetakappa} that
$$\thtk(t) \leq\int_0^t \bigl(\ttau^\f14 \|\epk\p_y \ghpk (\tau)\|_{\cB^{\f12,0}}+\e^{-\f12}\ttau^\f54\|\UBpk(\tau)\|_{\cB^\f12_\h}\bigr)d\tau.$$
We take $\e_0\leq \frac{\ell_\ka}{\fc},$ then we deduce from
the assumption \eqref{UBdecay2} and Proposition \ref{prop3.6} that
$$
\thtk(t) \leq C\sqrt{\e}\bigl(1+\|e^{\f{y^2}{8\ka}}e^{\de|D_x|}(u_0,b_0)\|_{\cB^{\f12,0}}\bigr)\quad\mbox{for any}\ \ \e\leq\e_0\andf t<T^\ast_\ka.
$$
If we take $\e_0$ so small that $\e_0\leq \min\Bigl(\Bigl(\f\de{2C\lam\bigl(1+\|e^{\f{y^2}{8\ka}}e^{\de|D_x|}(u_0,b_0)\|_{\cB^{\f12,0}}\bigr)}\Bigr)^2, \f{l_\ka}{\fc}\Bigr)$,
 then we achieve
$$
\sup_{t\in[0,T^*_\ka[} \thtk(t)\leq \f\de{2\lam}  \qquad \mbox{ for any}\  \e\leq\e_0.
$$
Then in view of \eqref{def:T^*kappa}, we get by a continuous argument that $T^*_\ka=\infty$. And \eqref{S6eq25} holds for $t=\infty$, which implies \eqref{S3eq22}. Moreover, Propositions \ref{prop3.5} and \ref{prop3.6} ensure \eqref{S3eq23} and \eqref{S3eq24} respectively.
Finally, Proposition \ref{prop3.6} together with Lemma \ref{lem3.2} implies \eqref{S3eq25}. This finishes the existence part of Theorem \ref{thm1.2}. The uniqueness part has already been proved in \cite{XY19}.
\end{proof}

Before the proof of Lemmas \ref{lem3.1} and \ref{lem3.2}, following \cite{IV16, PZ5}, we first introduce the following
Poincar\'e type inequalities.

\begin{lem}\label{lem:Poincare}
{\sl Let $\Psi_\ka$ be defined by \eqref{def:Psikappa}. Let $f$ be a smooth enough function on $\R^2_+$ which decays to zero sufficiently fast as y approaching to $+\infty$. Then we have
\begin{subequations} \label{S3eq33ab}
\begin{gather}
\|\epk \p_y f\|^2_{L^2_+}\geq \f1{2\ka\tt}\| \epk f\|^2_{L^2_+}, \label{S3eq33}\\
\|\epk \p_y f\|^2_{L^2_+}\geq \f1{4\ka\tt}\| \epk f\|^2_{L^2_+}+\f1{16\ka^2\tt^2}\|\epk y f\|^2_{L^2_+},\label{S8eq33aq}
\end{gather}
\end{subequations}
Especially, when $\ka=1$, we have
\begin{subequations} \label{S3eq34ab}
\begin{gather}
\label{S3eq34}
\|\ep \p_y f\|^2_{L^2_+}\geq \f1{2\tt} \| \ep f\|^2_{L^2_+},\\
\label{S8eq34aq}
\|\ep \p_y f\|^2_{L^2_+}\geq \f1{4\tt} \| \ep f\|^2_{L^2_+}+\f1{16\tt^2}\|\ep y f\|^2_{L^2_+}.
\end{gather}
\end{subequations}
 where \eqref{S3eq34} recovers (3.1) of \cite{PZ5}.}
\end{lem}

\begin{proof} As in \cite{IV16, PZ5}, we get,
by using integration by parts, that
\beq\label{S3eq34a}
\begin{split}
\int_{\R^2_+} f^2(x,y) &e^\f{y^2}{4\ka\tt} \,dxdy =\int_{\R^2_+}(\p_y y)f^2(x,y) e^\f{y^2}{4\ka\tt} \,dxdy \\
&=-2\int_{\R^2_+} yf(x,y)\p_yf(x,y) e^\f{y^2}{4\ka\tt} \,dxdy-\f1{2\ka\tt} \int_{\R^2_+} y^2 f^2(x,y) e^\f{y^2}{4\ka\tt} \,dxdy \\
&\leq 2\ka\tt \int_{\R^2_+} (\p_yf)^2(x,y) e^\f{y^2}{4\ka\tt} \,dxdy.
\end{split}
\eeq
This leads to \eqref{S3eq33}.

We also observe from \eqref{S3eq34a} that
\beqo
\begin{split}
\int_{\R^2_+} f^2(x,y) e^\f{y^2}{4\ka\tt} \,dxdy
\leq &4\ka\tt \int_{\R^2_+} (\p_yf)^2(x,y) e^\f{y^2}{4\ka\tt} \,dxdy\\
&-\f1{4\ka\tt}\int_{\R^2_+} y^2 f^2(x,y) e^\f{y^2}{4\ka\tt} \,dxdy,
\end{split}
\eeqo which yields \eqref{S8eq33aq}.
\end{proof}

Let us end this section with the proof of Lemmas \ref{lem3.1} and \ref{lem3.2}.

\begin{proof}[Proof of Lemma \ref{lem3.1}] The proof of this lemma basically follows from that of Lemma 3.2 in \cite{PZ5}.
Due to $(u,b)=\p_y (\vphi,\psi)$ and $\vphi|_{y=0}=\psi|_{y=0}=0$, we deduce from \eqref{defGH} that
\begin{subequations} \label{S3eq356}
\begin{gather}
\label{S3eq35}
\vphi(t,x,y)=e^{-2\Psi(t,y)}\int_0^y e^{2\Psi(t,y')}G(t,x,y')\, dy',\\
\label{S3eq36}
\psi(t,x,y)=e^{-2\Psi_\ka(t,y)}\int_0^y e^{2\Psi_\ka(t,y')}H(t,x,y')\, dy',
\end{gather}
\end{subequations} with $\Psi$ and $\Psi_\ka$ being determined respectively by \eqref{def:Psi} and \eqref{def:Psikappa}.

Thanks to \eqref{S3eq356}, we infer
\begin{subequations} \label{S3eq378}
\begin{gather}
\label{S3eq37}
u=G-\fyt e^{-2\Psi}\int_0^y e^{2\Psi(y')}G \,dy',
\\
\label{S3eq38}
b=H-\fykt e^{-2\Psi_\ka}\int_0^y e^{2\Psi_\ka(y')}G\, dy',
\end{gather}
\end{subequations}
and
\begin{subequations} \label{S3eq394}
\begin{gather}
\label{S3eq39}
\p_yu=\p_yG-\fyt G+\f1{2\tt}\Bigl(\f{y^2}{2\tt}-1\Bigr) e^{-2\Psi}\int_0^y e^{2\Psi(y')}G\, dy',
\\
\label{S3eq40}
\p_yb=\p_yH-\fyt H+\f1{2\ka\tt}\Bigl(\f{y^2}{2\ka\tt}-1\Bigr) e^{-2\Psi_ka}\int_0^y e^{2\Psi_ka(y')}H\, dy'.
\end{gather}
\end{subequations}

Let us use the relationship between $b$ and $H$ to prove the second inequalities of \eqref{S3eq15}-\eqref{S3eq17}.
  The other inequalities follows by taking $\kappa=1$ in the proof.

In view of \eqref{S3eq38} and
\beq\label{ineq1}
\sup_{y\in[0,\infty[} \Bigl(e^{-y^2}\int_0^y e^{z^2}\,dz\Bigr)<\infty,
\eeq
we infer that (if $\ka\in]0,2[\Rightarrow\f2\ka-1>0$)
\beqo\begin{split}
\|\egp \Dhk \bp\|_{L^2_+}&\lesssim \|\ep \Dhk \hp\|_{L^2_+} \\
&\quad+\tt^{-1}\bigl\|ye^{(\ga-\f2\ka)\Psi} \bigl( \int_0^y e^{2(\f2\ka-1)\Psi}dy' \bigr)^\f12 \bigl( \int_0^\infty |\ep\Dhk\hp|^2 dy \bigr)^\f12 \bigr\|_{L^2_+}\\
&\lesssim \|\ep \Dhk \hp\|_{L^2_+}\Bigl(1 +\tt^{-1}\|ye^{(\ga-\f2\ka)\Psi} \bigl( \int_0^y e^{2(\f2\ka-1)\Psi}dy' \bigr)^\f12 \|_{L^2_{\rm v}}\Bigr)\\
&\lesssim \|\ep \Dhk \hp\|_{L^2_+}\bigl(1 +\tt^{-\f34}\|ye^{(\ga-1)\Psi}  \|_{L^2_{\rm v}}\bigr),
\end{split}\eeqo
from which and $\ga<1$, we deduce the second inequality of \eqref{S3eq16}.

\eqref{S3eq15} follows from \eqref{S3eq16} and Lemma \ref{lem:Poincare}.

Whereas due to $\lim_{y\rto+\infty}H(t,x,y)=0$, we write $H=-\int_y^\infty \p_y Hdy'$, and then, by using
\beq\label{ineq2}
\sup_{y\in[0,\infty[} \Bigl(e^{y^2}\int_y^\infty e^{-z^2}\,dz\Bigr)<\infty,
\eeq
we infer
\beqo\begin{split}
|\Dhk \hp|&\leq \bigl(\int_y^\infty e^{-2\Psi}dy'\bigr)^\f12\bigl( \int_0^\infty |\ep\Dhk\p_y\hp|^2dy\bigr)^\f12 \\&\lesssim \tt^\f14 e^{-\Psi} \bigl( \int_0^\infty |\ep\Dhk\p_y\hp|^2dy\bigr)^\f12,
\end{split}\eeqo
from which, \eqref{S3eq40} and \eqref{ineq1}, we infer that (if $\ka\in]0,2[\Rightarrow\f2\ka-1>0$)
\beqo\begin{split}
\|\egp\Dhk\p_y\bp\|_{L^2_+}&\lesssim\|\ep\Dhk\p_y\hp \|_{L^2_+}\Bigl(1+\tt^{-\f34}\|ye^{(\ga-1)\Psi} \|_{L^2_\rmv}\\
&\qquad +\tt^{-\f34} \bigl\|\bigl(\f{y^2}{\tt}-1\bigr)e^{(\ga-\f2\ka)\Psi} \int_0^y e^{(\f2\ka-1)\Psi}dy' \bigr\|_{L^2_\rmv} \Bigr)\\
 &\lesssim \|\ep\Dhk\p_y\hp \|_{L^2_+}\Bigl(1+\tt^{-\f34}\|ye^{(\ga-1)\Psi} \|_{L^2_\rmv}+\tt^{-\f14}
  \bigl\|\bigl(\f{y^2}{\tt}-1\bigr)e^{(\ga-1)\Psi}  \bigr\|_{L^2_\rmv} \Bigr),
\end{split}\eeqo
from which and $\ga<1$, we deduce \eqref{S3eq17}. We thus conclude the proof of Lemma \ref{lem3.1}.
\end{proof}

\begin{proof}[Proof of Lemma \ref{lem3.2}]
The proof of this lemma  is similar to that of Lemma \ref{lem3.1}. For completeness,
we shall deal with the first inequalities of \eqref{S3eq30}-\eqref{S3eq32}.

In view of \eqref{S3eq37} and \eqref{ineq1}, we infer that (if $\ka>\f12\Rightarrow 2\ka-1>0$)
\beqo\begin{split}
\|\egpk \Dhk \up\|_{L^2_+}&\lesssim \|\epk \Dhk \gp\|_{L^2_+}\\&\qquad +\tt^{-1}\bigl\|ye^{(\ga-2\ka)\Psi_\ka} \bigl( \int_0^y e^{2(2\ka-1)\Psi_\ka}dy' \bigr)^\f12 \bigl( \int_0^\infty |\epk\Dhk\gp|^2 dy \bigr)^\f12 \bigr\|_{L^2_+}\\
&\lesssim \|\epk \Dhk \gp\|_{L^2_+}\Bigl(1 +\tt^{-1}\bigl\|ye^{(\ga-2\ka)\Psi_\ka} \bigl( \int_0^y e^{2(2\ka-1)\Psi_\ka}dy' \bigr)^\f12 \bigr\|_{L^2_\rmv}\Bigr)\\
&\lesssim \|\epk \Dhk \gp\|_{L^2_+}\bigl(1 +\tt^{-\f34}\|ye^{(\ga-1)\Psi_\ka}  \|_{L^2_\rmv}\bigr),
\end{split}\eeqo
from which and $\ga<1$, we deduce the first inequality of \eqref{S3eq31}.

\eqref{S3eq30} follows from \eqref{S3eq31} and Lemma \ref{lem:Poincare}.

Whereas due to $\lim_{y\rto+\infty}G(t,x,y)=0$, we write $G=-\int_y^\infty \p_y Gdy'$, and then it follows from  \eqref{ineq2} that
\beqo\begin{split}
|\Dhk \gp|&\leq \bigl(\int_y^\infty e^{-2\Psi_\ka}dy'\bigr)^\f12\bigl( \int_0^\infty |\epk\Dhk\p_y\gp|^2dy\bigr)^\f12 \\
&\lesssim \tt^\f14 e^{-\Psi_\ka} \bigl( \int_0^\infty |\epk\Dhk\p_y\gp|^2dy\bigr)^\f12,
\end{split}\eeqo
from which, \eqref{S3eq39} and \eqref{ineq1}, we infer that(if $\ka>\f12\Rightarrow 2\ka-1>0$)
\beqo\begin{split}
\|\egpk\Dhk\p_y\up\|_{L^2_+}
&\lesssim\|\epk\Dhk\p_y\gp \|_{L^2_+}\Bigl(1+\tt^{-\f34}\|ye^{(\ga-1)\Psi_\ka} \|_{L^2_\rmv}\\
&\qquad +\tt^{-\f34 } \bigl\|\bigl(\f{y^2}{\tt}-1\bigr)e^{(\ga-2\ka)\Psi_\ka} \int_0^y e^{(2\ka-1)\Psi_\ka}dy' \bigr\|_{L^2_\rmv} \Bigr)\\
&\lesssim \|\epk\Dhk\_y\gp \|_{L^2_+}\Bigl(1+\tt^{-\f34}\|ye^{(\ga-1)\Psi_\ka} \|_{L^2_\rmv}\\
&\qquad\qquad\qquad\qquad\quad\quad+\tt^{-\f14}
 \bigl\|\bigl(\f{y^2}{\tt}-1\bigr)e^{(\ga-1)\Psi_\ka}  \bigr\|_{L^2_\rmv} \Bigr),
\end{split}\eeqo
from which and $\ga<1$, we deduce the first inequality of \eqref{S3eq32}. We thus conclude the proof of Lemma \ref{lem3.2}.
\end{proof}

\renewcommand{\theequation}{\thesection.\arabic{equation}}
\setcounter{equation}{0}
\section{Preliminaries and some technical Lemmas}\label{sect4}

In this section, we will present a few technical lemmas which will be frequently used in the subsequent sections.

In order to overcome the difficulty that one can not use Gronwall
type argument in the framework of Chemin-Lerner space, we  need to use the time-weighted
Chemin-Lerner norm, which was introduced by Paicu and the second author in
\cite{PZ1}.

\begin{Def}\label{def2.3} {\sl Let $f(t)\in L^1_{\mbox{loc}}(\R_+)$
be a nonnegative function and $t_0, t\in [0,\infty].$ We define \beq \label{1.4}
\|a\|_{\wt{L}^p_{[t_0,t];f}(\cB^{s,0})}\eqdefa
\sum_{k\in\Z}2^{ks}\Bigl(\int_{t_0}^t f(t')\|\D_k^{\rm
h}a(t')\|_{L^2_+}^p\,dt'\Bigr)^{\f1p}. \eeq When $t_0=0,$ we simplify the notation $\|a\|_{\wt{L}^p_{[0,t]:f}(\cB^{s,0})}$ as $\|a\|_{\wt{L}^p_{t,f}(\cB^{s,0})}.$}
\end{Def}

We also recall the following anisotropic
Bernstein type lemma from \cite{CZ1, Pa02}:

\begin{lem} \label{lem:Bern}
 {\sl Let $\cB_{\rm h}$ be a ball
of~$\R_{\rm h}$, and~$\cC_{\rm h}$  a ring of~$\R_{\rm
h}$; let~$1\leq p_2\leq p_1\leq \infty$ and ~$1\leq q\leq \infty.$
Then there holds:

\smallbreak\noindent If the support of~$\wh a$ is included
in~$2^k\cB_{\rm h}$, then
\[
\|\partial_{x}^\alpha a\|_{L^{p_1}_{\rm h}(L^{q}_{\rm v})} \lesssim
2^{k\left(|\al|+\left(\f 1 {p_2}-\f 1 {p_1}\right)\right)}
\|a\|_{L^{p_2}_{\rm h}(L^{q}_{\rm v})}.
\]

\smallbreak\noindent If the support of~$\wh a$ is included
in~$2^k\cC_{\rm h}$, then
\[
\|a\|_{L^{p_1}_{\rm h}(L^{q}_{\rm v})} \lesssim
2^{-kN} \|\partial_{x}^N a\|_{L^{p_1}_{\rm
h}(L^{q}_{\rm v})}.
\]
}
\end{lem}

In the following context, we shall constantly use  Bony's decomposition (see \cite{Bo}) for
the horizontal variable: \ben\label{Bony} fg=T^{\rm h}_fg+T^{\rm
h}_{g}f+R^{\rm h}(f,g), \een where \beno T^{\rm h}_fg\eqdefa \sum_kS^{\rm
h}_{k-1}f\Delta_k^{\rm h}g\andf R^{\rm
h}(f,g)\eqdefa\sum_k{\Delta}_k^{\rm h}f\widetilde{\Delta}_{k}^{\rm h}g
\with \widetilde{\Delta}_k^{\rm h}g\eqdefa
\displaystyle\sum_{|k-k'|\le 1}\Delta_{k'}^{\rm h}g. \eeno

In the rest of this paper, we always designate $\Psi$ and $\Psi_\ka$  to be the functions
defined respectively by \eqref{def:Psi} and \eqref{def:Psikappa}.
We next present the weighted energy estimate for the linear heat equations:

\begin{lem}\label{lem:heateqs}
{\sl Let $f$ be a smooth enough function on $\R^2_+$, which decays to zero sufficiently fast as y approaching to $+\infty$ and satisfies $f\p_yf|_{y=0}=0$. Then we have, if $\ka \in ]0,2[$,
\begin{subequations} \label{S4eqa12}
\begin{gather}
\label{S4eq1} \bigl(\p_tf-\p^2_yf\big|\epp f \bigr)_{L^2_+}\geq \f12 \f{d}{dt}\|\ep f\|^2_{L^2_+}+\f12\|\ep \p_yf\|^2_{L^2_+},
\\
\label{S4eq2} \bigl(\p_tf-\ka\p^2_yf\big|\epp f \bigr)_{L^2_+}\geq \f12 \f{d}{dt}\|\ep f\|^2_{L^2_+}+\f{\ka(2-\ka)}2\|\ep \p_yf\|^2_{L^2_+},
\end{gather}
\end{subequations}
and if $\ka>\f12$,
\begin{subequations} \label{S4eqa34}
\begin{gather}
\label{S4eq3} \bigl(\p_tf-\p^2_yf\big|\eppk f \bigr)_{L^2_+}\geq \f12 \f{d}{dt}\|\epk f\|^2_{L^2_+}+\f{2\ka-1}{2\ka}\|\epk \p_yf\|^2_{L^2_+},
\\
\label{S4eq4} \bigl(\p_tf-\ka\p^2_yf\big|\eppk f \bigr)_{L^2_+}\geq \f12 \f{d}{dt}\|\epk f\|^2_{L^2_+}+\f{\ka}2\|\epk \p_yf\|^2_{L^2_+}.
 \end{gather}
\end{subequations}
}
\end{lem}
\begin{proof}
Indeed it is enough to prove the following inequality for any $\al, \be>0,$
\beq\label{S4eq5} \bigl(\p_tf-\be\p^2_yf\big|e^{2\al\Psi} f \bigr)_{L^2_+}\geq \f12 \f{d}{dt}\|e^{\al\Psi} f\|^2_{L^2_+}+\bigl(\be-\f{\be^2\al}2\bigr)\|e^{\al\Psi} \p_yf\|^2_{L^2_+}.
\eeq
Then \eqref{S4eqa12} and \eqref{S4eqa34} follow by taking $\al\in\{1,\f1\ka\}$ and $\be\in\{1,\ka\}$ in \eqref{S4eq5}.

We first get, by using integration by parts,  that
\beq\label{S4eq6} \bigl(\p_tf\big|e^{2\al\Psi} f \bigr)_{L^2_+}=\f12\f{d}{dt}\|e^{\al\Psi} f\|^2_{L^2_+}-\al\int_{\R^2_+}\p_t\Psi |e^{\al\Psi}f|^2 \,dxdy.
\eeq
While due to $f\p_yf|_{y=0}=0,$ we get, by using integration by parts and then Young's inequality, that
\beq\begin{split}\label{S4eq7}
\bigl(-\be\p^2_yf\big|e^{2\al\Psi} f \bigr)_{L^2_+}&= \be\|e^{\al\Psi} \p_yf\|^2_{L^2_+}+\al\be\int_{\R^2_+}\p_y\Psi e^{2\al\Psi} f\p_yf \,dxdy\\
&\geq (\be-\f{\be^2\al}2)\|e^{\al\Psi} \p_yf\|^2_{L^2_+}-2\al\int_{\R^2_+}(\p_y\Psi)^2 |e^{\al\Psi}f|^2\, dxdy.
\end{split}\eeq
Thanks to \eqref{S3eq5}, summing up \eqref{S4eq6} and \eqref{S4eq7} leads to \eqref{S4eq5}.
This finishes the proof of Lemma \ref{lem:heateqs}.
\end{proof}

To handle  the non-linear terms in \eqref{eqs1}, \eqref{eqs2} and \eqref{eqs3}, we need the following lemmas:

\begin{lem}\label{property:Phi}
{\sl If $\de-\lam\tht(t)>0$, then $\|\phi(D_x) [fg]_\Phi \|_{L^2_\h}\leq \|\phi(D_x)(f_\Phi g_\Phi) \|_{L^2_h}$.
Similar inequality holds for $\Phi_\ka$ under the assumption that $\de-\lam\tht_\ka(t)>0.$}
\end{lem}
\begin{proof}
Indeed we deduce from \eqref{S3eq11} and Plancher\'el equality that
\beqo
\begin{split}
\|\phi(D_x) [fg]_\Phi \|_{L^2_\h}&=\|\phi(\xi) e^{\Phi(t,\xi)}(\hat{f}*\hat{g})\|_{L^2_\xi}\\
&\leq\|\phi(\xi)(e^\Phi \hat{f})*(e^\Phi \hat{g})\|_{L^2_\xi}=\|\phi(D_x)(f_\Phi g_\Phi) \|_{L^2_\h}.
\end{split}
\eeqo
This  concludes the proof of this lemma.
\end{proof}

\begin{lem}\label{lem:nl1}
{\sl Let $A, B$ and $E$ be smooth enough functions on $[0,T]\times\R^2_+$ with $A$ satisfying $\lim_{y\rto+\oo}A=0$.
Let $f(t)\eqdefa \tt^\f14 \|\eap \p_y\Ap(t)\|_{\cB^{\f12,0}}.$  Then, for any $a,b,c,d>0$
with $a+b\geq c$ and any $t_1<t_2\in [0,T],$
 we have
\beq\label{S4eq8}
\begin{split}
\int_{t_1}^{t_2}\bigl|\bigl(\ecp\Dhk [A\p_xB]_\Phi |& \edp \Dhk E_\Phi\bigr)_{L^2_+}\bigr|\,dt\\
\lesssim &
d_k^2 2^{-k}\|\ebp \Bp\|_{\wt{L}^2_{[t_1,t_2];f}(\cB^{1,0})}\|\edp E_\Phi\|_{\wt{L}^2_{[t_1,t_2];f}(\cB^{1,0})}.
\end{split}
\eeq}
\end{lem}
\begin{proof}
Applying Bony's decomposition \eqref{Bony} in the horizontal variable to $A\p_xB$ yields
\beqo
A\p_xB= T^{\h}_A\p_xB+T^{\h}_{\p_xB}A+R^{\h}(A,\p_xB)
\eeqo

By virtue of  Lemma \ref{property:Phi}, and considering the support properties to the Fourier transform of the terms in $T^{\h}_A\p_xB$, we find
\beqo\begin{split}
\|\ecp\Dhk [T^{\h}_A\p_xB]_\Phi\|_{L^2_+}\lesssim \sum_{|k'-k|\leq4}\|e^{(c-b)\Psi}\Sh_{k'-1}\Ap \|_{L^\oo_+} \|\ebp\Dh_{k'}\p_x\Bp\|_{L^2_+}.
\end{split}\eeqo
While it follows from $\lim_{y\rto+\oo}A=0$,  \eqref{ineq1} and $c-a-b\leq 0$ that
\beq\label{S4eq9}\begin{split}
\|e^{(c-b)\Psi}\Dhk\Ap \|_{L^\oo_\rmv( L^2_\h)}&=\bigl\|e^{(c-b)\Psi}\int_y^\oo \Dhk\p_y\Ap\bigr\|_{L^\oo_\rmv(L^2_\h)} \\
&\leq \bigl\|e^{(c-b)\Psi}\bigl( \int_y^\oo e^{-2a\Psi}\bigr)^\f12 \bigr\|_{L^\oo_\rmv} \|\eap\Dhk\p_y\Ap\|_{L^2_+} \\
 &\lesssim \tt^\f14 d_k(t) 2^{-\f{k}2}\|\eap\p_y\Ap \|_{\cB^{\f12,0}}.
\end{split}\eeq
Here and in all that follows, we always denote $\left(d_k(t)\right)_{k\in\Z}$ to be a generic element of $\ell^1(\Z)$
so that $\sum_{k\in\Z}d_k(t)=1.$

\eqref{S4eq9} together with Lemma \ref{lem:Bern} ensures that
\beqo\begin{split}
\|e^{(c-b)\Psi}\Sh_{k'-1}\Ap \|_{L^\oo_+} &\lesssim\sum_{k\leq k'-2} 2^\f{k}2 \|e^{(c-b)\Psi}\Dhk\Ap \|_{L^\oo_\rmv(L^2_\h)} \\
&\lesssim \tt^\f14 \|\eap\p_y\Ap \|_{\cB^{\f12,0}}.
\end{split}\eeqo
So that by applying Lemma \ref{lem:Bern} and Definition \ref{def2.3}, we deduce that
\beq\label{S4eq10}\begin{split}
\int_{t_1}^{t_2}&\bigl|\bigl(\ecp\Dhk [T^{\h}_A\p_xB]_\Phi | \edp \Dhk E_\Phi\bigr)_{L^2_+}\bigr|\,dt\\
&\lesssim \sum_{|k'-k|\leq 4} 2^{k'}\int_{t_1}^{t_2}f(t)\|\ebp\Dh_{k'} \Bp(t)\|_{L^2_+}\|\edp \Dhk E_\Phi(t)\|_{L^2_+}\,dt\\
&\lesssim \sum_{|k'-k|\leq 4} 2^{k'}\Bigl(\int_{t_1}^{t_2}f(t)\|\ebp\Dh_{k'}\Bp(t)\|_{L^2_+}^2\,dt\Bigr)^{\f12}
\Bigl(\int_{t_1}^{t_2}f(t)\|\edp \Dhk E_\Phi(t)\|_{L^2_+}\,dt\Bigr)^{\f12}\\
&\lesssim d_k^2 2^{-k}\|\ebp \Bp\|_{\wt{L}^2_{[t_1,t_2];f}(\cB^{1,0})}\|\edp E_\Phi\|_{\wt{L}^2_{[t_1,t_2];f}(\cB^{1,0})}.
\end{split}\eeq

Similarly, by Lemma \ref{property:Phi} and considering the support properties to the Fourier transform of the terms in $T^{\h}_{\p_xB}A$,
 we get
\beqo\begin{split}
\|\ecp\Dhk [T^{\h}_{\p_xB}A]_\Phi\|_{L^2_+}
\lesssim \sum_{|k'-k|\leq4}\int_{t_1}^{t_2}\|e^{(c-b)\Psi}\Dh_{k'}\Ap \|_{L^\oo_+} \|\ebp\Sh_{k'-1}\p_x\Bp\|_{L^2_+}\,dt,
\end{split}\eeqo
from which, \eqref{S4eq9}, we infer
\begin{align*}
\int_{t_1}^{t_2}&\bigl|\bigl(\ecp\Dhk [T^{\h}_{\p_xB}A]_\Phi | \edp \Dhk E_\Phi\bigr)_{L^2_+}\bigr|\,dt\\
&\lesssim \sum_{|k'-k|\leq 4}\int_{t_1}^{t_2}d_{k'}(t)f(t)\|\ebp \p_x\Bp\|_{L^2_+}\|\edp \Dhk E_\Phi\|_{L^2_+}\,dt\\
&\lesssim \sum_{|k'-k|\leq 4} d_{k'}\Bigl(\int_{t_1}^{t_2}f(t)\|\ebp \p_x\Bp(t)\|_{L^2_+}^2\,dt\Bigr)^{\f12}
\Bigl(\int_{t_1}^{t_2}f(t)\|\edp \Dhk E_\Phi(t)\|_{L^2_+}^2\,dt\Bigr)^{\f12}.
\end{align*}
Yet it follows from   Lemma \ref{lem:Bern} and Definition \ref{def2.3} that
\beq \label{S4eq11a}
\begin{split}
\Bigl(\int_{t_1}^{t_2}f(t)\|\ebp \p_x\Bp(t)\|_{L^2_+}^2\,dt\Bigr)^{\f12}\lesssim &\sum_{k\in\Z}2^{k}
\Bigl(\int_{t_1}^{t_2}f(t)\|\ebp \Dhk B_\Phi(t)\|_{L^2_+}^2\,dt\Bigr)^{\f12}\\
\lesssim& \|\ebp \Bp\|_{\wt{L}^2_{[t_1,t_2];f}(\cB^{1,0})}.
\end{split}
\eeq
As a result, we obtain
 \beq\label{S4eq11}\begin{split}
\int_{t_1}^{t_2}\bigl|\bigl(\ecp\Dhk [T^{\h}_{\p_xB}A]_\Phi | &\edp \Dhk E_\Phi\bigr)_{L^2_+}\bigr|\,dt\\
\lesssim& d_k^2 2^{-k}\|\ebp \Bp\|_{\wt{L}^2_{[t_1,t_2];f}(\cB^{1,0})}\|\edp E_\Phi\|_{\wt{L}^2_{[t_1,t_2];f}(\cB^{1,0})}.
\end{split}\eeq

Finally again due to Lemma \ref{property:Phi} and the support properties to the Fourier transform of the terms in $R^{\h}(A,\p_xB)$, we get, by applying Lemma \ref{lem:Bern}, that
\beqo\begin{split}
\|\ecp\Dhk [R^{\h}(A,\p_xB)]_\Phi\|_{L^2_+}
 &\lesssim 2^\f{k}2 \sum_{k'\geq k-3} \|e^{(c-b)\Psi}\Dh_{k'}\Ap \|_{L^\oo_\rmv(L^2_\h)}\|\ebp\wtDh_{k'}\p_x\Bp\|_{L^2_+},
\end{split}\eeqo
from which, \eqref{S4eq9} and Lemma \ref{lem:Bern}, we deduce that
\beq\label{S4eq12}\begin{split}
\int_{t_1}^{t_2}&\bigl|\bigl(\ecp\Dhk [R^{\h}(A,\p_xB)]_\Phi | \edp \Dhk E_\Phi\bigr)_{L^2_+}\bigr|\,dt\\
&\lesssim 2^\f{k}2 \sum_{k'\geq k-3} 2^{\f{k'}2}\int_{t_1}^{t_2} f(t)\|\ebp\wtDh_{k'}\Bp(t)\|_{L^2_+}\|\edp \Dhk E_\Phi\|_{L^2_+}\,dt\\
&\lesssim 2^\f{k}2 \sum_{k'\geq k-3} 2^{\f{k'}2}\Bigl(\int_{t_1}^{t_2} f(t)\|\ebp\wtDh_{k'}\Bp(t)\|_{L^2_+}^2\,dt\Bigr)^{\f12}
\Bigl(\int_{t_1}^{t_2} f(t)\|\edp \Dhk E_\Phi\|_{L^2_+}^2\,dt\Bigr)^{\f12}\\
&\lesssim d_k 2^{-\f{k}2}\|\ebp \Bp\|_{\wt{L}^2_{[t_1,t_2];f}(\cB^{1,0})}\|\edp E_\Phi\|_{\wt{L}^2_{[t_1,t_2];f}(\cB^{1,0})} \sum_{k'\geq k-3} d_{k'} 2^{-\f{k'}2}\\
&\lesssim d_k^2 2^{-k}\|\ebp \Bp\|_{\wt{L}^2_{[t_1,t_2];f}(\cB^{1,0})}\|\edp E_\Phi\|_{\wt{L}^2_{[t_1,t_2];f}(\cB^{1,0})}.
\end{split}\eeq

By summing up the estimates \eqref{S4eq10}, \eqref{S4eq11} and \eqref{S4eq12}, we conclude the proof of  \eqref{S4eq8}.
\end{proof}

\begin{lem}\label{lem:nl2}
{\sl Under the assumptions of Lemma \ref{lem:nl1}, one has
\beq\label{S4eq13}
\begin{split}
\int_{t_1}^{t_2}\bigl|\bigl(\|\ecp\Dhk [\p_yA\int_y^\oo \p_xB dy']_\Phi |& \edp \Dhk E_\Phi\bigr)_{L^2_+}\bigr|\,dt\\
\lesssim &
d_k^2 2^{-k}\|\ebp \Bp\|_{\wt{L}^2_{[t_1,t_2];f}(\cB^{1,0})}\|\edp E_\Phi\|_{\wt{L}^2_{[t_1,t_2];f}(\cB^{1,0})}.
\end{split}
\eeq}
\end{lem}
\begin{proof}
By applying Bony's decomposition \eqref{Bony} in the horizontal variable to $\p_yA\int_y^\oo \p_xB dy',$ we write
\beqo
\p_yA\int_y^\oo \p_xB dy'= T^{\h}_{\p_yA}\int_y^\oo \p_xB dy'+T^{\h}_{\int_y^\oo \p_xB dy'}\p_yA+R^{\h}(\p_yA,\int_y^\oo \p_xB dy')
\eeqo

Again thanks to Lemma \ref{property:Phi} and considering the support properties to the Fourier transform of the terms in $T^{\h}_{\p_yA}\int_y^\oo \p_xB dy'$, we find
\beqo\begin{split}
\|\ecp\Dhk [T^{\h}_{\p_yA}\int_y^\oo \p_xB dy']_\Phi\|_{L^2_+}\lesssim &\sum_{|k'-k|\leq4}
\|\eap\Sh_{k'-1}\p_y\Ap \|_{L^2_\rmv(L^\oo_\h) } \\
&\qquad\qquad\times\|e^{(c-a)\Psi}\int_y^\oo\Dh_{k'}\p_x\Bp dy'\|_{L^\oo_\rmv(L^2_\h) }.
\end{split}\eeqo
While it follows from \eqref{ineq1} and $c-a-b\leq 0$ that
\beq\label{S4eq14}\begin{split}
\|e^{(c-a)\Psi}\int_y^\oo\Dhk\p_x\Bp dy' \|_{L^\oo_\rmv(L^2_\h)}
&\leq\|e^{(c-a)\Psi}\bigl( \int_y^\oo e^{-2b\Psi}\bigr)^\f12 \|_{L^\oo_\rmv} \|\ebp\Dhk\p_x\Bp\|_{L^2_+} \\
&\lesssim  2^k\tt^\f14 \|e^{(c-a-b)\Psi}\|_{L^\oo_\rmv}  \|\ebp\Dhk\Bp\|_{L^2_+}\\
&\lesssim 2^k\tt^\f14  \|\ebp\Dhk\Bp\|_{L^2_+}.
\end{split}\eeq
As a result, it comes out
\beqo
\begin{split}
\int_{t_1}^{t_2}&\bigl|\bigl(\ecp\Dhk [T^{\h}_{\p_yA}\int_y^\oo \p_xB dy']_\Phi | \edp \Dhk E_\Phi\bigr)_{L^2_+}\bigr|\,dt\\
&\lesssim \sum_{|k'-k|\leq 4} 2^{k'}\int_{t_1}^{t_2}f(t)\|\ebp\D_{k'}^\h \Bp(t)\|_{L^2_+}\|\edp \Dhk E_\Phi(t)\|_{L^2_+}\,dt,
\end{split}
\eeqo
from which, we get, by a similar derivation of \eqref{S4eq10}, that
\beq\label{S4eq15}\begin{split}
\int_{t_1}^{t_2}\bigl|\bigl(&\ecp\Dhk [T^{\h}_{\p_yA}\int_y^\oo \p_xB dy']_\Phi | \edp \Dhk E_\Phi\bigr)_{L^2_+}\bigr|\,dt\\
&\lesssim d_k^2 2^{-k}\|\ebp \Bp\|_{\wt{L}^2_{[t_1,t_2];f}(\cB^{1,0})}\|\edp E_\Phi\|_{\wt{L}^2_{[t_1,t_2];f}(\cB^{1,0})}.
\end{split}\eeq

Along the same line, we get
\beqo\begin{split}
\int_{t_1}^{t_2}&\bigl|\bigl(\ecp\Dhk [T^{\h}_{\int_y^\oo \p_xB dy'}\p_yA]_\Phi | \edp \Dhk E_\Phi\bigr)_{L^2_+}\bigr|\,dt\\
\lesssim &\sum_{|k'-k|\leq4}\int_{t_1}^{t_2}\|\eap\Dh_{k'}\p_y\Ap \|_{L^2_\rmv(L^\oo_\h)} \\
&\qquad\qquad\times\|e^{(c-a)\Psi}\int_y^\oo \Sh_{k'-1}\p_x\Bp dy'\|_{L^\oo_\rmv(L^2_\h)}\|\edp \Dhk E_\Phi(t)\|_{L^2_+}\,dt\\
\lesssim &\sum_{|k'-k|\leq4}d_{k'}\Bigl(\int_{t_1}^{t_2}f(t)\w{t}^{-\f12} \|e^{(c-a)\Psi}\int_y^\oo \p_x\Bp dy'\|_{L^\oo_\rmv(L^2_\h)}^2\,dt\Bigr)^{\f12}\\
&\qquad\qquad\qquad\qquad\qquad\qquad\qquad\qquad\times\Bigl(\int_{t_1}^{t_2}f(t)\|\edp \Dhk E_\Phi(t)\|_{L^2_+}\,dt\Bigr)^{\f12}.
\end{split}\eeqo
Yet in view of \eqref{S4eq14}, we get, by a similar derivation of \eqref{S4eq11a}, that
\beno
\Bigl(\int_{t_1}^{t_2}f(t)\w{t}^{-\f12} \|e^{(c-a)\Psi}\int_y^\oo \p_x\Bp dy'\|_{L^\oo_\rmv(L^2_\h)}^2\,dt\Bigr)^{\f12}
\lesssim \|\ebp \Bp\|_{\wt{L}^2_{[t_1,t_2];f}(\cB^{1,0})}.
\eeno
Hence we obtain
\beq\label{S4eq16}\begin{split}
\int_{t_1}^{t_2}\bigl|\bigl(&\ecp\Dhk [T^{\h}_{\int_y^\oo \p_xB dy'}\p_yA]_\Phi | \edp \Dhk E_\Phi\bigr)_{L^2_+}\bigr|\,dt\\
&\lesssim d_k^2 2^{-k}\|\ebp \Bp\|_{\wt{L}^2_{[t_1,t_2];f}(\cB^{1,0})}\|\edp E_\Phi\|_{\wt{L}^2_{[t_1,t_2];f}(\cB^{1,0})}.
\end{split}\eeq

Finally, we get, by applying Lemma \ref{lem:Bern} and \eqref{S4eq14}, that
\beqo\begin{split}
&\int_{t_1}^{t_2}\bigl|\bigl(\ecp\Dhk [R^{\h}(\p_yA,\int_y^\oo \p_xB  dy')]_\Phi  | \edp \Dhk E_\Phi\bigr)_{L^2_+}\bigr|\,dt\\
&\lesssim 2^\f{k}2 \sum_{k'\geq k-3} \int_{t_1}^{t_2}\|\eap\Dh_{k'}\p_y\Ap \|_{L^2_+}\|e^{(c-a)\Psi}\int_y^\oo \wtDh_{k'}\p_x\Bp dy'\|_{L^\oo_\rmv(L^2_\h)}\|\edp \Dhk E_\Phi(t)\|_{L^2_+}\,dt\\
&\lesssim 2^\f{k}2 \sum_{k'\geq k-3} 2^{\f{k'}2}\int_{t_1}^{t_2} f(t)\|\ebp\wtDh_{k'}\Bp(t)\|_{L^2_+}\|\edp \Dhk E_\Phi\|_{L^2_+}\,dt,
\end{split}\eeqo
which together with a similar derivation of \eqref{S4eq12} gives rise to
\beqo
\begin{split}
\int_{t_1}^{t_2}\bigl|\bigl(\ecp\Dhk [R^{\h}(&\p_yA,\int_y^\oo \p_xB  dy')]_\Phi  | \edp \Dhk E_\Phi\bigr)_{L^2_+}\bigr|\,dt \\
&\lesssim d_k^2 2^{-k}\|\ebp \Bp\|_{\wt{L}^2_{[t_1,t_2];f}(\cB^{1,0})}\|\edp E_\Phi\|_{\wt{L}^2_{[t_1,t_2];f}(\cB^{1,0})}.
\end{split}\eeqo
Together with \eqref{S4eq15}, \eqref{S4eq16}, we complete the proof of \eqref{S4eq13}.
\end{proof}

\begin{lem}\label{lem:nl3}
{\sl Under the assumptions of Lemma \ref{lem:nl1}, if we assume moreover that  $a+b>c.$ Then  we have
\beq\label{S4eq18}
\begin{split}
\int_{t_1}^{t_2}\bigl|\bigl(\ecp\int_y^\oo \Dhk [\p_yA\p_xB]_\Phi dy' |& \edp \Dhk E_\Phi\bigr)_{L^2_+}\bigr|\,dt\\
\lesssim &
d_k^2 2^{-k}\|\ebp \Bp\|_{\wt{L}^2_{[t_1,t_2];f}(\cB^{1,0})}\|\edp E_\Phi\|_{\wt{L}^2_{[t_1,t_2];f}(\cB^{1,0})}.
\end{split}
\eeq}
\end{lem}
\begin{proof}
Applying Bony's decomposition \eqref{Bony} in the horizontal variable to $\p_yA\p_xB$ yields
\beqo
\p_yA\p_xB= T^{\h}_{\p_yA}\p_xB+T^{\h}_{\p_xB}\p_yA+R^{\h}(\p_yA,\p_xB)
\eeqo

We first observe that
\beqo\begin{split}
\|\ecp\int_y^\oo& \Dhk[T^{\h}_{\p_yA}  \p_xB]_\Phi dy'\|_{L^2_+}\\
&\lesssim
\|e^{(c-a-b)\Psi} \|_{L^2_\rmv}\sum_{|k'-k|\leq4}\|\eap\Sh_{k'-1}\p_y\Ap \|_{L^2_\rmv(L^\oo_\h) } \|\ebp\Dh_{k'}\p_x\Bp \|_{L^2_+}.
\end{split}\eeqo
Due to $c-a-b< 0,$ one has
\beq\label{S4eq19}
\|e^{(c-a-b)\Psi}\|_{L^2_\rmv}\lesssim \tt^\f14,
\eeq
from which and Lemma \ref{lem:Bern}, we infer
\beqo\begin{split}
\|\ecp\int_y^\oo \Dhk[T^{\h}_{\p_yA}\p_xB]_\Phi  dy'\|_{L^2_+}\lesssim \sum_{|k'-k|\leq4}2^{k'}f(t) \|\ebp\Dh_{k'}\Bp \|_{L^2_+}.
\end{split}\eeqo
Then we get, by a similar derivation of \eqref{S4eq10}, that
\beq\label{S4eq20}\begin{split}
\int_{t_1}^{t_2}&\bigl|\bigl( \ecp\int_y^\oo \Dhk[T^{\h}_{\p_yA}\p_xB]_\Phi  dy | \edp \Dhk E_\Phi\bigr)_{L^2_+}\bigr|\,dt\\
&\lesssim d_k^2 2^{-k}\|\ebp \Bp\|_{\wt{L}^2_{[t_1,t_2];f}(\cB^{1,0})}\|\edp E_\Phi\|_{\wt{L}^2_{[t_1,t_2];f}(\cB^{1,0})}.
\end{split}\eeq

Along the same line, by applying Lemma \ref{lem:Bern} and \eqref{S4eq19}, we find
\beqo\begin{split}
\|\ecp\int_y^\oo \Dhk & [T^{\h}_{\p_xB}\p_yA]_\Phi dy'\|_{L^2_+}\\
&\lesssim \|e^{(c-a-b)\Psi}\|_{L^2_\rmv}
 \sum_{|k'-k|\leq4}\|\eap\Dh_{k'}\p_y\Ap \|_{L^2_\rmv(L^\oo_\h)} \|\ebp \Sh_{k'-1}\p_x\Bp\|_{L^2_+}\\
 &\lesssim \sum_{|k'-k|\leq4} d_{k'}(t) f(t)\|\ebp \p_x\Bp\|_{L^2_+}.
\end{split}\eeqo
Then thanks to \eqref{S4eq11a}, we get, by applying a similar derivation of \eqref{S4eq11}, that
\beq\label{S4eq21}\begin{split}
\int_{t_1}^{t_2}&\bigl|\bigl(\ecp\int_y^\oo \Dhk  [T^{\h}_{\p_xB}\p_yA]_\Phi dy' | \edp \Dhk E_\Phi\bigr)_{L^2_+}\bigr|\,dt\\
&\lesssim d_k^2 2^{-k}\|\ebp \Bp\|_{\wt{L}^2_{[t_1,t_2];f}(\cB^{1,0})}\|\edp E_\Phi\|_{\wt{L}^2_{[t_1,t_2];f}(\cB^{1,0})}.
\end{split}\eeq

Finally applying  Lemma \ref{lem:Bern} and \eqref{S4eq19} gives rise to
\beqo\begin{split}
\|\ecp\int_y^\oo &\Dhk[R^{\h}(\p_yA, \p_xB)]_\Phi dy'\|_{L^2_+}\\
&\lesssim 2^\f{k}2 \|e^{(c-a-b)\Psi}\|_{L^2_\rmv} \sum_{k'\geq k-3} \|\eap\Dh_{k'}\p_y\Ap \|_{L^2_+}\|\ebp\wtDh_{k'}\p_x\Bp\|_{L^2_+}
\\ &\lesssim 2^\f{k}2  \sum_{k'\geq k-3} 2^{\f{k'}2}f(t)\|\ebp\wtDh_{k'}\Bp\|_{L^2_+},
\end{split}\eeqo
Then a similar derivation of \eqref{S4eq12} leads to
\beq\label{S4eq22}
\begin{split}
\int_{t_1}^{t_2}&\bigl|\bigl(\ecp\int_y^\oo \Dhk[R^{\h}(\p_yA, \p_xB)]_\Phi dy' | \edp \Dhk E_\Phi\bigr)_{L^2_+}\bigr|\,dt\\
&\lesssim d_k^2 2^{-k}\|\ebp \Bp\|_{\wt{L}^2_{[t_1,t_2];f}(\cB^{1,0})}\|\edp E_\Phi\|_{\wt{L}^2_{[t_1,t_2];f}(\cB^{1,0})}.
\end{split}\eeq

By summing up \eqref{S4eq20}, \eqref{S4eq21} and \eqref{S4eq22}, we finish the proof of  \eqref{S4eq18}.
\end{proof}

To handle the terms in \eqref{eqs1}, \eqref{eqs2} and \eqref{eqs3} involving the far field state $(U,B),$
we need the following lemmas:

\begin{lem}
{\sl Let $(m_U,m_B)$ and $(M_U,M_B)$ be determined respectively by \eqref{def:m} and \eqref{eqs2}. Then
under the assumption of  \eqref{UBdecay}, for any $T<T^*$, one has
\beq\label{mdecay}
\|\tt^\f74 \ep (m_U,m_B,M_U,M_B)_\Phi \|_{\wt{L}^2_{T}(\cB^{\f12,0})}\leq C\e
\eeq
Similar estimate holds with $\Phi$, $\Psi$ and $T^*$ being replaced respectively by $\Phi_\ka$, $\Psi_\ka$ and $T^*_\ka$.
}\end{lem}
\begin{proof}
In view of  \eqref{def:m} and the construction of $\chi$, we observe that $(m_U,m_B,M_U,M_B)$ are supported in $[0,2]$
for each fixed $t>0.$ Hence we get
\beqo\begin{split}
\| \tt^\f74 \ep (m_U,m_B &,M_U,M_B)_\Phi \|_{\wt{L}^2_{T}(\cB^{\f12,0})}\\ &\lesssim \|\tt^\f74(\p_t U,\p_t B,U,B,U\p_xU,B\p_xB,U\p_xB,B\p_xU)_\Phi \|_{\wt{L}^2_{T}(\cB^{\f12}_\h)}.
\end{split}
\eeqo
Here and all in that follows, we always denote $\cB^s_\h$ to the Besov space $B^s_{2,1}(\R_\h).$

Then thanks to  the following law of product in  Besov spaces (see \cite{BCD} for instance)
\beq\label{s8eq2}\begin{split}
\|fg\|_{\wt{L}^2_{T}(\cB^{\f12}_\h)}
& \lesssim \|f\|_{\wt{L}^\oo_{T}(\cB^{\f12}_\h)} \|g\|_{\wt{L}^2_{T}(\cB^{\f12}_\h)},
\end{split}
\eeq
 and  Lemma \ref{property:Phi} and \eqref{UBdecay}, we deduce that
\beqo
\begin{split}
\|\tt^\f74(U\p_xU,B\p_xB,U\p_xB,B\p_xU)_\Phi \|_{\wt{L}^2_{T}(\cB^{\f12}_\h)}&\lesssim\|\tt^\f74\UBp\|_{\wt{L}^2_{T}(\cB^{\f12})}\|\UBp\|_{\wt{L}^\oo_{T}(\cB^{\f32}_\h)}\\
 & \lesssim\e\|\tt^\f74\UBp\|_{\wt{L}^2_{T}(\cB^{\f12}_\h)}.
\end{split}
\eeqo
As a result, it comes out
$$
\|\tt^\f74 \ep (m_U,m_B,M_U,M_B)_\Phi \|_{\wt{L}^2_{T}(\cB^{\f12,0})}\lesssim \|\tt^\f74 e^{\delta|D_x|}(\p_t U,\p_t B,U,B)_\Phi \|_{\wt{L}^2_{T}(\cB^{\f12}_\h)},
$$
which together with \eqref{UBdecay} ensures \eqref{mdecay}.
\end{proof}

\begin{lem}\label{ff1}
{\sl Let $F$ be smooth enough function on $[0,T]\times\R_{\rm h}$ and $A, E$ be smooth enough function on $[0,T]\times\R^2_+$.
We denote $g(t)\eqdefa \|F_\Phi(t) \|_{\cB^\f12_\h}.$  Then one has
\beq
\label{s8eq6}
\begin{split}
\int_{t_1}^{t_2}\bigl|\bigl(\epk \Dhk[F\p_xA]_\Phi &| \epk \Dhk E_\Phi\bigr)_{L^2_+}\bigr|\,dt\\
&\lesssim
d_k^2 2^{-k}\|\epk \Ap\|_{\wt{L}^2_{[t_1,t_2];g}(\cB^{1,0})}\|\epk E_\Phi\|_{\wt{L}^2_{[t_1,t_2];g}(\cB^{1,0})}.
\end{split}
\eeq
}\end{lem}

\begin{proof}
By applying Bony's decomposition \eqref{Bony}, we write
$$
F\p_xA=T^{\rm h}_F \p_xA+T^{\rm h}_{\p_xA}F+R^{\rm h}(F,\p_xA).
$$
We first observe that
\beqo\begin{split}
\int_{t_1}^{t_2}&\bigl|\bigl(\epk \Dhk[T^{\rm h}_F \p_xA]_\Phi | \epk \Dhk E_\Phi\bigr)_{L^2_+}\bigr|\,dt\\
&\lesssim \sum_{|k'-k|\leq 4} \int_{t_1}^{t_2}\|\Sh_{k'-1}F_\Phi\|_{L^\oo_{\rm h}} \| \epk\Dh_{k'} \p_xA_\Phi\|_{L^2_+} \| \epk\Dh_{k} E_\Phi\|_{L^2_+}\,dt\\
&\lesssim \sum_{|k'-k|\leq 4} 2^{k'}\int_{t_1}^{t_2}g(t)\| \epk\Dh_{k'} A_\Phi\|_{L^2_+} \| \epk\Dh_{k} E_\Phi\|_{L^2_+}\,dt.
\end{split}
\eeqo
Then a similar derivation of \eqref{S4eq10} gives rise to
\beqo\begin{split}
\int_{t_1}^{t_2}&\bigl|\bigl(\epk \Dhk [T^{\rm h}_F \p_xA]_\Phi | \epk \Dhk E_\Phi\bigr)_{L^2_+}\bigr|\,dt
\lesssim d_k^2 2^{-k}\|\ebp \Ap\|_{\wt{L}^2_{[t_1,t_2];g}(\cB^{1,0})}\|\ebp E_\Phi\|_{\wt{L}^2_{[t_1,t_2];g}(\cB^{1,0})}.
\end{split}\eeqo

Notice that
\beqo
\begin{split}
\|\epk \Dhk[T^{\rm h}_{\p_xA}F]_\Phi\|_{L^2_+}
&\lesssim \sum_{|k'-k|\leq 4} \|\Dh_{k'}F_\Phi\|_{L^\oo_{\rm h}} \|\epk \Sh_{k'-1} \p_xA_\Phi\|_{L^2_+}\\
&\lesssim \sum_{|k'-k|\leq 4} d_{k'}(t) g(t) \|\epk  \p_xA_\Phi\|_{L^2_+},
\end{split}
\eeqo
we get, by a similar derivation of \eqref{S4eq11}, that
 \beqo\begin{split}
\int_{t_1}^{t_2}&\bigl|\bigl(\epk\Dhk [T^{\rm h}_{\p_xA}F]_\Phi | \epk \Dhk E_\Phi\bigr)_{L^2_+}\bigr|\,dt
\lesssim d_k^2 2^{-k}\|\epk \Ap\|_{\wt{L}^2_{[t_1,t_2];g}(\cB^{1,0})}\|\epk E_\Phi\|_{\wt{L}^2_{[t_1,t_2];g}(\cB^{1,0})}.
\end{split}\eeqo

Whereas it follows from Lemma \ref{lem:Bern} that
\beqo\begin{split}
\|\epk \Dhk[R^{\rm h}(F,\p_xA)]_\Phi\|_{L^2_+}
&\lesssim 2^\f{k}2 \sum_{k'\geq k-3} \|\wtDh_{k'}F_\Phi\|_{L^2_{\rm h}} \|\epk \Dh_{k'} \p_xA_\Phi\|_{L^2_+}\\
&\lesssim 2^\f{k}2 \sum_{k'\geq k-3} 2^{-\f{k'}2} g(t)\|\epk \Dh_{k'} \p_x A_\Phi\|_{L^2_+},
\end{split}
\eeqo
from which and a similar derivation of \eqref{S4eq12}, we infer
\beqo\begin{split}
\int_{t_1}^{t_2}&\bigl|\bigl(\epk\Dhk [R^{\h}(F,\p_xA)]_\Phi | \epk \Dhk E_\Phi\bigr)_{L^2_+}\bigr|\,dt\\
&\lesssim d_k^2 2^{-k}\|\epk \Ap\|_{\wt{L}^2_{[t_1,t_2];g}(\cB^{1,0})}\|\epk E_\Phi\|_{\wt{L}^2_{[t_1,t_2];g}(\cB^{1,0})}.
\end{split}\eeqo
This proves \eqref{s8eq6}.
\end{proof}

\begin{rmk}
It follows from a similar proof of \eqref{s8eq6} that
\beq
\label{s8eq10}
\begin{split}
\int_{t_1}^{t_2}\bigl|\bigl(\epk& \Dhk[\p_xFA]_\Phi  | \epk \Dhk E_\Phi\bigr)_{L^2_+}\bigr|\,dt\\
&\lesssim
d_k^2 2^{-{k}}\|F_\Phi\|_{\wt{L}^\infty_{[t_1,t_2]}(\cB^{\f32}_\h)}\|\epk \Ap\|_{\wt{L}^2_{[t_1,t_2]}(\cB^{\f12,0})}\|\epk E_\Phi\|_{\wt{L}^2_{[t_1,t_2]}(\cB^{\f12,0})}.
\end{split} \eeq
\end{rmk}

\renewcommand{\theequation}{\thesection.\arabic{equation}}
\setcounter{equation}{0}
\section{Analytic energy estimate to the primitive functions $(\vphi,\psi)$}\label{sect5}

The goal of this section is to present the {\it a priori} weighted analytic energy estimate to the primitive functions $(\vphi,\psi)$ which solves \eqref{eqs2}. Namely, we shall present the proof of Propositions \ref{prop3.1} and \ref{prop3.4}. The key ingredients will be the following two lemmas:

\begin{lemma}\label{prop5.1}
{\sl Let $(\vphi,\psi)$ be a smooth enough solution of \eqref{eqs2}. Let $\Phi(t,\xi)$ and $\Psi(t,y)$ be given by \eqref{def:Phi} and \eqref{def:Psi} respectively. Then if $\ka\in ]0,2[,$ for $\lk\eqdefa \f{\ka(2-\ka)}4$ and for any non-negative and non-decreasing function $\hs\in C^1(\R_+)$, there exists a positive constant $\fc$ so that
\beq\begin{split}\label{S5eq1}
&\| \hsf \ep\Dhk \ppp\|^2_{L^\oo_{[t_0,t]}(L^2_+)} +4\lk\| \hsf \ep\Dhk\p_y \ppp\|^2_{L^2_{[t_0,t]}(L^2_+)}\\
& +2c\lam 2^k \|\hsf\ep\Dhk\ppp \|^2_{L^2_{[t_0,t];\dtht}(L^2_+)}
\leq \| \hsf \ep\Dhk \ppp(t_0)\|^2_{L^2_+}\\ &+\| \qhs \ep\Dhk \ppp\|^2_{L^2_{[t_0,t]}(L^2_+)}+\frak{c}\e d^2_k 2^{-k}\|\hsf\ep \p_y\ppp \|^2_{\wt{L}^2_{[t_0,t]}(\cB^{\f12,0})}\\
 &+C d^2_k 2^{-k}\Bigl(\e^{-1}\|\ttau^\f12 \hsf \ep \Mp \|^2_{\wt{L}^2_{[t_0,t]}(\cB^{\f12,0})}+ \|\hsf\ep\ppp \|^2_{\wt{L}^2_{[t_0,t];\dtht}(\cB^{1,0})}\Bigr),
\end{split}\eeq
for any $t_0\in[0,t]$ with $t<T^*$, which is defined by \eqref{def:T^*}.}
\end{lemma}
\begin{proof}
In view of \eqref{S3eq1}, by applying operator $e^{\Phi(t,|D_x|)}$ to \eqref{eqs2}, we write
\beq\label{S5eq2}\begin{split}
\p_t \vp &-\p^2_y\vp-\bar{B}_\ka\p_x\pp+\lam\dtht(t)|D_x| \vp+ [u\p_x \vphi-b\p_x\psi]_\Phi\\
&+2\int_y^\infty [\p_x\vphi\p_yu-\p_x\psi\p_y b]_\Phi dy'+\chi'[U\p_x\vphi-B\p_x\psi]_\Phi\\
&+2\int_y^\infty \chi''[U\p_x\vphi-B\p_x\psi]_\Phi dy'+\chi[-\p_xUu+\p_xBb]_\Phi\\
&+2\chi'[\p_xU\vphi-\p_xB\psi]_\Phi+2\int_y^\infty \chi''[\p_xU\vphi-\p_xB\psi]_\Phi dy'=(M_U)_\Phi,
\end{split} \eeq
and
\beq\label{S5eq3}\begin{split}
&\p_t\pp-\ka\p^2_y\pp-\bar{B}_\ka\p_x\vp+\lam\dtht(t)|D_x|\pp+[u\p_x\psi-b\p_x\vphi]_\Phi \\ &\qquad+\chi'[U\p_x\psi-B\p_x\vphi]_\Phi+\chi[\p_xBu-\p_xUb]_\Phi=(M_B)_\Phi.
\end{split}\eeq

It is easy to observe  that the terms involving the constant $\bar{B}_\ka$ in \eqref{S5eq1} and \eqref{S5eq2}
will be canceled when performing energy estimate
\beq \label{symmetry}\begin{split}
\bar{B}_\ka (\ep\Dhk\p_x\pp\big|\ep\Dhk\vp \bigr)_{L^2_+}+&\bar{B}_\ka\bigl(\ep\Dhk\p_x\vp\big|\ep\Dhk\pp \bigr)_{L^2_+} \\
=&\bar{B}_\ka \int_{\R^2_+} \p_x\bigl( e^{2\Phi} \Dh_k\vp \Dh_k \pp \bigr)dxdy =0
\end{split}
  \eeq

Thanks to \eqref{symmetry}, by applying the dyadic operator $\Dhk$ to \eqref{S5eq2}, \eqref{S5eq3} and then taking the $L^2_+$ inner product of the resulting equations with   $\epp\Dhk\ppp$ respectively, we find
\beq
\begin{split}\label{S5eq4}
&\bigl(\ep\Dhk(\p_t-\p^2_y)\vp\big|\ep\Dhk\vp \bigr)_{L^2_+}+\bigl(\ep\Dhk(\p_t-\ka\p^2_y)\pp\big|\ep\Dhk\pp \bigr)_{L^2_+}\\
&
+\lam\dtht(t)\bigl(\ep|D_x|\Dhk\vp\big|\ep\Dhk\vp \bigr)_{L^2_+}+\lam\dtht(t)\bigl(\ep|D_x|\Dhk\pp\big|\ep\Dhk\pp \bigr)_{L^2_+} \\
&
+\bigl(\ep\Dhk[u\p_x\vphi-b\p_x\psi]_\Phi\big|\ep\Dhk\vp \bigr)_{L^2_+}+\bigl(\ep\Dhk[u\p_x\psi-b\p_x\vphi]_\Phi\big|\ep\Dhk\pp \bigr)_{L^2_+} \\
 &
+2\bigl(\ep\int_y^\oo \Dhk[\p_yu\p_x\vphi-\p_yb\p_x\psi]_\Phi dy'\big|\ep\Dhk\vp \bigr)_{L^2_+}\\
 &
+\bigl(\ep\Dhk[U\p_x\vphi-B\p_x\psi]_\Phi\big|\ep\chi'\Dhk\vp\bigr)_{L^2_+}+\bigl(\ep\Dhk[U\p_x\psi-B\p_x\vphi]_\Phi\big|\ep\chi'\Dhk\pp\bigr)_{L^2_+}\\ &
+2\bigl(\ep\int_y^\infty \chi'' \Dhk[U\p_x\vphi-B\p_x\psi]_\Phi dy'\big|\ep\Dhk\vp\bigr)_{L^2_+}\\
&
+\bigl(\ep\Dhk[-\p_xUu+\p_xBb]_\Phi\big|\ep\chi\Dhk\vp\bigr)_{L^2_+}+\bigl(\ep\Dhk[-\p_xUb+\p_xBu]_\Phi\big|\ep\chi\Dhk\pp\bigr)_{L^2_+}\\
&
+2\bigl(\ep\Dhk[\p_xU\vphi-\p_xB\psi]_\Phi\big|\ep\chi'\Dhk\vp\bigr)_{L^2_+}\\
&
+2\bigl(\ep\int_y^\infty \chi''\Dhk[\p_xU\vphi-\p_xB\psi]_\Phi dy'\big|\ep\Dhk\vp\bigr)_{L^2_+}\\ &
=\bigl(\ep\Dhk(M_U)_\Phi\big|\ep\Dhk\vp\bigr)_{L^2_+}+\bigl(\ep\Dhk(M_B)_\Phi\big|\ep\Dhk\pp\bigr)_{L^2_+}.
\end{split}\eeq
In what follows, we shall  use the technical lemmas in Section \ref{sect4} to handle term by term in \eqref{S5eq4}.

Notice that for $\ka\in ]0,2[,$  $\ka(2-\ka)\leq1$, we get, by applying \eqref{S4eq1} and \eqref{S4eq2}, that
\beqo\begin{split}
\int_{t_0}^t&\hbar(t')\Bigl(\bigl(\ep\Dhk(\p_t-\p^2_y)\vp\big|\ep\Dhk\vp \bigr)_{L^2_+}+\bigl(\ep\Dhk(\p_t-\ka\p^2_y)\pp\big|\ep\Dhk\pp \bigr)_{L^2_+}\Bigr)\,dt'\\
&\geq \ \f12\bigl(\|\hsf \ep\Dhk\ppp(t)\|^2_{L^2_+}-\|\hsf \ep\Dhk\ppp(t_0)\|^2_{L^2_+}\bigr) \\
&\qquad-\f12 \|\sqrt{\hbar'}\ep\Dhk\ppp\|^2_{L^2_{[t_0,t]}(L^2_+)}
+2\lk\|\hsf \ep\Dhk\p_y\ppp\|^2_{L^2_{[t_0,t]}(L^2_+)}.
\end{split}\eeqo

In view of Lemma \ref{lem:Bern}, we have
\beqo\begin{split}
&\lam\int_{t_0}^t\hbar(t')\dtht(t')\Bigl(\bigl(\ep|D_x|\Dhk\vp\big|\ep\Dhk\vp \bigr)_{L^2_+}+\bigl(\ep|D_x|\Dhk\pp\big|\ep\Dhk\pp \bigr)_{L^2_+}\Bigr)\,dt'\\
&=\ \lam\int_{t_0}^t\dtht(t')\|\hsf\ep|D_x|^\f12\Dhk\ppp\|^2_{L^2_+}\,dt'
\geq c\lam 2^k\int_{t_0}^t\dtht(t')\|\hsf\ep\Dhk\ppp \|^2_{L^2_+}\,dt'.
\end{split}\eeqo

Whereas due to $\lim_{y\rto+\oo}u=\lim_{y\rto+\oo}b=0$, we get, by applying Lemma \ref{lem:nl1} with $a=\f12$ and $b=c=d=1$, that
\beqo\begin{split}
&\int_{t_0}^t\hbar(t')\Bigl(\bigl|\bigl(\ep\Dhk[u\p_x\vphi-b\p_x\psi]_\Phi\big|\ep\Dhk\vp \bigr)_{L^2_+}\bigr|+\bigl|\bigl(\ep\Dhk[u\p_x\psi-b\p_x\vphi]_\Phi\big|\ep\Dhk\pp \bigr)_{L^2_+}\bigr|\Bigr)\,dt'\\
&\qquad\qquad\lesssim d_k^22^{-k}\|\hsf\ep\ppp\|_{\wt{L}^2_{[t_0,t];f}(\cB^{1,0})}^2 \with f(t)=\w{t}^{\f14}\|e^{\f{\Psi}2}\p_y(u,b)_\Phi(t)\|_{\cB^{\f12,0}}.
\end{split}\eeqo
Yet it follows from Lemma \ref{lem3.1} that
\beq\label{S5eq5}
f(t)\lesssim \w{t}^{\f14}\|\ep\p_y(G,H)_\Phi(t)\|_{\cB^{\f12,0}},
\eeq
Hence according to Definition \ref{def2.3} and \eqref{def:theta}, we achieve
\beq\begin{split}\label{S5eq7}
\int_{t_0}^t&\hbar(t')\Bigl(\bigl|\bigl(\ep\Dhk[u\p_x\vphi-b\p_x\psi]_\Phi\big|\ep\Dhk\vp \bigr)_{L^2_+}\bigr|\\
&+\bigl|\bigl(\ep\Dhk[u\p_x\psi-b\p_x\vphi]_\Phi\big|\ep\Dhk\pp \bigr)_{L^2_+}\bigr|\Bigr)\,dt'\lesssim d_k^22^{-k}\|\hsf\ep\ppp\|_{\wt{L}^2_{[t_0,t];\dtht}(\cB^{1,0})}^2.
\end{split}\eeq
Similarly, by applying Lemma \ref{lem:nl3}, we get, by a similar derivation of \eqref{S5eq7}, that
\beqo\begin{split}\label{S5eq8}
\int_{t_0}^t\hbar(t')\bigl|\bigl(\ep\int_y^\oo \Dhk[\p_yu\p_x\vphi-&\p_yb\p_x\psi]_\Phi dy'\big|\ep\Dhk\vp \bigr)_{L^2_+}\bigr|\,dt'\\
&\lesssim d_k^22^{-k}\|\hsf\ep\ppp\|_{\wt{L}^2_{[t_0,t];\dtht}(\cB^{1,0})}^2.
\end{split}\eeqo

While it follows from  Lemma \ref{ff1} that
\beqo \begin{split}
\int_{t_0}^t&\hbar(t')\Bigl(\bigl|\bigl(\ep\Dhk[U\p_x\vphi-B\p_x\psi]_\Phi\big|\ep\chi'\Dhk\vp\bigr)_{L^2_+}\bigr|\\
&
+\bigl|\bigl(\ep\Dhk[U\p_x\psi-B\p_x\vphi]_\Phi\big|\ep\chi'\Dhk\pp\bigr)_{L^2_+}\bigr|\Bigr)\,dt'
\lesssim \e^{\f12}d_k^22^{-k}\|\hsf\ep\ppp\|_{\wt{L}^2_{[t_0,t];\dtht}(\cB^{1,0})}^2.
\end{split}\eeqo

Notice from the definition of $\chi$ above \eqref{change} that $\mbox{supp}\chi''\subset [1,2]$, we
get, by a similar derivation of \eqref{s8eq6}, that
\beqo \begin{split}
\int_{t_0}^t\hbar(t')\bigl|\bigl(\ep\int_y^\infty \chi'' \Dhk[U\p_x\vphi-B\p_x\psi]_\Phi dy'\big|&\ep\Dhk\vp\bigr)_{L^2_+}\bigr|\,dt'\\
\lesssim & \e^{\f12}d_k^22^{-k}\|\hsf\ep\ppp\|_{\wt{L}^2_{[t_0,t];\dtht}(\cB^{1,0})}^2.
\end{split}\eeqo

On the other hand, it follows from  \eqref{S8eq34aq}  that
\beno
\|\ep y\Dhk\ppp\|_{L^2_+}\lesssim \w{t}\|\ep\p_y\Dhk\ppp\|_{L^2_+},
\eeno
which implies
\beno
\|\ep y\ppp\|_{\wt{L}^p_t(\cB^{s,0})}\lesssim \|\ttau\ep\p_y\ppp\|_{\wt{L}^p_t(\cB^{s,0})}.
\eeno
So that
 by applying \eqref{s8eq10} and  \eqref{UBdecay}, we find
\beqo \begin{split}
\int_{t_0}^t&\hbar(t')\Bigl(\bigl|\bigl(\ep\Dhk[-\p_xUu+\p_xBb]_\Phi\big|\ep\chi\Dhk\vp\bigr)_{L^2_+}\bigr|\\
&\qquad+\bigl|\bigl(\ep\Dhk[-\p_xUb+\p_xBu]_\Phi\big|\ep\chi\Dhk\pp\bigr)_{L^2_+}\bigr|\Bigr)\,dt\\
\lesssim \ & d_k^2 2^{-{k}}\|\ttau\UBp\|_{\wt{L}^\oo_{[t_0,t]}(\cB^{\f32}_\h)}\|\hsf\ep\ubp\|_{\wt{L}^2_{[t_0,t]}(\cB^{\f12,0})}\|\w{\tau}^{-1}\hsf\ep y\ppp\|_{\wt{L}^2_{[t_0,t]}(\cB^{\f12,0})}\\
\lesssim \ &\e d^2_k 2^{-k}\|\hsf\ep\p_y\ppp\|^2_{\wt{L}^2_{[t_0,t]}(\cB^{\f12,0})}.
\end{split}\eeqo
Along the same line, we infer
\beqo \begin{split}
\int_{t_0}^t&\hbar(t')\bigl|\bigl(\ep\Dhk[\p_xU\vphi-\p_xB\psi]_\Phi\big|\ep\chi'\Dhk\vp\bigr)_{L^2_+}\bigr|\,dt'\\
\lesssim \ & d_k^22^{-{k}}\|\ttau\UBp \|_{\wt{L}^\oo_{[t_0,t]}(\cB^{\f32}_\h)}\|\ttau^{-\f12}\hsf\ep\ppp \|_{\wt{L}^2_{[t_0,t]}(\cB^{\f12,0})}
\|\ttau^{-\f12}\hsf\ep \vp\|_{\wt{L}^2_{[t_0,t]}(\cB^{\f12,0})}\\
\lesssim \ &\e d^2_k 2^{-k}\|\hsf\ep\p_y\ppp\|^2_{\wt{L}^2_{[t_0,t]}(\cB^{\f12,0})}.
\end{split}\eeqo
Moreover, again due to $\mbox{supp}\chi''\subset[1,2]$,, we deduce that the term
$
\int_{t_0}^t\hbar(t')\bigl|\bigl(\ep\int_y^\infty \chi''\Dhk[\p_xU\vphi-\p_xB\psi]_\Phi dy'\big|\ep\Dhk\vp\bigr)_{L^2_+}\bigr|\,dt'$ shares
the same estimate as above.

Finally for the remained source term, by applying Young's inequality, we achieve
\beqo \label{S5eq9}\begin{split}
\int_{t_0}^t&\hbar(t')\Bigl(\bigl|\bigl(\ep\Dhk(M_U)_\Phi\big|\ep\Dhk\vp\bigr)_{L^2_+}\bigr|
+\bigl|\bigl(\ep\Dhk(M_B)_\Phi\big|\ep\Dhk\pp\bigr)_{L^2_+}\bigr|\Bigr)\,dt'\\
\lesssim \ & d^2_k2^{-k} \|\ttau^{\f12}\hsf\ep\Mp\|_{\wt{L}^2_{[t_0,t]}(\cB^{\f12,0})}\|\ttau^{-\f12}\hsf\ep\ppp\|_{\wt{L}^2_{[t_0,t]}(\cB^{\f12,0})}\\
\lesssim \ & d^2_k2^{-k}\bigl( \e \|\hsf\ep\p_y\ppp\|^2_{\wt{L}^2_{[t_0,t]}(\cB^{\f12,0})}+\e^{-1} \|\ttau^\f12\hsf\ep\Mp\|^2_{\wt{L}^2_{[t_0,t]}(\cB^{\f12,0})}\bigr)
\end{split}\eeqo

By multiplying \eqref{S5eq4} by $\hbar(t)$ and then integrating the resulting equality over $[t_0,t],$
and finally inserting the above estimates to the equality, we
 obtain \eqref{S5eq1}. This completes the proof of Lemma \ref{prop5.1}.
\end{proof}

\begin{lemma}\label{prop5.2}
{\sl Let $(\vphi,\psi)$ be a smooth enough solution of \eqref{eqs2}. Let $\Phi_\ka(t,\xi)$ and $\Psi_\ka(t,y)$ be given by \eqref{def:Phikappa} and \eqref{def:Psikappa} respectively. Then if $\ka\in ]1/2,\oo[,$ for $\elk\eqdefa \f{2\ka-1}{4\ka^2}$  and for any non-negative and non-decreasing function $\hs\in C^1(\R_+)$, there exists a positive constant $\fc$ so that
\beq\begin{split}\label{S5eq11}
&\| \hsf \epk\Dhk \pppk\|^2_{L^\oo_{[t_0,t]} (L^2_+)} +4\ka\elk\| \hsf \epk\Dhk\p_y \pppk\|^2_{L^2_{[t_0,t]} (L^2_+)}\\
& +2c\lam 2^k \|\hsf\epk\Dhk\pppk \|^2_{L^2_{[t_0,t];\dthtk} (L^2_+)}
\leq \| \hsf \epk\Dhk \pppk(t_0)\|^2_{L^2_+}\\
 &+\| \qhs \epk\Dhk \pppk\|^2_{L^2_{[t_0,t]} (L^2_+)}+\fc\ka \e d^2_k 2^{-k}\|\hsf\epk \p_y\pppk \|^2_{\wt{L}^2_{[t_0,t]}(\cB^{\f12,0})}\\
 &+Cd^2_k 2^{-k}\Bigl(\e^{-1} \|\ttau^\f12 \hsf \epk \Mpk \|^2_{\wt{L}^2_{[t_0,t]}(\cB^{\f12,0})}+ \|\hsf\epk\pppk \|^2_{\wt{L}^2_{[t_0,t];\dthtk}(\cB^{1,0})}\Bigr),
\end{split}\eeq
for any $t_0\in[0,t]$ with $t<T^*_\ka$, which is defined by \eqref{def:T^*kappa}.}
\end{lemma}

\begin{proof}
Along the same line to the proof of  Lemma \ref{prop5.1}, by applying the dyadic operator $\Dhk$ to
the modified versions of \eqref{S5eq2} and \eqref{S5eq3} (with $\Phi$ being replaced by $\Phi_\ka$), and then taking the $L^2_+$ inner product of the resulting equations with   $\eppk\Dhk\pppk$ respectively, we  find that \eqref{S5eq4} holds
with $\tht, \Psi$ and $\Phi$ there  being replaced respectively by $\thtk, \Psi_\kappa$ and $\Phi_\ka.$
 Next we only present the detailed estimates to the terms in \eqref{S5eq4}, which are different from the proof of Lemma \ref{prop5.1}.

Observing that for $\ka\in ]1/2,\infty[,$ $2\ka-1\leq \ka^2$, we get, by applying \eqref{S4eq3} and \eqref{S4eq4}, that
\beqo\begin{split}
\int_{t_0}^t&\hbar(t')\Bigl(\bigl(\epk\Dhk(\p_t-\p^2_y)\vpk\big|\epk\Dhk\vpk \bigr)_{L^2_+}+\bigl(\epk\Dhk(\p_t-\ka\p^2_y)\ppk\big|\epk\Dhk\ppk \bigr)_{L^2_+}\Bigr)\,dt'\\
&\geq \ \f12\bigl(\|\hsf \epk\Dhk\pppk(t)\|^2_{L^2_+}-\|\hsf \epk\Dhk\pppk(t_0)\|^2_{L^2_+}\bigr) \\
&\qquad-\f12 \|\sqrt{\hbar'}\epk\Dhk\pppk\|^2_{L^2_{[t_0,t]}(L^2_+)}
+2\ka\elk\|\hsf \epk\Dhk\p_y\pppk\|^2_{L^2_{[t_0,t]}(L^2_+)}.
\end{split}\eeqo

While due to $\lim_{y\rto+\oo}u=\lim_{y\rto+\oo}b=0$, we get, by applying Lemma \ref{lem:nl1} with $a=\f1{2\ka}$ and $b=c=d=\f1\ka$, that
\beqo\begin{split}
&\int_{t_0}^t\hbar(t')\Bigl(\bigl|\bigl(\epk\Dhk[u\p_x\vphi-b\p_x\psi]_{\Phi_\ka}\big|\epk\Dhk\vpk \bigr)_{L^2_+}\bigr|\\
&\qquad+\bigl|\bigl(\epk\Dhk[u\p_x\psi-b\p_x\vphi]_{\Phi_\ka}\big|\epk\Dhk\ppk \bigr)_{L^2_+}\bigr|\Bigr)\,dt'\\
&\lesssim d_k^22^{-k}\|\hsf\epk\pppk\|_{\wt{L}^2_{[t_0,t];f}(\cB^{1,0})}^2 \with f(t)=\w{t}^{\f14}\|e^{\f{\Psi_\ka}2}\p_y(u,b)_{\Phi_\ka}(t)\|_{\cB^{\f12,0}}.
\end{split}\eeqo
Yet it follows from Lemma \ref{lem3.2} that
\beq\label{S6eq15}
f(t)\lesssim \w{t}^{\f14}\|\epk\p_y(G,H)_{\Phi_\ka}(t)\|_{\cB^{\f12,0}},
\eeq
hence we get, by a similar derivation of \eqref{S5eq7}, that
\beq\begin{split}\label{S5eq17}
\int_{t_0}^t\hbar(t')&\Bigl(\bigl|\bigl(\epk\Dhk[u\p_x\vphi-b\p_x\psi]_{\Phi_\ka}\big|\epk\Dhk\vpk \bigr)_{L^2_+}\bigr|\\
&+\bigl|\bigl(\epk\Dhk[u\p_x\psi-b\p_x\vphi]_{\Phi_\ka}\big|\epk\Dhk\ppk \bigr)_{L^2_+}\bigr|\Bigr)\,dt'\\
\lesssim & d_k^22^{-k}\|\hsf\epk\pppk\|_{\wt{L}^2_{[t_0,t];\dthtk}(\cB^{1,0})}^2.
\end{split}\eeq
Similarly, we get,  by applying Lemma \ref{lem:nl3} and a similar derivation of \eqref{S5eq17}, that
\beqo\begin{split}\label{S5eq18}
\int_{t_0}^t\hbar(t')\bigl| \bigl(\epk\int_y^\oo \Dhk[\p_yu\p_x\vphi-\p_yb\p_x\psi]_{\Phi_\ka} dy'\big| &\epk\Dhk\vpk \bigr)_{L^2_+}\bigr|\,dt' \\
&\lesssim d_k^22^{-k}\|\hsf\epk\ppp\|_{\wt{L}^2_{[t_0,t];\dthtk}(\cB^{1,0})}^2.
\end{split}\eeqo

On the other hand, we deduce from  \eqref{S8eq33aq}  that
\beno
\|\epk y\Dhk\pppk\|_{L^2_+}\leq 4\ka \w{t}\|\epk\p_y\Dhk\ppp\|_{L^2_+},
\eeno
which implies
\beno
\|\epk y\ppp\|_{\wt{L}^p_t(\cB^{s,0})}\lesssim \ka\|\ttau\epk\p_y\ppp\|_{\wt{L}^p_t(\cB^{s,0})}.
\eeno
So that
 by applying \eqref{s8eq10} and  \eqref{UBdecay}, we find
\beqo \begin{split}
\int_{t_0}^t&\hbar(t')\Bigl(\bigl|\bigl(\epk\Dhk[-\p_xUu+\p_xBb]_{\Phi_\ka}\big|\epk\chi\Dhk\vpk\bigr)_{L^2_+}\bigr|\\
&\qquad+\bigl|\bigl(\epk\Dhk[-\p_xUb+\p_xBu]_{\Phi_\ka}\big|\epk\chi\Dhk\ppk\bigr)_{L^2_+}\bigr|\Bigr)\,dt\\
\lesssim \ & d_k^2 2^{-{k}}\|\ttau\UBpk\|_{\wt{L}^\oo_{[t_0,t]}(\cB^{\f32}_\h)}\|\hsf\epk\ubpk\|_{\wt{L}^2_{[t_0,t]}(\cB^{\f12,0})}\|\ttau^{-1}\hsf\epk y\pppk\|_{\wt{L}^2_{[t_0,t]}(\cB^{\f12,0})}\\
\lesssim \ &\ka\e d^2_k 2^{-k}\|\hsf\epk\p_y\pppk\|^2_{\wt{L}^2_{[t_0,t]}(\cB^{\f12,0})}.
\end{split}\eeqo
Along the same line, we infer
\beqo \begin{split}
\int_{t_0}^t&\hbar(t')\bigl|\bigl(\epk\Dhk[\p_xU\vphi-\p_xB\psi]_{\Phi_\ka}\big|\epk\chi'\Dhk\vpk\bigr)_{L^2_+}\bigr|\,dt'\\
\lesssim \ & d_k^22^{-{k}}\|\ttau\UBpk \|_{\wt{L}^\oo_{[t_0,t]}(\cB^{\f32}_\h)}\|\ttau^{-\f12}\hsf\epk\pppk \|_{\wt{L}^2_{[t_0,t]}(\cB^{\f12,0})}
\|\ttau^{-\f12}\hsf\epk \vpk\|_{\wt{L}^2_{[t_0,t]}(\cB^{\f12,0})}\\
\lesssim \ &\ka\e d^2_k 2^{-k}\|\hsf\epk\p_y\pppk\|^2_{\wt{L}^2_{[t_0,t]}(\cB^{\f12,0})}.
\end{split}\eeqo
Moreover, due to $\mbox{supp}\chi''\subset[1,2]$,, we deduce that the term
$
\int_{t_0}^t\hbar(t')\bigl|\bigl(\epk\int_y^\infty \chi''\Dhk[\p_xU\vphi-\p_xB\psi]_{\Phi_\ka} dy'\big|\epk\Dhk\vpk\bigr)_{L^2_+}\bigr|\,dt'$ shares
the same estimate as above.

With the above estimates, \eqref{S5eq11} follows  by a similar derivation of \eqref{S5eq1}.
 This completes the proof of Lemma \ref{prop5.2}.
\end{proof}

An immediate corollary of  Lemmas  \ref{prop5.1} and \ref{prop5.2} gives rise to

\begin{col}\label{col5.1}
{\sl Under the assumptions of Lemma \ref{prop5.1}, there exist positive constants $\e_0, \lam_0$ so that for any $\lam\geq\lam_0$ and $\e\leq \e_0,$ one has
\beq\begin{split}\label{S5eq19}
&\| \hsf\ep \ppp\|_{\wt{L}^\oo_{[t_0,t]} (\cB^{\f12,0})} +\sqrt{l_\ka}\|\hsf\ep\p_y\ppp\|_{\wt{L}^2_{[t_0,t]} (\cB^{\f12,0})}\\
&\  \leq \|\hsf\ep\ppp(t_0)\|_{\cB^{\f12,0}}+\|\qhs \ep \ppp\|_{\wt{L}^2_{[t_0,t]}(\cB^{\f12,0})} \\
&\qquad\qquad\qquad\qquad+C\e^{-\f12}\|\ttau^\f12\hsf\ep\Mp \|_{\wt{L}^2_{[t_0,t]}(\cB^{\f12,0})}.
\end{split}\eeq}
\end{col}
\begin{proof} Indeed
by taking square root of \eqref{S5eq1} and then multiplying the resulting inequality by $2^\f{k}2$ and finally summing over $k\in\Z$, we arrive at
\beqo\begin{split}
&\| \hsf\ep\ppp\|_{\wt{L}^\oo_{[t_0,t]} (\cB^{\f12,0})} +2\sqrt{l_\ka}\|\hsf\ep\p_y\ppp\|_{\wt{L}^2_{[t_0,t]} (\cB^{\f12,0})}\\
&\qquad+\sqrt{2c\lam}\|\hsf\ep\ppp\|_{\wt{L}^2_{[t_0,t];\dtht} (\cB^{1,0})}\\
&\leq\  \|\hsf\ep\ppp(t_0)\|_{\cB^{\f12,0}}+\|\qhs \ep\ppp\|_{\wt{L}^2_{[t_0,t]}(\cB^{\f12,0})}+C\|\hsf\ep\ppp\|_{\wt{L}^2_{[t_0,t];\dtht} (\cB^{1,0})}\\
&\qquad+\sqrt{\fc\e}\|\hsf\ep\p_y\ppp \|_{\wt{L}^2_{[t_0,t]} (\cB^{\f12,0})}+C\e^{-\f12}\|\ttau^\f12\hsf\ep\Mp\|_{\wt{L}^2_{[t_0,t]} (\cB^{\f12,0})}.
\end{split}\eeqo
Taking $\lam$ big enough and $\e\leq \e_0$ small enough in the above inequality gives rise to \eqref{S5eq19}.
\end{proof}

\begin{col}\label{col5.2}
{\sl Under the assumptions of Lemma \ref{prop5.2}, there exist positive constants $\e_0, \lam_0$ so that for any $\lam\geq\lam_0$ and $\e\leq \e_0,$ one has
\beq\begin{split}\label{S5eq20}
&\| \hsf\epk\Dhk\pppk\|_{\wt{L}^\oo_{[t_0,t]} (\cB^{\f12,0})} +\sqrt{\ka\elk}\|\hsf\epk\Dhk\p_y\pppk\|_{\wt{L}^2_{[t_0,t]} (\cB^{\f12,0})}\\
&\leq \|\hsf\epk\Dhk\pppk(t_0)\|_{\cB^{\f12,0}}+\|\qhs \epk\Dhk \pppk\|_{\wt{L}^2_{[t_0,t]}(\cB^{\f12,0})}\\
&\qquad\qquad\qquad\qquad\qquad\qquad+C\e^{-\f12}\|\ttau^\f12\hsf\epk\Mpk \|_{\wt{L}^2_{[t_0,t]}(\cB^{\f12,0})}.
\end{split}\eeq}
\end{col}
\begin{proof}
This proof is the same as that of Corollary \ref{col5.1}, we omit the detail here.
\end{proof}

Now we are in a position to prove Propositions \ref{prop3.1} and \ref{prop3.4}.

\begin{proof}[Proof of Proposition \ref{prop3.1}] Let $\frak{c}$ be the constant given by  Lemma \ref{prop5.1}.
 We observe from Lemma \ref{lem:Poincare} that
\beqo
\int_0^t \ttau^{2\lk-2\frak{c}\e-1}\|\ep\Dhk\ppp\|^2_{L^2_+}d\tau\leq 2\int_0^t\|\ttau^{\lk-\frak{c}\e}\ep\Dhk\p_y\ppp\|^2_{L^2_+}d\tau.
\eeqo
So that by taking $t_0=0$ and $\hs(t)=\w{t}^{2(l_\ka-\frak{c}\e)}$ in \eqref{S5eq1}, we obtain
\beqo\begin{split}
&\|\tauk\ep\Dhk\ppp\|^2_{L^\oo_{t}(L^2_+)}+4\frak{c}\e\|\tauk\ep\p_y\Dhk\ppp\|^2_{L^2_t (L^2_+)}\\
&\quad+2c\lam 2^k\|\tauk\ep\Dhk\ppp\|^2_{L^2_{t;\dtht} (L^2_+)}\\
&\leq\|e^\f{y^2}8 e^{\de|D_x|} \Dhk(\vphi_0,\psi_0)\|^2_{L^2_+}+\fc\e d^2_k 2^{-k}\|\tauk\ep\p_y\ppp\|^2_{\wt{L}^2_{t}(\cB^{\f12,0})}\\
&\quad+Cd^2_k 2^{-k}\Bigl(\|\tauk\ep\ppp\|^2_{\wt{L}^2_{t;\dtht} (\cB^{1,0})}+\e^{-1}\|\ttau^{\f12+\lk-\fc\e} \ep \Mp \|^2_{\wt{L}^2_{t}(\cB^{\f12,0})}\Bigr).
\end{split}\eeqo
By taking square root of the above inequality and then multiplying the resulting one by $2^\f{k}2$ and finally summing over $k\in\Z$, we find
\beqo\begin{split}
&\|\tauk\ep\ppp\|_{\wt{L}^\oo_{t}(\cB^{\f12,0})}+\sqrt{2c\lam} \|\tauk\ep\ppp\|_{\wt{L}^2_{t;\dtht} (\cB^{1,0})}\\
&\leq\|e^\f{y^2}8 e^{\de|D_x|} (\vphi_0,\psi_0)\|_{\cB^{\f12,0}}
+\sqrt{C}\|\ttau^{\lk-\fc\e}\ep\ppp\|_{\wt{L}^2_{t;\dtht} (\cB^{1,0})}\\
&\qquad\qquad\qquad\qquad\qquad\qquad+\sqrt{\f{C}{{\e}}} \|\ttau^{\f12+\lk-\fc\e} \ep \Mp \|_{\wt{L}^2_{t}(\cB^{\f12,0})}.
\end{split}\eeqo
By taking $\lam$ to be so large that $c\lam>C$, and using  \eqref{mdecay}, we obtain
\beq \begin{split} \label{S5eq31}
\|\tauk\ep\ppp\|_{\wt{L}^\oo_{t}(\cB^{\f12,0})}\leq \|e^\f{y^2}8 e^{\de |D_x|}(\vphi_0,\psi_0) \|_{\cB^{\f12,0}}+ C\sqrt{\e}.
\end{split}  \eeq

Thanks to  \eqref{S5eq31}, we get, by taking $t_0=\f{t}2$ and $\hs(t)=\w{t}^{2(l_\ka-\fc\e)}$ in Corollary \ref{col5.1}, that
\beqo\begin{split}
\|\tauk&\ep\p_y\ppp \|_{\wt{L}^2_{[\f{t}2,t]}(\cB^{\f12,0})}\lesssim
 \|\langle{t}/2\rangle^{\lk-\fc\e}\ep\ppp ({t}/2)\|_{\cB^{\f12,0}}\\
&+\|\ttau^{\lk-\fc\e-\f12}\ep\ppp \|_{\wt{L}^2_{[\f{t}2,t]}(\cB^{\f12,0})}
+\e^{-\f12}\| \ttau^{\f12+\lk-\fc\e}\ep\Mp\|_{\wt{L}^2_{[\f{t}2,t]}(\cB^{\f12,0})}\\
\lesssim &\|\ttau^{\lk-\fc\e-\f12}\ep\ppp \|_{\wt{L}^2_{[\f{t}2,t]}(\cB^{\f12,0})}+\|e^\f{y^2}8 e^{\de |D_x|}(\vphi_0,\psi_0) \|_{\cB^{\f12,0}}+\sqrt{\e}.
\end{split}  \eeqo
Yet by virtue of Definition \ref{def2.2}, we deduce from \eqref{S5eq31} that
\beqo \begin{split}
\|\ttau^{\lk-\fc\e-\f12}\ep\ppp \|_{\wt{L}^2_{[\f{t}2,t]}(\cB^{\f12,0})}\lesssim & \|\ttau^{\lk-\fc\e}\ep\ppp \|_{\wt{L}^\oo_{[\f{t}2,t]}(\cB^{\f12,0})}\\
\lesssim &\|e^\f{y^2}8 e^{\de |D_x|}(\vphi_0,\psi_0) \|_{\cB^{\f12,0}}+ \sqrt{\e}.
\end{split}  \eeqo
As a consequence, we arrive at
\beqo \label{S5eq32}\begin{split}
\|\tauk\ep\p_y\ppp \|_{\wt{L}^2_{[\f{t}2,t]}(\cB^{\f12,0})}\lesssim \|e^\f{y^2}8 e^{\de |D_x|}(\vphi_0,\psi_0) \|_{\cB^{\f12,0}}+ \sqrt{\e}.
\end{split}  \eeqo
This together with \eqref{S5eq31} leads to \eqref{S3eq12as}.
\end{proof}

\begin{proof}[Proof of Proposition \ref{prop3.4}]
Similar to the proof of Proposition \ref{prop3.1}, we first observe from Lemma \ref{lem:Poincare} that
\beqo
\int_0^t \ttau^{2\elk-2\fc\e-1}\|\epk\Dhk\pppk\|^2_{L^2_+}\,d\tau\leq 2\ka\int_0^t\|\w{\tau}^{\elk-\fc\e}\epk\Dhk\p_y\pppk\|^2_{L^2_+}\,d\tau.
\eeqo
So that by taking $t_0=0$ and $\hs(t)=\w{t}^{2(\elk-\fc\e)}$ in \eqref{S5eq11}, we obtain
\beqo\begin{split}
&\|\tauek\epk\Dhk\pppk\|^2_{L^\oo_{[t_0,t]}(L^2_+)}+4\ka \fc\e\|\tauek\epk\p_y\Dhk\pppk \|^2_{L^2_{t}(L^2_+)}\\
&\quad+2c\lam 2^k\|\tauek\epk\Dhk\pppk\|^2_{L^2_{t;\dthtk} (L^2_+)}\\
&\leq\|e^\f{y^2}{8\ka} e^{\de|D_x|} \Dhk(\vphi_0,\psi_0)\|^2_{L^2_+}
+\ka \fc\e d^2_k 2^{-k}\|\tauek\epk\p_y\pppk \|^2_{\wt{L}^2_{t}(\cB^{\f12,0})}\\
&\quad+{C}d^2_k 2^{-k}\Bigl(\|\tauek\epk\pppk\|^2_{\wt{L}^2_{t;\dthtk} (\cB^{1,0})}
+\e^{-1} \|\ttau^{\f12+\elk-\fc\e}\epk\Mpk \|^2_{\wt{L}^2_{t}(\cB^{\f12,0})}\Bigr),
\end{split}\eeqo from which, we get, by a similar derivation of \eqref{S5eq19}, that
\beqo\begin{split}
&\|\tauek\epk\pppk\|_{\wt{L}^\oo_{t}(\cB^{\f12,0})}+\sqrt{2c\lam} \|\tauek\epk\pppk\|_{\wt{L}^2_{t;\dthtk} (\cB^{1,0})}\\
&\leq \|e^\f{y^2}{8\ka} e^{\de|D_x|} (\vphi_0,\psi_0)\|_{\cB^{\f12,0}}+\sqrt{C}\|\tauek\epk\pppk\|_{\wt{L}_{t;\dthtk} (\cB^{1,0})}\\ &\qquad\qquad\qquad\qquad\qquad+\sqrt{\f{C}\e}\|\ttau^{\f12+\elk-c\e}\epk\Mpk \|_{\wt{L}^2_{t}(\cB^{\f12,0})}.
\end{split}\eeqo
By taking $\lam$ to be so large that $c\lam>C$ and using \eqref{mdecay}, we obtain
\beq \begin{split} \label{S5eq33}
\|\tauek\epk\ppp\|_{\wt{L}^\oo_{t}(\cB^{\f12,0})}\leq  \|e^\f{y^2}{8\ka} e^{\de|D_x|} (\vphi_0,\psi_0)\|_{\cB^{\f12,0}}+C\sqrt{\e}
\end{split}  \eeq

On the other hand, thanks to \eqref{S5eq33}, we get, by taking $t_0=\f{t}2$ and $\hs(t)=\tt^{2(\elk-\fc\e)}$ in Corollary \ref{col5.2},
that
\beqo\begin{split}
&\|\tauek\epk\p_y\pppk \|_{\wt{L}^2_{[\f{t}2,t]}(\cB^{\f12,0})} \lesssim \|\langle{t}/2\rangle^{\elk-c\e}\epk\pppk ({t}/2)\|_{\cB^{\f12,0}}\\
&\quad+\|\ttau^{\elk-c\e-\f12}\epk\pppk \|_{\wt{L}^2_{[\f{t}2,t]}(\cB^{\f12,0})}
+\| \ttau^{\f12+\elk-\fc\e}\epk\Mpk\|_{\wt{L}^2_{[\f{t}2,t]}(\cB^{\f12,0})}\\
&\lesssim |\ttau^{\elk-\fc\e-\f12}\epk\pppk \|_{\wt{L}^2_{[\f{t}2,t]}(\cB^{\f12,0})}+\|e^\f{y^2}{8\ka} e^{\de|D_x|} (\vphi_0,\psi_0)\|_{\cB^{\f12,0}}+\sqrt{\e}.
\end{split}  \eeqo
Whereas it follows from \eqref{S5eq33} that
\beqo \begin{split}
\|\ttau^{\elk-\fc\e-\f12}\epk\pppk \|_{\wt{L}^2_{[\f{t}2,t]}(\cB^{\f12,0})}\lesssim & \|\ttau^{\elk-\fc\e}\epk\pppk \|_{\wt{L}^\oo_{[\f{t}2,t]}(\cB^{\f12,0})}\\
\lesssim & \|e^\f{y^2}{8\ka} e^{\de|D_x|} (\vphi_0,\psi_0)\|_{\cB^{\f12,0}}+\sqrt{\e}.
\end{split}  \eeqo
As a consequence, we arrive at
\beqo \label{S5eq34}\begin{split}
\|\tauek\epk\p_y\pppk \|_{\wt{L}^2_{[\f{t}2,t]}(\cB^{\f12,0})}\lesssim  \|e^\f{y^2}{8\ka} e^{\de|D_x|} (\vphi_0,\psi_0)\|_{\cB^{\f12,0}}+\sqrt{\e}.
\end{split}  \eeqo
Together with \eqref{S5eq33}, we conclude the proof  of Proposition \ref{prop3.4}.
\end{proof}

\renewcommand{\theequation}{\thesection.\arabic{equation}}
\setcounter{equation}{0}
\section{Analytic energy estimate of $(u,b)$ }\label{sect6}

The goal of this section is to present the {\it a priori} weighted analytic energy estimate to the solution $(u,b)$ of \eqref{eqs1}. That is,
 we are going to present the proof of Propositions \ref{prop3.2} and \ref{prop3.5}. The key ingredients lie in the following lemmas:

\begin{lemma}\label{prop6.1}
{\sl Let $(u,b)$ be a smooth enough solution of \eqref{eqs1}. Let $\Phi(t,\xi)$ and $\Psi(t,y)$ be given by \eqref{def:Phi} and \eqref{def:Psi} respectively. Then if $\ka\in ]0,2[,$ for $\lk\eqdefa \f{\ka(2-\ka)}4,$ and for any non-negative and non-decreasing function $\hs\in C^1(\R_+)$, there exists a positive constant $\fc$ so that
\beq\begin{split}\label{S6eq1}
&\| \hsf \ep\Dhk \ubp\|^2_{L^\oo_{[t_0,t]}(L^2_+)} +4\lk\| \hsf \ep\Dhk\p_y \ubp\|^2_{L^2_{[t_0,t]}(L^2_+)}\\
& +2c\lam 2^k \|\hsf\ep\Dhk\ubp \|^2_{L^2_{[t_0,t];\dtht}(L^2_+)}
\leq \| \hsf \ep\Dhk \ubp(t_0)\|^2_{L^2_+}\\ &+\| \qhs \ep\Dhk \ubp\|^2_{L^2_{[t_0,t]}(L^2_+)}+\fc\e d^2_k 2^{-k}\|\hsf\ep \p_y\ubp \|^2_{\wt{L}^2_{[t_0,t]}(\cB^{\f12,0})}\\
 &+Cd^2_k 2^{-k}\Bigl(\e^{-1} \|\ttau^\f12 \hsf \ep \mp \|^2_{\wt{L}^2_{[t_0,t]}(\cB^{\f12,0})}+ \|\hsf\ep\ubp \|^2_{\wt{L}^2_{[t_0,t];\dtht}(\cB^{1,0})}\Bigr),
\end{split}\eeq
for any $t_0\in[0,t]$ with $t<T^*$, which is defined by \eqref{def:T^*}.}
\end{lemma}

\begin{proof}
In view of \eqref{S3eq1}, by applying operator $e^{\Phi(t,|D_x|)}$ to \eqref{eqs1}, we write
\beq\begin{split}\label{S6eq2}
\p_t \up&-\p^2_y\up-\bar{B}_\ka\p_xb_\Phi+\lam\dtht(t)|D_x|\up +[u\p_x u-b\p_x b]_\Phi\\
&+[v\p_yu-h\p_y b]_\Phi+\chi'[U\p_xu-B\p_xb]_\Phi
+\chi'[\p_xUu-\p_xBb]_\Phi\\
&+\chi[-\p_xU\p_yu+\p_xB\p_yb]_\Phi+\chi''[Uv-Bh]_\Phi=(m_U)_\Phi,
\end{split} \eeq
and
\beq\begin{split}\label{S6eq3}
\p_t \bp-&\ka\p^2_y\bp-\bar{B}_\ka\p_xu_\Phi+\lam\dtht(t)|D_x|\bp+[u\p_x b-b\p_x u]_\Phi\\
&+[v\p_y b-h\p_yu]_\Phi+\chi'[U\p_xb-B\p_xu]_\Phi+\chi'[\p_xBu-\p_xUb]_\Phi\\
&+\chi[-\p_xU\p_yb+\p_xB\p_yu]_\Phi+\chi''[Bv-Uh]_\Phi=(m_B)_\Phi.
\end{split}\eeq

By virtue of a similar version of \eqref{symmetry}, we get, by   applying the dyadic operator $\Dhk$ to \eqref{S6eq2}, \eqref{S6eq3} and then taking $L^2_+$ inner product of the resulting equations with   $\epp\Dhk\ubp$, that
\beq\begin{split}\label{S6eq4}
&\bigl(\ep\Dhk(\p_t-\p^2_y)\up\big|\ep\Dhk\up \bigr)_{L^2_+}+\bigl(\ep\Dhk(\p_t-\ka\p^2_y)\bp\big|\ep\Dhk\bp \bigr)_{L^2_+}\\
&
+\lam\dtht(t)\bigl(\ep|D_x|\Dhk\up\big|\ep\Dhk\up \bigr)_{L^2_+}+\lam\dtht(t)\bigl(\ep|D_x|\Dhk\bp\big|\ep\Dhk\bp \bigr)_{L^2_+} \\
&
+\bigl(\ep\Dhk[u\p_xu-b\p_xb]_\Phi\big|\ep\Dhk\up \bigr)_{L^2_+}+\bigl(\ep\Dhk[u\p_xb-b\p_xu]_\Phi\big|\ep\Dhk\bp \bigr)_{L^2_+} \\
&
+\bigl(\ep\Dhk[v\p_yu-h\p_yb]_\Phi\big|\ep\Dhk\up \bigr)_{L^2_+}+\bigl(\ep\Dhk[v\p_yb-h\p_yu]_\Phi\big|\ep\Dhk\bp \bigr)_{L^2_+}\\
&
+\bigl(\ep\Dhk[U\p_xu-B\p_xb]_\Phi\big|\ep\chi'\Dhk\up \bigr)_{L^2_+}+\bigl(\ep\Dhk[U\p_xb-B\p_xu]_\Phi\big|\ep\chi'\Dhk\bp \bigr)_{L^2_+} \\
&
+\bigl(\ep\Dhk[Uv-Bh]_\Phi\big|\ep\chi''\Dhk\up \bigr)_{L^2_+}+\bigl(\ep\Dhk[Bv-Uh]_\Phi\big|\ep\chi''\Dhk\bp \bigr)_{L^2_+}\\
&
+\bigl(\ep\Dhk[-\p_xU\p_yu+\p_xB\p_yb]_\Phi\big|\ep\chi\Dhk\up \bigr)_{L^2_+}\\
&
+\bigl(\ep\Dhk[-\p_xU\p_yb+\p_xB\p_yu]_\Phi\big|\ep\chi\Dhk\bp \bigr)_{L^2_+} \\
&
+\bigl(\ep\Dhk[\p_xUu-\p_xBb]_\Phi\big|\ep\chi'\Dhk\up \bigr)_{L^2_+}+\bigl(\ep\Dhk[\p_xBu-\p_xUb]_\Phi\big|\ep\chi'\Dhk\bp \bigr)_{L^2_+} \\
&
=\bigl(\ep\Dhk(m_U)_\Phi\big|\ep\Dhk\up \bigr)_{L^2_+}+\bigl(\ep\Dhk(m_B)_\Phi\big|\ep\Dhk\bp \bigr)_{L^2_+}.
\end{split}\eeq
In what follows, we shall  use the technical lemmas in Section \ref{sect4} to handle term by term in \eqref{S6eq4}.

Notice that for $\ka\in]0,2[,$ $\ka(2-\ka)\leq1$, we get, by applying \eqref{S4eq1} and \eqref{S4eq2}, that
\beqo\begin{split}
\int_{t_0}^t&\hbar(t')\Bigl(\bigl(\ep\Dhk(\p_t-\p^2_y)\up\big|\ep\Dhk\up \bigr)_{L^2_+}+\bigl(\ep\Dhk(\p_t-\ka\p^2_y)\bp\big|\ep\Dhk\bp \bigr)_{L^2_+}\Bigr)\,dt'\\
&\geq \ \f12\bigl(\|\hsf \ep\Dhk\ubp(t)\|^2_{L^2_+}-\|\hsf \ep\Dhk\ubp(t_0)\|^2_{L^2_+}\bigr) \\
&\qquad-\f12 \|\sqrt{\hbar'}\ep\Dhk\ubp\|^2_{L^2_{[t_0,t]}(L^2_+)}
+2\lk\|\hsf \ep\Dhk\p_y\ubp\|^2_{L^2_{[t_0,t]}(L^2_+)}.
\end{split}\eeqo

We deduce from Lemma \ref{lem:Bern}  that
\beqo\begin{split}
\lam\int_{t_0}^t\hbar(t')\dtht(t')\Bigl(\bigl(\ep|D_x|\Dhk\up\big|\ep\Dhk\up \bigr)_{L^2_+}+&\bigl(\ep|D_x|\Dhk\bp\big|\ep\Dhk\bp \bigr)_{L^2_+}\Bigr)\,dt'\\
&
\geq c\lam 2^k\int_{t_0}^t\dtht(t')\|\hsf\ep\Dhk\ubp \|^2_{L^2_+}\,dt'.
\end{split}\eeqo

Whereas due to $\lim_{y\rto+\oo}u=\lim_{y\rto+\oo}b=0$, we get, by applying Lemma \ref{lem:nl1} with $a=\f12$ and $b=c=d=1$, that
\beqo\begin{split}
&\int_{t_0}^t\hbar(t')\Bigl(\bigl|\bigl(\ep\Dhk[u\p_xu-b\p_xb]_\Phi\big|\ep\Dhk\up \bigr)_{L^2_+}\bigr|+\bigl|\bigl(\ep\Dhk[u\p_xb-b\p_xu]_\Phi\big|\ep\Dhk\bp \bigr)_{L^2_+}\bigr|\Bigr)\,dt'\\
&\qquad\qquad\lesssim d_k^22^{-k}\|\hsf\ep\ubp\|_{\wt{L}^2_{[t_0,t];f}(\cB^{1,0})}^2 \with f(t)=\w{t}^{\f14}\|e^{\f{\Psi}2}\p_y(u,b)_\Phi(t)\|_{\cB^{\f12,0}},
\end{split}\eeqo from which and \eqref{S5eq5}, we deduce from a similar derivation of \eqref{S5eq7} that
\beq\begin{split}\label{S6eq7}
\int_{t_0}^t&\hbar(t')\Bigl(\bigl|\bigl(\ep\Dhk[u\p_xu-b\p_xb]_\Phi\big|\ep\Dhk\up \bigr)_{L^2_+}\bigr|\\
&+\bigl|\bigl(\ep\Dhk[u\p_xb-b\p_xu]_\Phi\big|\ep\Dhk\bp \bigr)_{L^2_+}\bigr|\Bigr)\,dt'\lesssim d_k^22^{-k}\|\hsf\ep\ubp\|_{\wt{L}^2_{[t_0,t];\dtht}(\cB^{1,0})}^2.
\end{split}\eeq

Due to $\p_xu+\p_yv=0=\p_xb+\p_yh,$ we write
\beq \label{S6eq8}
(v,h)=-\int_y^\oo \p_y (v,h)dy'=\int_y^\oo \p_x (u,b)dy'.
\eeq
Then by applying Lemma \ref{lem:nl2}, we get, by a similar derivation of \eqref{S6eq7}, that
\beqo\begin{split}
\int_{t_0}^t&\hbar(t')\Bigl(\bigl(\bigl|\ep\Dhk[v\p_yu-h\p_yb]_\Phi\big|\ep\Dhk\up \bigr)_{L^2_+}\bigr|\\&+\bigl|\bigl(\ep\Dhk[v\p_yb-h\p_yu]_\Phi\big|\ep\Dhk\bp \bigr)_{L^2_+}\bigr|\Bigr)\,dt'
\lesssim d_k^22^{-k}\|\hsf\ep\ubp\|_{\wt{L}^2_{[t_0,t];\dtht}(\cB^{1,0})}^2.
\end{split}\eeqo

While it follows from  Lemma \ref{ff1} that
\beqo \begin{split}
\int_{t_0}^t&\hbar(t')\Bigl(\bigl|\bigl(\ep\Dhk[U\p_xu-B\p_xb]_\Phi\big|\ep\chi'\Dhk\up\bigr)_{L^2_+}\bigr|\\
&
+\bigl|\bigl(\ep\Dhk[U\p_xb-B\p_xu]_\Phi\big|\ep\chi'\Dhk\bp\bigr)_{L^2_+}\bigr|\Bigr)\,dt'
\lesssim \e^{\f12}d_k^22^{-k}\|\hsf\ep\ubp\|_{\wt{L}^2_{[t_0,t];\dtht}(\cB^{1,0})}^2.
\end{split}\eeqo

Notice from the definition of $\chi$ above \eqref{change} that $\mbox{supp}\chi''\subset [1,2]$. Then in view
of  \eqref{S6eq8},  we
get, by a similar derivation of \eqref{s8eq6}, that
\beqo \begin{split}
\int_{t_0}^t&\hbar(t')\Bigl(\bigl|\bigl(\ep\Dhk[Uv-Bh]_\Phi\big|\ep\chi''\Dhk\up\bigr)_{L^2_+}\bigr|\\
&
+\bigl|\bigl(\ep\Dhk[Bv-Uh]_\Phi\big|\ep\chi''\Dhk\bp\bigr)_{L^2_+}\bigr|\Bigr)\,dt'
\lesssim \e^{\f12} d_k^22^{-k}\|\hsf\ep\ubp\|_{\wt{L}^2_{[t_0,t];\dtht}(\cB^{1,0})}^2.
\end{split}\eeqo

On the other hand, we observing from  \eqref{S8eq34aq}  that
\beno
\|\ep y\Dhk\ubp\|_{L^2_+}\lesssim \w{t}\|\ep\p_y\Dhk\ubp\|_{L^2_+},
\eeno
which implies
\beno
\|\ep y\ubp\|_{\wt{L}^p_t(\cB^{s,0})}\lesssim \|\ttau\ep\p_y\ubp\|_{\wt{L}^p_t(\cB^{s,0})}.
\eeno
So that
 by applying \eqref{s8eq10} and  \eqref{UBdecay}, we infer
\beqo \begin{split}
\int_{t_0}^t&\hbar(t')\Bigl(\bigl|\bigl(\ep\Dhk[-\p_xU\p_yu+\p_xB\p_yb]_\Phi\big|\ep\chi\Dhk\up\bigr)_{L^2_+}\bigr|\\
&\qquad+\bigl|\bigl(\ep\Dhk[-\p_xU\p_yb+\p_xB\p_yu]_\Phi\big|\ep\chi\Dhk\bp\bigr)_{L^2_+}\bigr|\Bigr)\,dt\\
\lesssim \ & d_k^2 2^{-{k}}\|\ttau\UBp\|_{\wt{L}^\oo_{[t_0,t]}(\cB^{\f32}_\h)}\|\hsf\ep\p_y\ubp\|_{\wt{L}^2_{[t_0,t]}(\cB^{\f12,0})}\|\ttau^{-1}\hsf\ep y\ubp\|_{\wt{L}^2_{[t_0,t]}(\cB^{\f12,0})}\\
\lesssim \ &\e d^2_k 2^{-k}\|\hsf\ep\p_y\ubp\|^2_{\wt{L}^2_{[t_0,t]}(\cB^{\f12,0})}.
\end{split}\eeqo
Along the same line, we deduce that
\beqo \begin{split}
\int_{t_0}^t&\hbar(t')\bigl(\bigl|\bigl(\ep\Dhk[\p_xUu-\p_xBb]_\Phi\big|\ep\chi'\Dhk\vp\bigr)_{L^2_+}\bigr|\\
&\qquad+\bigl|\bigl(\ep\Dhk[\p_xBu-\p_xUb]_\Phi\big|\ep\chi'\Dhk\vp\bigr)_{L^2_+}\bigr|\bigr)\,dt'\\
\lesssim \ & d_k^22^{-{k}}\|\ttau\UBp \|_{\wt{L}^\oo_{[t_0,t]}(\cB^{\f32}_\h)}\|\ttau^{-\f12}\hsf\ep\ubp \|_{\wt{L}^2_{[t_0,t]}(\cB^{\f12,0})}
\|\ttau^{-\f12}\hsf\ep \ubp\|_{\wt{L}^2_{[t_0,t]}(\cB^{\f12,0})}\\
\lesssim \ &\e d^2_k 2^{-k}\|\hsf\ep\p_y\ubp\|^2_{\wt{L}^2_{[t_0,t]}(\cB^{\f12,0})}.
\end{split}\eeqo

Finally for the remained source term, we get, by using  Young's inequality, that
\beqo \label{S6eq9}\begin{split}
\int_{t_0}^t&\hbar(t')\Bigl(\bigl|\bigl(\ep\Dhk(m_U)_\Phi\big|\ep\Dhk\up\bigr)_{L^2_+}\bigr|
+\bigl|\bigl(\ep\Dhk(m_B)_\Phi\big|\ep\Dhk\bp\bigr)_{L^2_+}\bigr|\Bigr)\,dt'\\
\lesssim \ & d^2_k2^{-k} \|\ttau^{\f12}\hsf\ep\mp\|_{\wt{L}^2_{[t_0,t]}(\cB^{\f12,0})}\|\ttau^{-\f12}\hsf\ep\ubp\|_{\wt{L}^2_{[t_0,t]}(\cB^{\f12,0})}\\
\lesssim \ & d^2_k2^{-k}\bigl( \e \|\hsf\ep\p_y\ubp\|^2_{\wt{L}^2_{[t_0,t]}(\cB^{\f12,0})}+\e^{-1} \|\ttau^\f12\hsf\ep\mp\|^2_{\wt{L}^2_{[t_0,t]}(\cB^{\f12,0})}\bigr)
\end{split}\eeqo

By multiplying \eqref{S6eq4} by $\hbar(t)$ and then integrating the resulting equality over $[t_0,t],$
and finally inserting the above estimates to the equality, we
 obtain \eqref{S6eq1}. This completes the proof of Lemma \ref{prop6.1}.
\end{proof}

\begin{lemma}\label{prop6.2}
{\sl Let $(u,b)$ be a smooth enough solution of \eqref{eqs1}. Let $\Phi_\ka(t,\xi)$ and $\Psi_\ka(t,y)$ be given by \eqref{def:Phikappa} and \eqref{def:Psikappa} respectively. Then if $\ka\in ]1/2,\oo[,$ for $\elk\eqdefa \f{2\ka-1}{4\ka^2},$  and for any non-negative and non-decreasing function $\hs\in C^1(\R_+)$, there exists a positive constant $\fc$ so that
\beq\begin{split}\label{S6eq11}
&\| \hsf \epk\Dhk \ubpk\|^2_{L^\oo_{[t_0,t]}(L^2_+)} +4\ka\elk\| \hsf \epk\Dhk\p_y \ubpk\|^2_{L^2_{[t_0,t]}(L^2_+)}\\
&\ +2c\lam 2^k \|\hsf\epk\Dhk\ubpk \|^2_{L^2_{[t_0,t];\dthtk}(L^2_+)}\leq \| \hsf \epk\Dhk \ubp(t_0)\|^2_{L^2_+}\\
&\ +\| \qhs \epk\Dhk \ubpk\|^2_{L^2_{[t_0,t]}(L^2_+)}
+\ka\fc\e d^2_k 2^{-k}\|\hsf\epk \p_y\ubpk \|^2_{\wt{L}^2_{[t_0,t]}(\cB^{\f12,0})}\\
 &\ +C d^2_k 2^{-k} \Bigl(\|\hsf\epk\ubpk \|^2_{\wt{L}^2_{[t_0,t];\dthtk}(\cB^{1,0})}+\e^{-1} \|\ttau^\f12 \hsf\epk \mpk \|^2_{\wt{L}^2_{[t_0,t]}(\cB^{\f12,0})}\Bigr),
\end{split}\eeq
for any $t_0\in[0,t]$ with $t<T^*_\ka$, which is defined by \eqref{def:T^*kappa}.}
\end{lemma}
\begin{proof} The proof of Lemma \ref{prop6.2} is similar to that of
Lemma \ref{prop6.1}. We first observe that  \eqref{S6eq4}  holds with $\tht, \Phi$ and $\Psi$
 being replaced  by $\thtk, \Phi_\ka$  and $\Psi_\ka.$ In what follows, we only present
 the detailed estimates for terms which are different from those in the proof of Lemma \ref{prop6.1}.

Notice that for $\ka\in]1/2,\infty[,$ $2\ka-1\leq\ka^2$, we get, by applying \eqref{S4eq3} and \eqref{S4eq4}, that
\beqo\begin{split}
\int_{t_0}^t&\hbar(t')\Bigl(\bigl(\epk\Dhk(\p_t-\p^2_y)\upk\big|\epk\Dhk\upk \bigr)_{L^2_+}+\bigl(\epk\Dhk(\p_t-\ka\p^2_y)\bpk\big|\epk\Dhk\bpk \bigr)_{L^2_+}\Bigr)\,dt'\\
&\geq \ \f12\bigl(\|\hsf \epk\Dhk\ubpk(t)\|^2_{L^2_+}-\|\hsf \epk\Dhk\ubpk(t_0)\|^2_{L^2_+}\bigr) \\
&\qquad-\f12 \|\sqrt{\hbar'}\epk\Dhk\ubp\|^2_{L^2_{[t_0,t]}(L^2_+)}
+2\ka\elk\|\hsf \epk\Dhk\p_y\ubpk\|^2_{L^2_{[t_0,t]}(L^2_+)}.
\end{split}\eeqo

Whereas due to $\lim_{y\rto+\oo}u=\lim_{y\rto+\oo}b=0$, we get, by applying Lemma \ref{lem:nl1} with $a=\f1{2\ka}$ and $b=c=d=\f1\ka$, that
\beqo\begin{split}
&\int_{t_0}^t\hbar(t')\Bigl(\bigl|\bigl(\epk\Dhk[u\p_xu-b\p_xb]_\Phi\big|\epk\Dhk\up \bigr)_{L^2_+}\bigr|+\bigl|\bigl(\epk\Dhk[u\p_xb-b\p_xu]_\Phi\big|\epk\Dhk\bp \bigr)_{L^2_+}\bigr|\Bigr)\,dt'\\
&\qquad\qquad\lesssim d_k^22^{-k}\|\hsf\epk\ubp\|_{\wt{L}^2_{[t_0,t];f}(\cB^{1,0})}^2 \with f(t)=\w{t}^{\f14}\|e^{\f1{2} \Psi_\ka}\p_y(u,b)_\Phi(t)\|_{\cB^{\f12,0}},
\end{split}\eeqo
from which and a similar derivation of \eqref{S5eq17}, we infer
\beq\begin{split}
\int_{t_0}^t&\hbar(t')\Bigl(\bigl|\bigl(\epk\Dhk[u\p_xu-b\p_xb]_\Phi\big|\epk\Dhk\up \bigr)_{L^2_+}\bigr|\\
&+\bigl|\bigl(\epk\Dhk[u\p_xb-b\p_xu]_\Phi\big|\epk\Dhk\bp \bigr)_{L^2_+}\bigr|\Bigr)\,dt'\lesssim d_k^22^{-k}\|\hsf\epk\ubp\|_{\wt{L}^2_{[t_0,t];\dthtk}(\cB^{1,0})}^2.
\end{split}\eeq

On the other hand, we deduce  from  \eqref{S8eq33aq}  that
\beno
\|\epk y\Dhk\ubp\|_{L^2_+}\leq 4\ka\w{t}\|\epk\p_y\Dhk\ubp\|_{L^2_+},
\eeno
which implies
\beno
\|\epk y\ubp\|_{\wt{L}^p_t(\cB^{s,0})}\leq 4\ka \|\ttau\epk\p_y\ubp\|_{\wt{L}^p_t(\cB^{s,0})}.
\eeno
So that
 by applying \eqref{s8eq10} and  \eqref{UBdecay}, we find
\beqo \begin{split}
\int_{t_0}^t&\hbar(t')\Bigl(\bigl|\bigl(\epk\Dhk[-\p_xU\p_yu+\p_xB\p_yb]_\Phi\big|\epk\chi\Dhk\up\bigr)_{L^2_+}\bigr|\\
&\qquad+\bigl|\bigl(\epk\Dhk[-\p_xU\p_yb+\p_xB\p_yu]_\Phi\big|\epk\chi\Dhk\bp\bigr)_{L^2_+}\bigr|\Bigr)\,dt\\
\lesssim \ & d_k^2 2^{-{k}}\|\ttau\UBp\|_{\wt{L}^\oo_{[t_0,t]}(\cB^{\f32}_\h)}\|\hsf\epk\p_y\ubp\|_{\wt{L}^2_{[t_0,t]}(\cB^{\f12,0})}\|\ttau^{-1}\hsf\epk y\ubp\|_{\wt{L}^2_{[t_0,t]}(\cB^{\f12,0})}\\
\lesssim \ &\ka\e d^2_k 2^{-k}\|\hsf\epk\p_y\ubp\|^2_{\wt{L}^2_{[t_0,t]}(\cB^{\f12,0})}.
\end{split}\eeqo
Along the same line, we infer
\beqo \begin{split}
\int_{t_0}^t&\hbar(t')\bigl(\bigl|\bigl(\epk\Dhk[\p_xUu-\p_xBb]_\Phi\big|\epk\chi'\Dhk\vp\bigr)_{L^2_+}\bigr|\\
&\qquad+\bigl|\bigl(\epk\Dhk[\p_xBu-\p_xUb]_\Phi\big|\epk\chi'\Dhk\vp\bigr)_{L^2_+}\bigr|\bigr)\,dt'\\
\lesssim \ & d_k^22^{-{k}}\|\ttau\UBp \|_{\wt{L}^\oo_{[t_0,t]}(\cB^{\f32}_\h)}\|\ttau^{-\f12}\hsf\epk\ubp \|_{\wt{L}^2_{[t_0,t]}(\cB^{\f12,0})}
\|\ttau^{-\f12}\hsf\epk \ubp\|_{\wt{L}^2_{[t_0,t]}(\cB^{\f12,0})}\\
\lesssim \ & \ka\e d^2_k 2^{-k}\|\hsf\epk\p_y\ubp\|^2_{\wt{L}^2_{[t_0,t]}(\cB^{\f12,0})}.
\end{split}\eeqo

With the above estimates, we can repeat the derivation of \eqref{S6eq1} to complete the proof of \eqref{S6eq11}. This completes the proof of Lemma \ref{prop6.2}.
\end{proof}

It is easy to observe the following corollaries from  Lemmas \ref{prop6.1} and \ref{prop6.2}.

\begin{col}\label{col6.1}
{\sl Under the assumptions of Lemma \ref{prop6.1}, there exist positive constants $\e_0, \lam_0$ so that for any $\lam\geq\lam_0$ and $\e\leq \e_0,$ one has
\beq\begin{split}\label{S6eq19}
\| \hsf&\ep \ubp\|_{\wt{L}^\oo_{[t_0,t]} (\cB^{\f12,0})} +{\sqrt{c\lam}}\|\hsf\ep\ubp\|_{\wt{L}^2_{[t_0,t];\dtht} (\cB^{1,0})}\\
&+\sqrt{l_\ka}\|\hsf\ep\p_y\ubp\|_{\wt{L}^2_{[t_0,t]} (\cB^{\f12,0})}
\leq \|\hsf\ep\ubp(t_0)\|_{\cB^{\f12,0}}\\
&+\|\qhs \ep \ubp\|_{\wt{L}^2_{[t_0,t]}(\cB^{\f12,0})}+C\e^{-\f12}\|\ttau^\f12\hsf\ep\mp \|_{\wt{L}^2_{[t_0,t]}(\cB^{\f12,0})}.
\end{split}\eeq}
\end{col}
\begin{proof}
By taking square root of \eqref{S6eq1} and then multiplying the resulting inequality by $2^\f{k}2$ and finally summing over $k\in\Z$, we find
\beqo\begin{split}
&\| \hsf\ep\ubp\|_{\wt{L}^\oo_{[t_0,t]} (\cB^{\f12,0})} +2\sqrt{l_\ka}\|\hsf\ep\p_y\ubp\|_{\wt{L}^2_{[t_0,t]} (\cB^{\f12,0})}\\
&\qquad+\sqrt{2c\lam}\|\hsf\ep\ubp\|_{\wt{L}^2_{[t_0,t];\dtht} (\cB^{1,0})}\leq\  \|\hsf\ep\ubp(t_0)\|_{\cB^{\f12,0}}\\
&\qquad+\|\qhs \ep\ubp\|_{\wt{L}^2_{[t_0,t]}(\cB^{\f12,0})}+\sqrt{\fc\e}\|\hsf\ep\p_y\ubp \|_{\wt{L}^2_{[t_0,t]} (\cB^{\f12,0})}\\
&\qquad+C\Bigl(\|\hsf\ep\ubp\|_{\wt{L}^2_{[t_0,t];\dtht} (\cB^{1,0})}+\e^{-\f12} \|\ttau^\f12\hsf\ep\mp\|_{\wt{L}^2_{[t_0,t]} (\cB^{\f12,0})}\Bigr).
\end{split}\eeqo
Taking $\lam$ big enough and $\e\leq\e_0$ small enough in the above inequality gives rise to \eqref{S6eq19}.
\end{proof}

\begin{col}\label{col6.2}
{\sl Under the assumption of Lemma  \ref{prop6.2}, there exist positive constants $\e_0, \lam_0$ so that for any $\lam\geq\lam_0$ and $\e\leq \e_0,$ one has
\beq\begin{split}\label{S6eq20}
&\| \hsf\epk\ubpk\|_{\wt{L}^\oo_{[t_0,t]} (\cB^{\f12,0})} +{\sqrt{c\lam}}\|\hsf\ep\ubpk\|_{\wt{L}^2_{[t_0,t];\dthtk} (\cB^{1,0})}\\
&+\sqrt{\ka\elk}\|\hsf\epk\p_y\ubpk\|_{\wt{L}^2_{[t_0,t]} (\cB^{\f12,0})}
\leq \|\hsf\epk\ubp(t_0)\|_{\cB^{\f12,0}}\\
&+\|\qhs \epk\ubpk\|_{\wt{L}^2_{[t_0,t]}(\cB^{\f12,0})} +C\e^{-\f12}\|\ttau^\f12\hsf\epk\mpk \|_{\wt{L}^2_{[t_0,t]}(\cB^{\f12,0})}.
\end{split}\eeq}
\end{col}
\begin{proof}
The proof  is the same as that of Corollary \ref{col6.1}. We omit the details here.
\end{proof}

Now we are in a position to prove Propositions \ref{prop3.2} and \ref{prop3.5}.

\begin{proof}[Proof of Proposition \ref{prop3.2}]
Taking $\hs(t)=1$ and $t_0=0$ in Corollary \ref{col6.1} gives rise to
\beq\label{S6eq21}\begin{split}
\|\ep\ubp\|_{\wt{L}^\oo_t (\cB^{\f12,0})}+{\sqrt{c\lam}}\|\ep\ubp&\|_{\wt{L}^2_{t;\dtht} (\cB^{1,0})} +\sqrt{l_\ka}\|\ep\p_y\ubp\|_{\wt{L}^2_t (\cB^{\f12,0})}\\
&\qquad\leq \|e^{\f{y^2}8}e^{\de|D_x|}(u_0,b_0)\|_{\cB^{\f12,0}}+ C\sqrt{\e}.
\end{split}
\eeq

While by taking $t_0=\f{t}2$ and $\hs(t)=(t-t_0)^{1+2\lk-2c\e}$ in Corollary \ref{col6.1}, we find
\beq\label{S6eq22}\begin{split}
\|t^{\f12+\lk-\fc\e}&\ep\ubp(t)\|_{\cB^{\f12,0}}\lesssim\   \|(\tau-{t}/2)^{\f12+\lk-\fc\e}\ep\ubp\|_{\wt{L}^\oo_{[{t}/2,t]} (\cB^{\f12,0})}\\
 \lesssim\ & \|\tau^{\lk-\fc\e}\ep\ubp\|_{\wt{L}^2_{[{t}/2,t]} (\cB^{\f12,0})}+\e^{-\f12}\|\ttau^{1+\lk-\fc\e}\ep\mp \|_{\wt{L}^2_{t}(\cB^{\f12,0})}\\
\lesssim\ &\|\tau^{\lk-\fc\e}\ep\ubp\|_{\wt{L}^2_{[{t}/2,t]} (\cB^{\f12,0})}+\sqrt{\e}.
\end{split}
\eeq
Due to $(u,b)=\p_y(\vphi,\psi)$, we get, by applying Proposition \ref{prop3.1},  that
\beq\label{S6eq24}
\|t^{\f12+\lk-\fc\e}\ep\ubp(t)\|_{\cB^{\f12,0}}\lesssim \|e^\f{y^2}8 e^{\de |D_x|}(\vphi_0,\psi_0) \|_{\cB^{\f12,0}}+ \sqrt{\e}.
\eeq

Thanks to Proposition \ref{prop3.1} and \eqref{S6eq24}, we get, by taking $t_0=\f{t}2$ and $\hs(t)=t^{1+2\lk-2\fc\e}$ in Corollary \ref{col6.1}, that
\beqo\begin{split}
\|&\tau^{\f12+\lk-c\e}\ep\p_y\ubp\|_{\wt{L}^2_{[{t}/2,t]}(\cB^{\f12,0})}
\lesssim  \|({t}/2)^{\f12+\lk-\fc\e}\ep\ubp({t}/2)\|_{\cB^{\f12,0}}\\
&\quad+\|\tau^{\lk-\fc\e}\ep\ubp\|_{\wt{L}^2_{[{t}/2,t]} (\cB^{\f12,0})}+\e^{-\f12}\|\ttau^{1+\lk-\fc\e}\ep\mp \|_{\wt{L}^2_{t}(\cB^{\f12,0})}\\
&\lesssim \|e^\f{y^2}8 e^{\de |D_x|}(\vphi_0,\psi_0) \|_{\cB^{\f12,0}}+ C\sqrt{\e}.
\end{split}\eeqo
which together with \eqref{S6eq21} and \eqref{S6eq24} ensures Proposition \ref{prop3.2}.
\end{proof}

\begin{proof}[Proof of Proposition \ref{prop3.5}] The proof of this proposition is the same as that of Proposition \ref{prop3.2}.
Indeed taking $\hs(t)=1$ and $t_0=0$ in Corollary \ref{col6.2} gives rise to
\beq\label{S6eq25}\begin{split}
\|\epk\ubpk&\|_{\wt{L}^\oo_t (\cB^{\f12,0})} +\sqrt{\ka\elk}\|\epk\p_y\ubpk\|_{\wt{L}^2_t (\cB^{\f12,0})}\\
&+{\sqrt{c\lam}}\|\hsf\epk\ubpk\|_{\wt{L}^2_{[t_0,t];\dthtk} (\cB^{1,0})}\leq   \|e^{\f{y^2}{8\ka}}e^{\de|D_x|}(u_0,b_0)\|_{\cB^{\f12,0}}+ C\sqrt{\e}..
\end{split}
\eeq
We omit the other details here.
\end{proof}

\renewcommand{\theequation}{\thesection.\arabic{equation}}
\setcounter{equation}{0}
\section{Analytic energy estimates of the quantities $(G,H)$}\label{sect7}

In this section, we shall present the {\it a priori} weighted analytic energy estimate to the quantities $(G,H)$ which are defined by \eqref{defGH}. Those estimates will be crucial for us to globally control the analytic band of solutions to \eqref{eqs1} and \eqref{eqs2}.

We first observe from \eqref{defGH} and the boundary conditions in \eqref{eqs1} and \eqref{eqs2} that
$$
G|_{y=0}=\p_yH|_{y=0}=0.$$
While by multiplying the $\varphi$ equation of \eqref{eqs2} by $\fyt$  (resp. the $\psi$ equation of \eqref{eqs2} by
 $\f{y}{2\ka\w{t}}$) and summing up the resulting equation with the $u$ equation of  \eqref{eqs1} (resp.  the $b$ equation of  \eqref{eqs1}),
  we find that $(G,H)$ verifies
\begin{equation}\label{eqs3}
 \quad\left\{\begin{array}{l}
\p_t G-\p^2_yG+\f{G}\tt-\bar{B}_\ka\p_xH + u\p_xG-\ka b\p_x H-(1-\ka)b\p_xb+v\p_yu-h\p_y b \\ \qquad+\f{y}\tt \int_y^\infty (\p_x\vphi\p_yu-\p_x\psi\p_yb)dy'+\chi'(U\p_xG-\ka B\p_xH-(1-\ka)B\p_xb)\\ \qquad+\chi'(\p_xUu-\p_xBb)+\f{y}\tt \int_y^\oo \chi''(U\p_x\vphi-B\p_x\psi)dy'\\ \qquad+\chi(-\p_xU\p_yu+\p_xB\p_yb)+\chi\fyt(-\p_xUu+\p_xBb)+\chi''(Uv-Bh)\\ \qquad+\chi'\f{y}\tt (\p_xU\vphi-\p_xB\psi)+\f{y}\tt \int_y^\oo \chi''(\p_xU\vphi-\p_xB\psi)=m_U+\fyt M_U,\\
\p_t H-\ka\p^2_yH+\f{H}\tt-\bar{B}_\ka\p_xG +u\p_x H-\f1\ka b\p_x G-(1-\f1\ka)b\p_xu+v\p_y b -h\p_y u\\ \qquad+\chi'(U\p_xH-\f1\ka B\p_xG-(1-\f1\ka)B\p_xu)+\chi'(\p_xB u-\p_xUb)\\ \qquad+\chi(-\p_xU\p_yb+\p_xB\p_yu)+\fykt\chi(-\p_xUb+\p_xBu)\\ \qquad+\chi''(Bv-Uh)=m_B+\fykt M_B,\\
G|_{y=0}=\p_yH|_{y=0}=0 \andf \lim_{y\rto+\infty} G=\lim_{y\rto+\infty} H=0,\\
G|_{t=0}=G_0,\ \ H|_{t=0}=H_0,
\end{array}\right.
\end{equation}

We remark that the restriction for $\bar{B}_\ka=0$ except $\ka=1$ is only used in the derivation of the system \eqref{eqs3}.

The key ingredients to prove Propositions \ref{prop3.3} and \ref{prop3.6} lie in the following two lemmas:

\begin{lemma}\label{prop7.1}
{\sl Let $(u,b)$ and $(\vphi,\psi)$ be smooth enough solutions of \eqref{eqs1} and \eqref{eqs2} respectively, and $(G,H)$ be determined by \eqref{defGH}. Let $\Phi(t,\xi)$ and $\Psi(t,y)$ be given   respectively by \eqref{def:Phi} and \eqref{def:Psi}. Then if $\ka\in ]0,2[,$ for $\lk\eqdefa \f{\ka(2-\ka)}4,$ and for any non-negative and non-decreasing function $\hs\in C^1(\R_+)$, there exists a positive  constant $\fc$ so that
\beq\begin{split}\label{S7eq1}
&\| \hsf \ep\Dhk \ghp\|^2_{L^\oo_{[t_0,t]}(L^2_+)} +4\lk\| \hsf \ep\Dhk\p_y \ghp\|^2_{L^2_{[t_0,t]}(L^2_+)}\\
&+2\|\ttau^{-\f12} \hsf \ep\Dhk  \ghp\|^2_{L^2_{[t_0,t]}(L^2_+)} +2c\lam 2^k \|\hsf\ep\Dhk\ghp \|^2_{L^2_{[t_0,t];\dtht}(L^2_+)}
\\ &\leq \| \hsf \ep\Dhk \ghp(t_0)\|^2_{L^2_+}+\| \qhs \ep\Dhk \ghp\|^2_{L^2_{[t_0,t]}(L^2_+)}\\
 &+\fc\e d^2_k 2^{-k}\|\hsf\ep \p_y\ghp \|^2_{\wt{L}^2_{[t_0,t]}(\cB^{\f12,0})}+C d^2_k 2^{-k} \Bigl(\|\hsf\ep\ghp \|^2_{\wt{L}^2_{[t_0,t];\dtht}(\cB^{1,0})}\\
 &\qquad+ \e\|\ttau^{-\f54}\hsf\ep\ubp \|^2_{\wt{L}^2_{[t_0,t];\dtht}(\cB^{1,0})}
 +\e\|\ttau^{-\f54}\hsf\ep \p_y\ubp \|^2_{\wt{L}^2_{[t_0,t]}(\cB^{\f12,0})}\\
 &\qquad\qquad\qquad\qquad\qquad\qquad+\e^{-1} \|\ttau^\f12 \hsf \ep \mMp \|^2_{\wt{L}^2_{[t_0,t]}(\cB^{\f12,0})}\Bigr),
\end{split}\eeq
for any $t_0\in[0,t]$ with $t<T^*$, which is defined by \eqref{def:T^*}.}
\end{lemma}

\begin{proof}
 In view of \eqref{S3eq1}, by applying operator $e^{\Phi(t,|D_x|)}$ to \eqref{eqs3}, we write
\beq\begin{split}\label{S7eq2}
&\p_t \gp-\p^2_y\gp+\lam\dtht(t)|D_x|\gp+\f{1}\tt\gp-\bar{B}_\ka\p_xH_\Phi+ [u\p_xG-\ka b\p_x H-(1-\ka)b\p_xb]_\Phi
\\&\qquad+[v\p_yu-h\p_y b]_\Phi+\f{y}\tt \int_y^\infty [\p_x\vphi\p_yu-\p_x\psi\p_yb]_\Phi dy' \\
& \qquad+\chi'[U\p_xG-\ka B\p_xH-(1-\ka)B\p_xb]_\Phi +\chi''[Uv-Bh]_\Phi\\& \qquad+\f{y}\tt \int_y^\oo \chi''[U\p_x\vphi-B\p_x\psi]_\Phi dy'+\chi'[\p_xUu-\p_xBb]_\Phi\\& \qquad+\chi[-\p_xU\p_yu+\p_xB\p_yb]_\Phi+\chi\fyt[-\p_xUu+\p_xBb]_\Phi\\& \qquad+\chi'\f{y}\tt [\p_xU\vphi-\p_xB\psi]_\Phi+\f{y}\tt \int_y^\oo \chi''[\p_xU\vphi-\p_xB\psi]_\Phi dy'=[m_U+\fyt M_U]_\Phi,
\end{split} \eeq
and
\beq\begin{split}\label{S7eq3}
&\p_t \hp-\ka\p^2_y\hp+\lam\dtht(t)|D_x|\hp+\f1\tt \hp-\bar{B}_\ka\p_xG_\Phi \\
&\qquad+[u\p_x H-\f1\ka b\p_x G-(1-\f1\ka)b\p_xu]_\Phi+[v\p_y b-h\p_y u]_\Phi  \\
& \qquad+\chi'[U\p_xH-\f1\ka B\p_xG-(1-\f1\ka)B\p_xu]_\Phi+\chi''[Bv-Uh]_\Phi\\
& \qquad+\chi'[\p_xB u-\p_xUb]_\Phi+\chi[-\p_xU\p_yb+\p_xB\p_yu]_\Phi\\&\qquad+\chi\fykt[-\p_xUb+\p_xB u]_\Phi=[m_B+\fykt M_B]_\Phi
\end{split}\eeq

Again thanks to a similar cancelation equality \eqref{symmetry}, we get, by first applying the dyadic operator $\Dhk$ to \eqref{S7eq2}, \eqref{S7eq3} and then taking the $L^2_+$ inner product of the resulting equations with   $\epp\Dhk\ghp$ respectively, that
\beq\begin{split}\label{S7eq4}
&\bigl(\ep\Dhk(\p_t-\p^2_y)\gp\big|\ep\Dhk\gp \bigr)_{L^2_+}+\bigl(\ep\Dhk(\p_t-\ka\p^2_y)\hp\big|\ep\Dhk\hp \bigr)_{L^2_+}\\
&
+\lam\dtht(t)\bigl(\ep|D_x|\Dhk\gp\big|\ep\Dhk\gp \bigr)_{L^2_+}+\lam\dtht(t)\bigl(\ep|D_x|\Dhk\hp\big|\ep\Dhk\hp \bigr)_{L^2_+} \\
&
+\f{1}{\tt} \|\ep\Dhk\gp\|^2_{L^2_+}+\f{1}{\tt} \|\ep\Dhk\hp\|^2_{L^2_+} \\
&
+\bigl(\ep\Dhk[u\p_xG-\ka b\p_xH]_\Phi\big|\ep\Dhk\gp \bigr)_{L^2_+}+\bigl(\ep\Dhk[u\p_xH-\f1\ka b\p_xG]_\Phi\big|\ep\Dhk\hp \bigr)_{L^2_+} \\
&
-(1-\ka)\bigl(\ep\Dhk[b\p_xb]_\Phi\big|\ep\Dhk\gp \bigr)_{L^2_+}-(1-\f1{\ka}) \bigl(\ep\Dhk[b\p_xu]_\Phi\big|\ep\Dhk\hp \bigr)_{L^2_+} \\
&
+\bigl(\ep\Dhk[v\p_yu-h\p_yb]_\Phi\big|\ep\Dhk\gp \bigr)_{L^2_+}+\bigl(\ep\Dhk[v\p_yb-h\p_yu]_\Phi\big|\ep\Dhk\hp \bigr)_{L^2_+}\\
&
+\bigl(\ep\f{y}\tt \int_y^\oo \Dhk[\p_yu\p_x\vphi-\p_yb\p_x\psi]_\Phi dy'\big|\ep\Dhk\gp \bigr)_{L^2_+}\\
&
+\bigl(\ep\Dhk[U\p_xG-\ka B\p_xH]_\Phi\big|\ep\chi'\Dhk\gp \bigr)_{L^2_+}\\ &
+\bigl(\ep\Dhk[U\p_xH-\f1\ka B\p_xG]_\Phi\big|\ep\chi'\Dhk\hp \bigr)_{L^2_+} \\
&
-(1-\ka)\bigl(\ep\Dhk[B\p_xb]_\Phi\big|\ep\chi'\Dhk\gp \bigr)_{L^2_+}-(1-\f1\ka)\bigl(\ep\Dhk[B\p_xu]_\Phi\big|\ep\chi'\Dhk\hp \bigr)_{L^2_+} \\
&
+\bigl(\ep\Dhk[Uv-Bh]_\Phi\big|\ep\chi''\Dhk\gp \bigr)_{L^2_+}+\bigl(\ep\Dhk[Bv-Uh]_\Phi\big|\ep\chi''\Dhk\hp \bigr)_{L^2_+} \\
&
+\f1\tt\bigl(\ep y\int_y^\oo\chi''\Dhk[U\p_x\vphi-B\p_x\psi]_\Phi dy'\big|\ep\Dhk\gp \bigr)_{L^2_+} \\
&
+\bigl(\ep\Dhk[-\p_xU\p_yu+\p_xB\p_yb]_\Phi\big|\ep\chi\Dhk\gp \bigr)_{L^2_+}\\
&
+\bigl(\ep\Dhk[-\p_xU\p_yb+\p_xB\p_yu]_\Phi\big|\ep\chi\Dhk\hp \bigr)_{L^2_+} \\
&
+\bigl(\ep\Dhk[\p_xUu-\p_xBb]_\Phi\big|\ep\chi'\Dhk\gp \bigr)_{L^2_+}+\bigl(\ep\Dhk[\p_xBu-\p_xUb]_\Phi\big|\ep\chi'\Dhk\hp \bigr)_{L^2_+} \\
&
+\f1{2\tt}\bigl(\ep\Dhk[-\p_xUu+\p_xBb]_\Phi\big|\ep y\chi\Dhk\gp \bigr)_{L^2_+}\\
&
+\f1{2\ka\tt}\bigl(\ep\Dhk[-\p_xUb+\p_xBu]_\Phi\big|\ep y\chi\Dhk\hp \bigr)_{L^2_+} \\
&
+\f1\tt\bigl(\ep\Dhk[\p_xU\vphi-\p_xB\psi]_\Phi\big|\ep y\chi'\Dhk\gp \bigr)_{L^2_+}\\
&
+\f1\tt\bigl(\ep\int_y^\oo \chi''\Dhk[\p_xU\vphi-\p_xB\psi]_\Phi dy'\big|\ep y\Dhk\gp \bigr)_{L^2_+}\\
&
=\bigl(\ep\Dhk[m_U+\fyt M_U]_\Phi\big|\ep \Dhk\gp \bigr)_{L^2_+}+\bigl(\ep\Dhk[m_B+\fykt M_B]_\Phi\big|\ep \Dhk\hp \bigr)_{L^2_+} .
\end{split}\eeq
In what follows, we shall  use the technical lemmas in Section \ref{sect4} to handle term by term in \eqref{S7eq4}.

Notice that for $\ka\in ]0,2[,$ $\ka(2-\ka)\leq1$, we get, by applying \eqref{S4eq1} and \eqref{S4eq2}, that
\beqo\begin{split}
\int_{t_0}^t&\hbar(t')\Bigl(\bigl(\ep\Dhk(\p_t-\p^2_y)\gp\big|\ep\Dhk\gp \bigr)_{L^2_+}+\bigl(\ep\Dhk(\p_t-\ka\p^2_y)\hp\big|\ep\Dhk\hp \bigr)_{L^2_+}\Bigr)\,dt'\\
&\geq \ \f12\bigl(\|\hsf \ep\Dhk\ghp(t)\|^2_{L^2_+}-\|\hsf \ep\Dhk\ghp(t_0)\|^2_{L^2_+}\bigr) \\
&\qquad-\f12 \|\sqrt{\hbar'}\ep\Dhk\ghp\|^2_{L^2_{[t_0,t]}(L^2_+)}
+2\lk\|\hsf \ep\Dhk\p_y\ghp\|^2_{L^2_{[t_0,t]}(L^2_+)}.
\end{split}\eeqo

We deduce from  Lemma \ref{lem:Bern} that
\beqo\begin{split}
\lam\int_{t_0}^t\hbar(t')\dtht(t')\Bigl(\bigl(\ep|D_x|\Dhk\gp\big|\ep\Dhk\gp \bigr)_{L^2_+}
&+\bigl(\ep|D_x|\Dhk\hp\big|\ep\Dhk\hp \bigr)_{L^2_+}\Bigr)\,dt'\\
&
\geq c\lam 2^k\int_{t_0}^t\dtht(t')\|\hsf\ep\Dhk\ghp \|^2_{L^2_+}\,dt'.
\end{split}\eeqo

Whereas due to $\lim_{y\rto+\oo}u=\lim_{y\rto+\oo}b=0$, we get, by applying Lemma \ref{lem:nl1} with $a=\f12$ and $b=c=d=1$, that
\beqo\begin{split}
&\int_{t_0}^t\hbar(t')\Bigl(\bigl|\bigl(\ep\Dhk[u\p_xG-b\p_xH]_\Phi\big|\ep\Dhk\gp \bigr)_{L^2_+}\bigr|+\bigl|\bigl(\ep\Dhk[u\p_xH-b\p_xG]_\Phi\big|\ep\Dhk\hp \bigr)_{L^2_+}\bigr|\Bigr)\,dt'\\
&\qquad\qquad\lesssim d_k^22^{-k}\|\hsf\ep\ghp\|_{\wt{L}^2_{[t_0,t];f}(\cB^{1,0})}^2 \with f(t)=\w{t}^{\f14}\|e^{\f{\Psi}2}\p_y(u,b)_\Phi(t)\|_{\cB^{\f12,0}},
\end{split}\eeqo from which, we get, by a similar derivation of \eqref{S5eq7}, that
\beq\begin{split}\label{S7eq7}
\int_{t_0}^t&\hbar(t')\Bigl(\bigl|\bigl(\ep\Dhk[u\p_xG-b\p_xH]_\Phi\big|\ep\Dhk\gp \bigr)_{L^2_+}\bigr|\\
&+\bigl|\bigl(\ep\Dhk[u\p_xH-b\p_xG]_\Phi\big|\ep\Dhk\hp \bigr)_{L^2_+}\bigr|\Bigr)\,dt'\\
\lesssim & d_k^22^{-k}\|\hsf\ep\ghp\|_{\wt{L}^2_{[t_0,t];\dtht}(\cB^{1,0})}^2.
\end{split}\eeq
Along the same line, by applying Lemma \ref{lem:nl1} with $a=\f12$, $b=\f34$ and $c=d=1$, and then using \eqref{S3eq16},
we obtain
\beqo\begin{split}
&\int_{t_0}^t\hbar(t')\Bigl((1-\ka)\bigl|\bigl(\ep\Dhk[b\p_xb]_\Phi\big|\ep\Dhk\gp \bigr)_{L^2_+}\bigr|+(1-\f1\ka)\bigl|\bigl(\ep\Dhk[b\p_xu]_\Phi\big|\ep\Dhk\hp \bigr)_{L^2_+}\bigr|\Bigr)\,dt'\\
&\lesssim d_k^22^{-k}\|\hsf e^{\f34 \Psi}\ubp\|_{\wt{L}^2_{[t_0,t];\dtht}(\cB^{1,0})}\|\hsf\ep\ghp\|_{\wt{L}^2_{[t_0,t];\dtht}(\cB^{1,0})}\\
&\lesssim d_k^22^{-k}\|\hsf\ep\ghp\|_{\wt{L}^2_{[t_0,t];\dtht}(\cB^{1,0})}^2.
\end{split}\eeqo
Similarly,  applying Lemma \ref{lem:nl2} and \eqref{S3eq16} gives rise to
\beqo\begin{split}
&\int_{t_0}^t\hbar(t')\Bigl(\bigl|\bigl(\ep\Dhk[v\p_yu-h\p_yb]_\Phi\big|\ep\Dhk\up \bigr)_{L^2_+}\bigr|+\bigl|\bigl(\ep\Dhk[v\p_yb-h\p_yu]_\Phi\big|\ep\Dhk\bp \bigr)_{L^2_+}\bigr|\Bigr)\,dt'\\
&\lesssim d_k^22^{-k}\|\hsf e^{\f34 \Psi}\ubp\|_{\wt{L}^2_{[t_0,t];\dtht}(\cB^{1,0})}\|\hsf\ep\ghp\|_{\wt{L}^2_{[t_0,t];\dtht}(\cB^{1,0})}\\
&\lesssim d_k^22^{-k}\|\hsf\ep\ghp\|_{\wt{L}^2_{[t_0,t];\dtht}(\cB^{1,0})}^2.
\end{split}\eeqo

Whereas by applying Lemma \ref{lem:nl3} with $a=\f34$, $b=\f34$, $c=\f54$ and $d=1$, and then using \eqref{S3eq15},
we find
\beqo\begin{split}
&\int_{t_0}^t\hbar(t')\bigl|\bigl(\ep {\langle t' \rangle}^{-1}y \int_y^\oo \Dhk[\p_yu\p_x\vphi-\p_yb\p_x\psi]_\Phi\big|\ep\Dhk\gp \bigr)_{L^2_+}\bigr|\,dt'\\
&\lesssim d_k^22^{-k}\bigl\|e^{-\f\Psi4} \ttau^{-\f12} y\bigr\|_{L^\oo_t(L^\oo_\rmv)}\|\ttau^{-\f12} \hsf e^{\f34 \Psi}\ppp\|_{\wt{L}^2_{[t_0,t];\dtht}(\cB^{1,0})}\\
&\qquad\qquad\qquad\qquad\qquad\qquad\qquad\qquad\qquad\times\|\hsf\ep\ghp\|_{\wt{L}^2_{[t_0,t];\dtht}(\cB^{1,0})}\\
&\lesssim d_k^22^{-k}\|\hsf\ep\ghp\|_{\wt{L}^2_{[t_0,t];\dtht}(\cB^{1,0})}^2.
\end{split}\eeqo

While it follows  Lemma \ref{ff1} that
\beqo \begin{split}
&\int_{t_0}^t\hbar(t')\Bigl(\bigl|\bigl(\ep\Dhk[U\p_xG-\ka B\p_xH]_\Phi\big|\ep\chi'\Dhk\gp\bigr)_{L^2_+}\bigr|\\
&
\ \ +\bigl|\bigl(\ep\Dhk[U\p_xH-\f1\ka B\p_xG]_\Phi\big|\ep\chi'\Dhk\hp\bigr)_{L^2_+}\bigr|\Bigr)\,dt'
\lesssim \e^{\f12}d_k^22^{-k}\|\hsf\ep\ghp\|_{\wt{L}^2_{[t_0,t];\dtht}(\cB^{1,0})}^2.
\end{split}\eeqo
Similarly, we get, by applying Lemma \ref{ff1} and \eqref{def:theta}, that
\beqo \begin{split}
&\int_{t_0}^t\hbar(t')\Bigl((1-\ka)\bigl|\bigl(\ep\Dhk[B\p_xb]_\Phi\big|\ep\chi'\Dhk\gp\bigr)_{L^2_+}\bigr|\\
&\qquad+(1-\f1\ka)\bigl|\bigl(\ep\Dhk[B\p_xu]_\Phi\big|\ep\chi'\Dhk\hp\bigr)_{L^2_+}\bigr|\Bigr)\,dt'\\
&\lesssim \e^{\f12} d_k^22^{-k}\|\ttau^{-\f54}\hsf\ep\ubp\|_{\wt{L}^2_{[t_0,t];\dtht}(\cB^{1,0})}\|\hsf\ep\ghp\|_{\wt{L}^2_{[t_0,t];\dtht}(\cB^{1,0})}\\
&\lesssim d^2_k 2^{-k}\bigl( \e\|\ttau^{-\f54}\hsf\ep\ubp\|^2_{\wt{L}^2_{[t_0,t];\dtht}(\cB^{1,0})}+ \|\hsf\ep\ghp\|^2_{\wt{L}^2_{[t_0,t];\dtht}(\cB^{1,0})} \bigr).
\end{split}\eeqo

 Notice from the definition of $\chi$ above \eqref{change} that $\mbox{supp}\chi''\subset [1,2]$. Then
 in view of  \eqref{S6eq8}, we
get, by a similar derivation of \eqref{s8eq6}, that
\beqo \begin{split}
&\int_{t_0}^t\hbar(t')\Bigl(\bigl|\bigl(\ep\Dhk[Uv-Bh]_\Phi\big|\ep\chi''\Dhk\gp\bigr)_{L^2_+}\bigr|
+\bigl|\bigl(\ep\Dhk[Bv-Uh]_\Phi\big|\ep\chi''\Dhk\hp\bigr)_{L^2_+}\bigr|\Bigr)\,dt'\\
&\lesssim d_k^22^{-k} \|\ep \|_{L^\oo_{0\leq y\leq 2}}\|\hsf\ubp\|_{\wt{L}^2_{[t_0,t];\dtht}(\cB^{1,0})}\|\hsf\ep\ghp\|_{\wt{L}^2_{[t_0,t];\dtht}(\cB^{1,0})}\\
&\lesssim d_k^22^{-k}\|\hsf\ep\ghp\|_{\wt{L}^2_{[t_0,t];\dtht}(\cB^{1,0})}^2.
\end{split}\eeqo
Moreover, again due to $\mbox{supp}\chi''\subset[1,2]$, we deduce that the term
$
\int_{t_0}^t\hbar(t')\bigl|\f1\tt \bigl(\ep y\int_y^\infty \chi''\Dhk[U\p_x\vphi-B\p_x\psi]_\Phi dy'\big|\ep\Dhk\gp\bigr)_{L^2_+}\bigr|\,dt'$ shares
the same estimate as above.

On the other hand, we  observe from  \eqref{S8eq34aq}  that
\beno
\|\ep y\Dhk\ghp\|_{L^2_+}\leq  4\w{t}\|\ep\p_y\Dhk\ghp\|_{L^2_+},
\eeno
which implies
\beno
\|\ep y\ghp\|_{\wt{L}^p_t(\cB^{s,0})}\leq 4\|\ttau\ep\p_y\ghp\|_{\wt{L}^p_t(\cB^{s,0})}.
\eeno
So that
 by applying \eqref{s8eq10} and  \eqref{UBdecay}, we find
\beqo \begin{split}
\int_{t_0}^t&\hbar(t')\Bigl(\bigl|\bigl(\ep\Dhk[-\p_xU\p_yu+\p_xB\p_yb]_\Phi\big|\ep\chi\Dhk\gp\bigr)_{L^2_+}\bigr|\\
&\qquad+\bigl|\bigl(\ep\Dhk[-\p_xU\p_yb+\p_xB\p_yu]_\Phi\big|\ep\chi\Dhk\hp\bigr)_{L^2_+}\bigr|\Bigr)\,dt\\
\lesssim \ & d_k^2 2^{-{k}}\|\ttau^\f94\UBp\|_{\wt{L}^\oo_{[t_0,t]}(\cB^{\f32}_\h)}\|\ttau^{-\f54}\hsf\ep\p_y\ubp\|_{\wt{L}^2_{[t_0,t]}(\cB^{\f12,0})}\\
&\qquad\qquad\qquad\qquad\qquad\qquad\qquad\qquad\times\|\ttau^{-1}\hsf\ep y\ghp\|_{\wt{L}^2_{[t_0,t]}(\cB^{\f12,0})}\\
\lesssim \ &\e d^2_k 2^{-k}\Bigl(\|\ttau^{-\f54}\hsf\ep\p_y\ubp\|_{\wt{L}^2_{[t_0,t]}(\cB^{\f12,0})}^2+\|\hsf\ep\p_y\ghp\|^2_{\wt{L}^2_{[t_0,t]}(\cB^{\f12,0})}\Bigr).
\end{split}\eeqo
Along the same line, we infer
\beqo \begin{split}
\int_{t_0}^t&\hbar(t')\Bigl(\bigl|\bigl(\ep\Dhk[\p_xUu-\p_xBb]_\Phi\big|\ep\chi'\Dhk\gp\bigr)_{L^2_+}\bigr|\\
&\qquad+\bigl|\bigl(\ep\Dhk[\p_xBu-\p_xUb]_\Phi\big|\ep\chi'\Dhk\hp\bigr)_{L^2_+}\bigr|\Bigr)\,dt'\\
\lesssim \ & d_k^22^{-{k}}\|\ttau^\f94 \UBp \|_{\wt{L}^\oo_{[t_0,t]}(\cB^{\f32}_\h)}\|\ttau^{-\f74}\hsf\ep\ubp \|_{\wt{L}^2_{[t_0,t]}(\cB^{\f12,0})}\\
&\qquad\qquad\qquad\qquad\qquad\qquad\qquad\qquad\times
\|\ttau^{-\f12}\hsf\ep \ghp\|_{\wt{L}^2_{[t_0,t]}(\cB^{\f12,0})}\\
\lesssim \ &\e d^2_k 2^{-k}\Bigl(\|\ttau^{-\f54}\hsf\ep\p_y\ubp\|_{\wt{L}^2_{[t_0,t]}(\cB^{\f12,0})}^2+\|\hsf\ep\p_y\ghp\|^2_{\wt{L}^2_{[t_0,t]}(\cB^{\f12,0})}\Bigr).
\end{split}\eeqo
and
\beqo \begin{split}
\int_{t_0}^t&\hbar(t')\Bigl(\f1{2{\langle t'\rangle}}\bigl|\bigl(\ep y\Dhk[-\p_xUu+\p_xBb]_\Phi\big|\ep\chi\Dhk\gp\bigr)_{L^2_+}\bigr|\\
&\qquad+\f1{2\ka\langle t'\rangle}\bigl|\bigl(\ep\Dhk[-\p_xUb+\p_xBu]_\Phi\big|\ep\chi\Dhk\hp\bigr)_{L^2_+}\bigr|\Bigr)\,dt'\\
\lesssim \ & d_k^22^{-{k}}\|\ttau^\f94 \UBp \|_{\wt{L}^\oo_{[t_0,t]}(\cB^{\f32})}\|\ttau^{-\f94}\hsf\ep y\ubp \|_{\wt{L}^2_{[t_0,t]}(\cB^{\f12,0})}\\
&\qquad\qquad\qquad\qquad\qquad\qquad\qquad\qquad\times\|\ttau^{-1}\hsf\ep y\ghp\|_{\wt{L}^2_{[t_0,t]}(\cB^{\f12,0})}\\
\lesssim \ &\e d^2_k 2^{-k}\Bigl(\|\ttau^{-\f54}\hsf\ep\p_y\ubp\|_{\wt{L}^2_{[t_0,t]}(\cB^{\f12,0})}^2+\|\hsf\ep\p_y\ghp\|^2_{\wt{L}^2_{[t_0,t]}(\cB^{\f12,0})}\Bigr).
\end{split}\eeqo
and
\beqo \begin{split}
\int_{t_0}^t&\hbar(t')\f1{\langle t'\rangle}\bigl|\bigl(\ep \Dhk[\p_xU\vphi-\p_xB\psi]_\Phi\big|\ep y\chi'\Dhk\gp\bigr)_{L^2_+}\bigr|\,dt'\\
\lesssim \ & d_k^22^{-{k}}\|\ttau^\f94 \UBp \|_{\wt{L}^\oo_{[t_0,t]}(\cB^{\f32})}\|\ttau^{-\f94}\hsf\ep \ppp \|_{\wt{L}^2_{[t_0,t]}(\cB^{\f12,0})}\\
&\qquad\qquad\qquad\qquad\qquad\qquad\qquad\qquad\times\|\ttau^{-1}\hsf\ep y\ghp\|_{\wt{L}^2_{[t_0,t]}(\cB^{\f12,0})}\\
\lesssim \ &\e d^2_k 2^{-k}\Bigl(\|\ttau^{-\f54}\hsf\ep\p_y\ubp\|_{\wt{L}^2_{[t_0,t]}(\cB^{\f12,0})}^2+\|\hsf\ep\p_y\ghp\|^2_{\wt{L}^2_{[t_0,t]}(\cB^{\f12,0})}\Bigr).
\end{split}\eeqo
Furthermore, due to $\mbox{supp}\chi''\subset[1,2]$,, we deduce that the term
$
\int_{t_0}^t\hbar(t')\bigl|{\langle t'\rangle}^{-1} \bigl(\ep \int_y^\infty \chi''\Dhk[\p_xU\vphi-\p_xB\psi]_\Phi dy'\big|\ep y\Dhk\gp\bigr)_{L^2_+}\bigr|\,dt'$ shares
the same estimate as above.

Finally for the remained source term, by Young's inequality, we achieve
\beqo \label{S7eq9}\begin{split}
\int_{t_0}^t&\hbar(t')\Bigl(\bigl|\bigl(\ep\Dhk(m_U+\f{y}{2\langle t'\rangle} M_U)_\Phi\big|\ep\Dhk\gp\bigr)_{L^2_+}\bigr|\\
&+\bigl|\bigl(\ep\Dhk(m_B+\f{y}{2\ka\langle t'\rangle} M_B)_\Phi\big|\ep\Dhk\hp\bigr)_{L^2_+}\bigr|\Bigr)\,dt'\\
\lesssim \ & d^2_k2^{-k} \|\ttau^{\f12}\hsf\ep\mMp\|_{\wt{L}^2_{[t_0,t]}(\cB^{\f12,0})}\|\ttau^{-\f12}\hsf\ep\ghp\|_{\wt{L}^2_{[t_0,t]}(\cB^{\f12,0})}\\
\lesssim \ & d^2_k2^{-k}\bigl( \e \|\hsf\ep\p_y\ghp\|^2_{\wt{L}^2_{[t_0,t]}(\cB^{\f12,0})}+\e^{-1} \|\ttau^\f12\hsf\ep\mMp\|^2_{\wt{L}^2_{[t_0,t]}(\cB^{\f12,0})}\bigr)
\end{split}\eeqo

By multiplying \eqref{S7eq4} by $\hbar(t)$ and then integrating the resulting equality over $[t_0,t],$
and finally inserting the above estimates to the equality, we
 obtain \eqref{S7eq1}. This completes the proof of Lemma \ref{prop7.1}.
\end{proof}

\begin{lemma}\label{prop7.2}
{\sl Let $(u,b)$ and $(\vphi,\psi)$ be smooth enough solutions of \eqref{eqs1} and \eqref{eqs2} respectively, and $(G,H)$ be determined by \eqref{defGH}. Let $\Phi_\ka(t,\xi)$ and $\Psi_\ka(t,y)$ be given by \eqref{def:Phikappa} and \eqref{def:Psikappa} respectively. Then if $\ka\in ]1/2,\oo[,$ for $\elk\eqdefa \f{2\ka-1}{4\ka^2},$ and for any non-negative and non-decreasing function $\hs\in C^1(\R_+)$, there exists a
positive constant $\fc$ so that
\beq\begin{split}\label{S7eq13}
&\| \hsf \epk\Dhk \ghpk\|^2_{L^\oo_{[t_0,t]}(L^2_+)} +4\ka\elk\| \hsf \epk\Dhk\p_y \ghpk\|^2_{L^2_{[t_0,t]}(L^2_+)}\\
&\ +2\|\ttau^{-\f12} \hsf \epk\Dhk  \ghpk\|^2_{L^2_{[t_0,t]}(L^2_+)} +2c\lam 2^k \|\hsf\epk\Dhk\ghpk \|^2_{L^2_{[t_0,t];\dthtk}(L^2_+)}
\\ &\leq \| \hsf \epk\Dhk \ghpk(t_0)\|^2_{L^2_+}+\| \qhs \epk\Dhk \ghpk\|^2_{L^2_{[t_0,t]}(L^2_+)}\\
 &\ +\fc\e d^2_k 2^{-k}\|\hsf\epk \p_y\ghpk \|^2_{\wt{L}^2_{[t_0,t]}(\cB^{\f12,0})}+C d^2_k 2^{-k}\Bigl( \|\hsf\epk\ghpk \|^2_{\wt{L}^2_{[t_0,t];\dthtk}(\cB^{1,0})}\\
 &\qquad\qquad +\e \|\ttau^{-\f54}\hsf\epk \p_y\ubpk \|^2_{\wt{L}^2_{[t_0,t]}(\cB^{\f12,0})}+\e\|\ttau^{-\f54}\hsf\epk\ubpk \|^2_{\wt{L}_{[t_0,t];\dthtk}(\cB^{1,0})} \\
 &\qquad\qquad\qquad\qquad\qquad\qquad\qquad+\e^{-1} \|\ttau^\f12 \hsf \epk \mMpk \|^2_{\wt{L}^2_{[t_0,t]}(\cB^{\f12,0})}\Bigr),
\end{split}\eeq
for any $t_0\in[0,t]$ with $t<T^*_\ka$, which is defined by \eqref{def:T^*kappa}.}
\end{lemma}

\begin{proof}
The proof of Lemma \ref{prop7.2} is almost the same as that of Lemma \ref{prop7.1}.
 Indeed \eqref{S7eq4} holds with $\tht, \Psi$ and $\Phi$ there being replaced by $\tht_\ka, \Psi_\ka$ and $\Phi_\ka.$
 In what follows, we just present the estimates to the  terms with the estimates of which are different from that of Lemma \ref{prop7.1}.

We first observe that due to $\ka\in ]1/2,\oo[,$
 $2\ka-1\leq\ka^2.$ Then we get, by applying \eqref{S4eq3} and \eqref{S4eq4}, that
\beqo\begin{split}
\int_{t_0}^t&\hbar(t')\Bigl(\bigl(\epk\Dhk(\p_t-\p^2_y)\gpk\big|\epk\Dhk\gpk \bigr)_{L^2_+}
+\bigl(\epk\Dhk(\p_t-\ka\p^2_y)\hpk\big|\epk\Dhk\hpk \bigr)_{L^2_+}\Bigr)\,dt'\\
&\geq \ \f12\bigl(\|\hsf \epk\Dhk\ghpk(t)\|^2_{L^2_+}-\|\hsf \epk\Dhk\ghpk(t_0)\|^2_{L^2_+}\bigr) \\
&\qquad-\f12 \|\sqrt{\hbar'}\epk\Dhk\ghpk\|^2_{L^2_{[t_0,t]}(L^2_+)}
+2\ka\elk\|\hsf \epk\Dhk\p_y\ghpk\|^2_{L^2_{[t_0,t]}(L^2_+)}.
\end{split}\eeqo

Whereas due to $\lim_{y\rto+\oo}u=\lim_{y\rto+\oo}b=0$, we get, by applying Lemma \ref{lem:nl1} with $a=\f1{2\ka}$ and $b=c=d=\f1\ka$, that
\beqo\begin{split}
&\int_{t_0}^t\hbar(t')\Bigl(\bigl|\bigl(\epk\Dhk[u\p_xG-b\p_xH]_{\Phi_\ka} \big|\epk\Dhk\gpk \bigr)_{L^2_+}\bigr|\\
&\qquad\qquad+\bigl|\bigl(\epk\Dhk[u\p_xH-b\p_xG]_{\Phi_\ka} \big| \epk\Dhk\hpk \bigr)_{L^2_+}\bigr|\Bigr)\,dt'\\
&\lesssim d_k^22^{-k}\|\hsf\epk\ghpk\|_{\wt{L}^2_{[t_0,t];f}(\cB^{1,0})}^2 \with f(t)=\w{t}^{\f14}\|e^{\f12{\Psi_\ka}}\p_y(u,b)_{\Phi_\ka}(t)\|_{\cB^{\f12,0}},
\end{split}\eeqo
from which, we get, by a similar derivation of \eqref{S5eq17}, that
\beq\begin{split}\label{S7eq17}
&\int_{t_0}^t\hbar(t')\Bigl(\bigl|\bigl(\epk\Dhk[u\p_xG-b\p_xH]_{\Phi_\ka}\big|\epk\Dhk\gpk \bigr)_{L^2_+}\bigr|\\
&\qquad\qquad+\bigl|\bigl(\epk\Dhk[u\p_xH-b\p_xG]_{\Phi_\ka}\big|\epk\Dhk\hpk \bigr)_{L^2_+}\bigr|\Bigr)\,dt'\\
&\lesssim d_k^22^{-k}\|\hsf\epk\ghpk\|_{\wt{L}^2_{[t_0,t];\dthtk}(\cB^{1,0})}^2.
\end{split}\eeq
Along the same line, by applying Lemma \ref{lem:nl1} with $a=\f1{2\ka}$, $b=\f3{4\ka}$ and $c=d=\f1\ka$, we get, by a similar derivation of \eqref{S7eq17}, that
\beqo\begin{split}
&\int_{t_0}^t\hbar(t')\Bigl((1-\ka)\bigl|\bigl(\epk\Dhk[b\p_xb]_{\Phi_\ka}\big|\epk\Dhk\gpk \bigr)_{L^2_+}\bigr|\\
&\qquad\qquad+(1-\f1\ka)\bigl|\bigl(\epk\Dhk[b\p_xu]_{\Phi_\ka}\big|\epk\Dhk\hpk \bigr)_{L^2_+}\bigr|\Bigr)\,dt'\\
&\lesssim d_k^22^{-k}\|\hsf e^{\f34 \Psi_\ka}\ubpk\|_{\wt{L}^2_{[t_0,t];\dthtk}(\cB^{1,0})}\|\hsf\epk\ghpk\|_{\wt{L}^2_{[t_0,t];\dthtk}(\cB^{1,0})}\\
&\lesssim d_k^22^{-k}\|\hsf\epk\ghpk\|_{\wt{L}^2_{[t_0,t];\dthtk}(\cB^{1,0})}^2.
\end{split}\eeqo
Similarly, by applying Lemma \ref{lem:nl3} with $a=\f3{4\ka}$, $b=\f3{4\ka}$, $c=\f5{4\ka}$ and $d=\f1\ka$, we achieve
\beqo\begin{split}
&\int_{t_0}^t\hbar(t')\bigl|\bigl(\epk{\langle t'\rangle}^{-1}y \int_y^\oo \Dhk[\p_yu\p_x\vphi-\p_yb\p_x\psi]_{\Phi_\ka}
\big| \epk\Dhk\gpk \bigr)_{L^2_+}\bigr|\,dt'\\
&\lesssim d_k^22^{-k}\|e^{-\f{\Psi_\ka}4} \ttau^{-\f12}y \|_{L^\oo_t(L^\oo_\rmv)}\|\ttau^{-\f12} \hsf e^{\f34 \Psi_\ka}\pppk\|_{\wt{L}^2_{[t_0,t];\dthtk(t)}(\cB^{1,0})}\\
&\qquad\qquad\qquad\qquad\qquad\qquad\qquad\qquad\times\|\hsf\epk\ghpk\|_{\wt{L}^2_{[t_0,t];\dthtk}(\cB^{1,0})}\\
&\lesssim d_k^22^{-k}\|\hsf\epk\ghpk\|_{\wt{L}^2_{[t_0,t];\dthtk}(\cB^{1,0})}^2.
\end{split}\eeqo

On the other hand, we deduce from  \eqref{S8eq33aq}  that
\beno
\|\epk y\Dhk\ghpk\|_{L^2_+}\lesssim \ka\w{t}\|\epk\p_y\Dhk\ghpk\|_{L^2_+},
\eeno
which implies
\beno
\|\epk y\ghpk\|_{\wt{L}^p_t(\cB^{s,0})}\lesssim \ka\|\ttau\epk\p_y\ghpk\|_{\wt{L}^p_t(\cB^{s,0})}.
\eeno
So that
 by applying \eqref{s8eq10} and  \eqref{UBdecay}, we find
\beqo \begin{split}
\int_{t_0}^t&\hbar(t')\Bigl(\bigl|\bigl(\epk\Dhk[-\p_xU\p_yu+\p_xB\p_yb]_{\Phi_\ka} \big|\epk\chi\Dhk\gpk\bigr)_{L^2_+}\bigr|\\
&\qquad+\bigl|\bigl(\epk\Dhk[-\p_xU\p_yb+\p_xB\p_yu]_{\Phi_\ka} \big| \epk\chi\Dhk\hpk\bigr)_{L^2_+}\bigr|\Bigr)\,dt\\
\lesssim \ & d_k^2 2^{-{k}}\|\ttau^\f94\UBpk|_{\wt{L}^\oo_{[t_0,t]}(\cB^{\f32}_\h)}\|\ttau^{-\f54}\hsf\epk\p_y\ubpk\|_{\wt{L}^2_{[t_0,t]}(\cB^{\f12,0})}\\
&\qquad\qquad\qquad\qquad\qquad\qquad\qquad\qquad\times\|\ttau^{-1}\hsf\epk y\ghpk\|_{\wt{L}^2_{[t_0,t]}(\cB^{\f12,0})}\\
\lesssim \ &\ka\e d^2_k 2^{-k}\Bigl(\|\ttau^{-\f54}\hsf\epk\p_y\ubpk\|_{\wt{L}^2_{[t_0,t]}(\cB^{\f12,0})}^2+\|\hsf\epk\p_y\ghpk\|^2_{\wt{L}^2_{[t_0,t]}(\cB^{\f12,0})}\Bigr).
\end{split}\eeqo

With the above estimates, we can repeat the proof of \eqref{S7eq1} to complete the proof of
\eqref{S7eq13}.
  This completes the proof of Lemma \ref{prop7.2}.
\end{proof}

An immediate application of Lemmas \ref{prop7.1} and \ref{prop7.2} leads to

\begin{col}\label{col7.1}
{\sl Under the assumption of Lemma \ref{prop7.1}, there exist positive constants $\e_0, \lam_0$ so that for any $\lam\geq\lam_0$ and $\e\leq \e_0,$ one has
\beq\begin{split}\label{S7eq23}
&\| \hsf\ep \ghp\|_{\wt{L}^\oo_{[t_0,t]} (\cB^{\f12,0})} +\sqrt{l_\ka}\|\hsf\ep\p_y\ghp\|_{\wt{L}^2_{[t_0,t]} (\cB^{\f12,0})}\\
&\leq\  \|\hsf\ep\ghp(t_0)\|_{\cB^{\f12,0}}+\|\qhs \ep \ghp\|_{\wt{L}^2_{[t_0,t]}(\cB^{\f12,0})}\\ &\quad+C\Bigl(\e^{\f12}\|\ttau^{-\f54}\hsf\ep\p_y\ubp\|_{\wt{L}^2_{[t_0,t]}(\cB^{\f12,0})}+\e^{\f12}\|\ttau^{-\f54}\hsf\ep\ubp\|_{\wt{L}^2_{[t_0,t];\dtht}(\cB^{1,0})}
\\
&\quad\qquad\qquad\qquad\qquad\qquad\qquad+\e^{-\f12}\|\ttau^\f12\hsf\ep\mMp \|_{\wt{L}^2_{[t_0,t]}(\cB^{\f12,0})}\Bigr).
\end{split}\eeq}
\end{col}
\begin{proof}
By taking square root of \eqref{S7eq1} and then multiplying the resulting inequality by $2^\f{k}2$ and finally summing over $k\in\Z$, we find
\beqo\begin{split}
&\|\hsf\ep\ghp\|_{\wt{L}^\oo_{[t_0,t]} (\cB^{\f12,0})} +2\sqrt{l_\ka}\|\hsf\ep\p_y\ghp\|_{\wt{L}^2_{[t_0,t]} (\cB^{\f12,0})}\\
&+\sqrt{2c\lam}\|\hsf\ep\ghp\|_{\wt{L}^2_{[t_0,t];\dtht} (\cB^{1,0})}\leq \|\hsf\ep\ghp(t_0)\|_{\cB^{\f12,0}}\\
&+\|\qhs \ep\ghp\|_{\wt{L}^2_{[t_0,t]}(\cB^{\f12,0})}+\sqrt{\fc\e}\|\hsf\ep\p_y\ghp\|_{\wt{L}^2_{[t_0,t]}(\cB^{\f12,0})}\\
&+C\Bigl(\e^{\f12}\|\ttau^{-\f54}\hsf\ep\p_y\ubp\|_{\wt{L}^2_{[t_0,t]}(\cB^{\f12,0})}+\e^{\f12}\|\ttau^{-\f54}\hsf\ep\ubp\|_{\wt{L}^2_{[t_0,t];\dtht}(\cB^{1,0})}\\
&
+\|\hsf\ep\ghp\|_{\wt{L}^2_{[t_0,t];\dtht} (\cB^{1,0})}
+\e^{-\f12}\|\ttau^\f12\hsf\ep\mMp \|_{\wt{L}^2_{[t_0,t]}(\cB^{\f12,0})}\Bigr).
\end{split}\eeqo
Then by taking $\lam$ big enough and $\e\leq \e_0$ small enough,  we arrive at \eqref{S7eq23}.
\end{proof}

\begin{col}\label{col7.2}
{\sl Under the assumption of Lemma \ref{prop6.2}, there exist positive constants $\e_0, \lam_0$ so that for any $\lam\geq\lam_0$ and $\e\leq \e_0,$ one has
\beq\begin{split}\label{S7eq24}
&\| \hsf\epk \ghpk\|_{\wt{L}^\oo_{[t_0,t]} (\cB^{\f12,0})} +\sqrt{\ka\elk}\|\hsf\epk\p_y\ghpk\|_{\wt{L}^2_{[t_0,t]} (\cB^{\f12,0})}\\
&\leq\ \|\hsf\epk\ghpk(t_0)\|_{\cB^{\f12,0}}+\|\qhs \epk \ghpk\|_{\wt{L}^2_{[t_0,t]}(\cB^{\f12,0})}\\ &+C\Bigl({\e}^{\f12}\|\ttau^{-\f54}\hsf\epk\p_y\ubpk\|_{\wt{L}^2_{[t_0,t]}(\cB^{\f12,0})}+\e^{\f12}\|\ttau^{-\f54}\hsf\epk\ubpk\|_{\wt{L}^2_{[t_0,t];\dthtk}(\cB^{1,0})}
\\
&\qquad\qquad\qquad\qquad\qquad\qquad\quad+\e^{-\f12}\|\ttau^\f12\hsf\epk\mMpk \|_{\wt{L}^2_{[t_0,t]}(\cB^{\f12,0})}\Bigr).
\end{split}\eeq}
\end{col}
\begin{proof}
The proof of this corollary is the same as that of Corollary \ref{col7.1}, we omit the details here.
\end{proof}

Now we are in a position to prove Proposition \ref{prop3.3} and \ref{prop3.6}.

\begin{proof}[Proof of Proposition \ref{prop3.3}]
We first observe from \eqref{S6eq21} that for $\hs(t)\leq C\tt^{\f52}$,
\beq \label{S7eq43}
\begin{split}
\|\ttau^{-\f54}\hsf \ep\p_y\ubp\|_{\wt{L}^2_t (\cB^{\f12,0})}&+\sqrt{c\lam}\|\ttau^{-\f54}\hsf \ep\ubp\|_{\wt{L}^2_{t;\dtht} (\cB^{1,0})}\\
&\qquad\leq \|e^{\f{y^2}8}e^{\de|D_x|}(u_0,b_0)\|_{\cB^{\f12,0}}+ C\sqrt{\e}.
\end{split}
\eeq
Whereas it  follows from Lemma \ref{lem:Poincare} that
\beqo
\|\ep\p_y\Dhk\ghp\|^2_{L^2_+}\geq\f1{2\tt}\|\ep\Dhk\ghp\|^2_{L^2_+}.
\eeqo
Then by taking $\hs(t)=\tt^{2+2\lk-2c\e}$ and $t_0=0$ in Lemma \ref{prop7.1}, we find
\beqo\begin{split}
&\|\taullk\ep\Dhk\ghp\|^2_{L^\oo_t (L^2_+)}+4\fc\e\|\taullk\ep\p_y\Dhk\ghp \|_{L^2_t (L^2_+)}^2\\
&+2c\lam 2^k\|\taullk\ep\Dhk\ghp\|^2_{L^2_{t;\dtht}(L^2_+)}\leq \|e^\f{y^2}8 e^{\de|D_x|}\Dhk(G_0,H_0)\|^2_{L^2_+}\\
&+\fc\e d^2_k 2^{-k}\|\taullk\ep\p_y\ghp \|_{\wt{L}^2_t (\cB^{\f12,0})}^2
+ Cd^2_k 2^{-k}\Bigl(\e\|\ep\p_y\ubp \|_{\wt{L}^2_t (\cB^{\f12,0})}^2\\
&\qquad\qquad+ \e\|\ep\ubp\|^2_{\wt{L}^2_{t;\dtht}(\cB^{1,0})}+ \|\taullk\ep\ghp\|^2_{\wt{L}^2_{t;\dtht}(\cB^{1,0})}\\
&\qquad\qquad\qquad\qquad\qquad+\e^{-1}\|\ttau^{\f32+\lk-c\e}\ep\mMp\|_{\wt{L}^2_{t}(\cB^{\f12,0})}^2\Bigr).
\end{split}\eeqo
Taking square root of the above inequality and then multiplying it by $2^\f{k}2$ and finally summing up the resulting one over $k\in\Z$ gives rise to
\beqo\begin{split}
&\|\taullk\ep\ghp\|_{\wt{L}^\oo_t (\cB^{\f12,0})}+\sqrt{2c\lam} \|\taullk\ep\ghp\|_{\wt{L}^2_{t;\dtht}(\cB^{1,0})}\\
&\leq  \|e^\f{y^2}8 e^{\de|D_x|}(G_0,H_0)\|_{\cB^{\f12,0}}+C\Bigl( \|\taullk\ep\ghp\|_{\wt{L}^2_{t;\dtht}(\cB^{1,0})}\\
&\quad+\e^{\f12}\|\ep\ubp\|_{\wt{L}^2_{t;\dtht}(\cB^{1,0})}\quad+\e^{\f12}\|\ep\p_y\ubp \|_{\wt{L}^2_t (\cB^{\f12,0})}\\
&\qquad\qquad\qquad\quad+\e^{-\f12}\|\ttau^{\f32+\lk-c\e}\ep\mMp\|_{\wt{L}^2_{t}(\cB^{\f12,0})}\Bigr).
\end{split}\eeqo
Then thanks to \eqref{smalldata2}, \eqref{mdecay} and \eqref{S7eq43}, we get,  by taking $\lam$ large enough, that
\beq \label{S7eq25}
\begin{split}
\|\taullk\ep\ghp\|_{\wt{L}^\oo_t (\cB^{\f12,0})}
\lesssim &\sqrt{\e}\bigl(\|e^{\f{y^2}8}e^{\de|D_x|}(u_0,b_0)\|_{\cB^{\f12,0}}+1\bigr).
\end{split}
\eeq

Along the same line, thanks to \eqref{mdecay} and \eqref{S7eq43}, we get, by taking $\hs(t)=1$ and $t_0=0$ in Corollary \ref{col7.1}, that
\beq\label{S7eq26}
\|\ep\p_y\ghp\|_{\wt{L}^2_t (\cB^{\f12,0})}\lesssim \sqrt{\e}\bigl(\|e^{\f{y^2}8}e^{\de|D_x|}(u_0,b_0)\|_{\cB^{\f12,0}}+1\bigr).
\eeq
Similarly, by taking $\hs(t)=\tt^{2+2\lk-2c\e}$ and $t_0=\f{t}2$ in Corollary \ref{col7.1},  we achieve
\beqo \begin{split}
\|\taullk\ep\p_y&\ghp\|_{\wt{L}^2_{[\f{t}2,t]}(\cB^{\f12,0})}\lesssim\e^{\f12}\|e^{\f{y^2}8}e^{\de|D_x|}(u_0,b_0)\|_{\cB^{\f12,0}} + \sqrt{\e}\\
&+ \bigl\|\langle {t}/2\rangle^{1+\lk}\ep\ghp({t}/2)\bigr\|_{\cB^{\f12,0}}+ \|\taulk\ep\ghp\|_{\wt{L}^2_{[\f{t}2,t]}(\cB^{\f12,0})}.
\end{split}
\eeqo
Notice that
\beqo
\|\taulk\ep\ghp\|_{\wt{L}^2_{[\f{t}2,t]}(\cB^{\f12,0})}\lesssim \|\taullk\ep\ghp\|_{\wt{L}^\oo_{[\f{t}2,t]}(\cB^{\f12,0})}.
\eeqo
As a consequence, we deduce from \eqref{S7eq25} that
\beq \label{S7eq27}
\|\taullk\ep\p_y\ghp\|_{\wt{L}^2_{[\f{t}2,t]} (\cB^{\f12,0})}\lesssim \sqrt{\e}\bigl(\|e^{\f{y^2}8}e^{\de|D_x|}(u_0,b_0)\|_{\cB^{\f12,0}}+1\bigr).
\eeq

With \eqref{S7eq25} and \eqref{S7eq27}, to finish the proof of Proposition \ref{prop3.3}, it remains to show that for any $t<T^*$,
\beq \label{S7eq28}
\int_0^t \ttau^\f14\|\ep\p_y\ghp\|_{\cB^{\f12,0}}d\tau\lesssim \sqrt{\e}\bigl(\|e^{\f{y^2}8}e^{\de|D_x|}(u_0,b_0)\|_{\cB^{\f12,0}}+1\bigr).
\eeq
Indeed taking $\e_0$ to be so small that $\fc\e_0<\lk.$ Then for $\e\leq\e_0,$ we deduce from
 \eqref{S7eq27} that for $1\leq t<T^*$
\beq \begin{split}\label{S7eq29}
\int_{\f{t}2}^t &\ttau^\f14\|\ep\p_y\ghp\|_{\cB^{\f12,0}}d\tau\\ &\leq \|\ttau^{-\f34-\lk+\fc\e}\|_{L^2_{[\f{t}2,t]}}\|\taullk\ep\p_y\ghp\|_{\wt{L}^2_{[\f{t}2,t]}( \cB^{\f12,0})}\\
&\lesssim t^{-\f14-\lk+\fc\e}\sqrt{\e}\bigl(\|e^{\f{y^2}8}e^{\de|D_x|}(u_0,b_0)\|_{\cB^{\f12,0}}+1\bigr).
\end{split}\eeq
By dividing the time interval $[0,t]$ into $[0,1],[1,2],[2,4],\cdots,[2^N,t]$, and using \eqref{S7eq26} and \eqref{S7eq29},
 we arrive at
\beqo \begin{split}
\int_0^t \ttau^\f14\|\ep\p_y\ghp\|_{\cB^{\f12,0}}d\tau
&\leq\bigl(\int_0^1+\sum_{j=1}^N \int_{2^{j-1}}^{2^j}+\int_{2^j}^t\bigr) \ttau^\f14\|\ep\p_y\ghp\|_{\cB^{\f12,0}}d\tau \\
&\lesssim\Bigl( 1+\sum_{j=1}^N 2^{-\f{j}4}+t^{-\f14}\Bigr)\sqrt{\e}\bigl(\|e^{\f{y^2}8}e^{\de|D_x|}(u_0,b_0)\|_{\cB^{\f12,0}}+1\bigr)\\
&\lesssim \sqrt{\e}\bigl(\|e^{\f{y^2}8}e^{\de|D_x|}(u_0,b_0)\|_{\cB^{\f12,0}}+1\bigr).
\end{split}\eeqo
This leads to \eqref{S7eq28}. We thus complete the proof of Proposition \ref{prop3.3}.
\end{proof}

\begin{proof}[Proof of Proposition \ref{prop3.6}] The proof of Proposition \ref{prop3.6} is the same as that of Proposition \ref{prop3.3}.
We omit the details here.
\end{proof}

\ \par
\noindent
{\bf Acknowledgements.}
P. Zhang is partially supported by NSF of China under Grants    11688101 and 11371347,  innovation grant from National Center for Mathematics and Interdisciplinary Sciences.


\begin{thebibliography}{50}

\bibitem{Alex} R. Alexandre, Y. G. Wang, C. J. Xu and T. Yang, Well-posedness of the Prandtl equation in Sobolev spaces,
{\it J. Amer. Math. Soc.}, {\bf 28} (2015), 745-784.

\bibitem{BCD} H. Bahouri, J.~Y. Chemin and R. Danchin, {\it Fourier analysis and
nonlinear partial differential equations}, Grundlehren der
mathematischen Wissenschaften 343, Springer-Verlag Berlin
Heidelberg, 2011.


\bibitem{Bo} J.~M. Bony, Calcul symbolique et propagation des singularit\'es pour
les \'equations aux d\'eriv\'ees partielles non lin\'eaires, {\it
Ann. Sci. \'Ecole Norm. Sup.}, {\bf 14} (1981), 209--246.



\bibitem{Ch04} J.-Y. Chemin,  Le syst\'eme de Navier-Stokes incompressible
soixante dix ans apr\'es Jean Leray, {\it Actes des Journ\'ees
Math\'ematiques $\grave{a}$ la M\'emoire de Jean Leray}, 99-123,
S\'emin. Congr., 9, Soc. Math. France, Paris, 2004.

\bibitem{CGP}  J.-Y. Chemin, I. Gallagher and M. Paicu,  Global regularity for  some classes of large solutions to the Navier-Stokes equations,
{\it  Ann. of Math.}, {\bf 173} (2011), 983-1012.

\bibitem{CN95} J.~Y. Chemin and N. Lerner,  Flot de champs de vecteurs non lipschitziens et \'equations de Navier-Stokes,
 {\it J. Differential Equations}, {\bf 121} (1995), 314-328.

\bibitem{CZ1} J.~Y. Chemin and P. Zhang, On the global wellposedness to the 3-D incompressible anisotropic
 Navier-Stokes equations, {\it Comm. Math. Phys.}, {\bf 272} (2007),
 529--566.



\bibitem{CMSZ} J. -M. Coron, F. Marbach, F. Sueur and P. Zhang, Controllability of the Navier-Stokes equation in a rectangle with a little help of a distributed phantom force,  {\it Ann. PDE}, {\bf  5} (2019), no. 2, Art. 17, 49 pp.

    \bibitem{Cow57} T. G. Cowling, {\it Magnetohydrodynamics}. Interscience Tracts on Physics and Astromy, {\bf 4}. Interscience,
New York; Interscience, London, 1957.

\bibitem{David01} P. A. Davidson, {\it An introduction to magnetohydrodynamics.} Combridge Texts in Applied Mathematics.
Combridge University Press, Cambridge, 2001.

\bibitem{DG08} H. Dietert and D. Gerard-Varet, Well-posedness of the Prandtl equation without any structural assumption,
{\it Ann. PDE}, {\bf 5} (2019), no. 1, Art. 8, 51 pp.

\bibitem{EE} W. E and B. Enquist, Blow up of solutions of the unstaedy Prandtl's equation,
{\it Comm. Pure Appl. Math.}, {\bf 50} (1998), 1287-1293.

\bibitem{GD} D. G\'erard-Varet and E. Dormy, On the ill-posedness of the Prandtl equation, {\it J. Amer. Math. Soc.}, {\bf 23} (2010), 591-609.

\bibitem{GM} D. G\'{e}rard-Varet and N. Masmoudi, Well-posedness for the Prandtl system without analyticity or monotonicity,  {\it
Ann. Sci. \'Ecole Norm. Sup. (4)},
   {\bf 48} (2015),  1273-1325.



 \bibitem{Ger2} D. G\'{e}rard-Varet and T. Nguyen, Remarks on the ill-posedness of the Prandtl equation,
{\it Asymptot. Anal.}, {\bf 77} (2012), 71-88.


 \bibitem{GP17} D. G\'erard-Varet and M. Prestipino,  Formal derivation and stability analysis of boundary layer models in MHD,
  {\it Z. Angew. Math. Phys.}, {\bf 68} (2017),  Paper No. 76, 16 pp.

  \bibitem{Guo} Y. Guo and T. Nguyen, A note on Prandtl boundary layers, {\it Comm. Pure Appl. Math.}, {\bf 64} (2011), 1416-1438.

 \bibitem{IV16} M. Ignatova and V. Vicol, Almost global existence for the Prandtl boundary layer equations,
  {\it Arch. Ration. Mech. Anal.}, {\bf  220} (2016),  809-848.


\bibitem{LY17} C. J. Liu and T. Yang, Ill-posedness of the Prandtl equations in Sobolev spaces around a shear flow with general decay,
 {\it J. Math. Pures Appl. (9)}, {\bf 108} (2017),  150-162.

\bibitem{LXY19cpam} C. J. Liu, F. Xie and T. Yang,  MHD boundary layers theory in Sobolev spaces without monotonicity I: Well-posedness theory,
 {\it Comm. Pure Appl. Math.}, {\bf 72} (2019),  63-121.

\bibitem{LXY19} C. J. Liu, F. Xie and T. Yang,  Justification of Prandtl ansatz for MHD boundary layer,
 {\it SIAM J. Math. Anal.}, {\bf 51} (2019),  2748-2791

 \bibitem{Can} M. C. Lombardo, M. Cannone and M. Sammartino, Well-posedness of the boundary layer equations,
{\it SIAM J. Math. Anal.}, {\bf 35} (2003), 987-1004.

\bibitem{MW}  N. Masmoudi and T. K. Wong, Local-in-time existence and uniqueness of solutions to the Prandtl equations by energy methods,
{\it Comm. Pure Appl. Math.}, {\bf 68} (2015), 1683-1741.

\bibitem{Olei63} O. A. Oleinik, The Prandtl system of equations in boundary layer theory, {\it Soviet Math. Dokl.}, {\bf 4} (1963), 583-586.


\bibitem{Pa02} M. Paicu, \'Equation anisotrope de Navier-Stokes dans des espaces  critiques,  {\it Rev. Mat. Iberoamericana,} {\bf 21} (2005),   179--235.

\bibitem{PZ1} M. Paicu and P. Zhang, Global solutions to the 3-D incompressible
 anisotropic Navier-Stokes system  in the critical spaces, {\it Comm. Math. Phys.}, {\bf 307} (2011), 713-759.

 \bibitem{PZ5} M. Paicu and P. Zhang, Global existence and decay of solutions to Prandtl system with small analytic data, arXiv:1911.03690.




\bibitem{mz1} M. Paicu and Z. Zhang, Global regularity for the Navier-Stokes equations with some classes of large initial data, {\it  Anal. PDE}, {\bf 4} (2011), 95-113.

\bibitem{mz2} M. Paicu and Z. Zhang, Global well-posedness for the 3D Navier-Stokes equations with ill-prepared initial data, {\it J. Inst. Math. Jussieu}, {\bf 13} (2014), 395-411.



\bibitem{Pra} L. Prandtl, $\ddot{U}$ber Fl$\ddot{u}$ssigkeitsbewegung bei sehr kleiner Reibung,
{\it Verhandlung des III Intern. Math.-Kongresses, Heidelberg,} 1904
484-491.


\bibitem{Caf} M. Sammartino and R. E. Caflisch, Zero viscosity limit for analytic solutions, of the Navier-Stokes equation on a half-space.
I. Existence for Euler and Prandtl equations, {\it Comm. Math.
Phys.}, {\bf 192} (1998),  433-461.

\bibitem{XY19} F. Xie and T. Yang, Lifespan of solutions to MHD boundary layer equations with analytic perturbation of general shear flow,
 {\it Acta Math. Appl. Sin. Engl. Ser.}, {\bf 35} (2019),  209-229.

\bibitem{Xin} Z. Xin and L. Zhang, On the global existence of solutions to the Prandtl's system,
{\it Adv. Math.}, {\bf 181} (2004), 88-133.



\bibitem{ZZ} P. Zhang and Z. Zhang, Long time well-posedness
of Prandtle system with small data, {\it J. Funct. Anal.}, {\bf 270} (2016), 2591-2615.



 \end{thebibliography}
\end{document}